\documentclass[reqno]{amsart}

\usepackage[utf8]{inputenc}
\usepackage[T1]{fontenc}
\usepackage[british]{babel,isodate}
\cleanlookdateon
\usepackage{lmodern}
\usepackage{enumerate}
\usepackage{tikz}
\usepackage{tikz-cd}
\usetikzlibrary{arrows}
\usepackage{cancel}
\usetikzlibrary{babel}
\usepackage{lscape}
\usepackage{url}

\usepackage{hyperref}
\hypersetup{linktocpage,colorlinks=true,citecolor=purple,linkcolor=blue,urlcolor=blue,filecolor=black}

\usepackage{float}
\usepackage{graphicx, import}
\usepackage{mathtools}
\usepackage{dsfont}
\usepackage{pinlabel}
\usepackage{yhmath}
\usepackage{enumitem}

\usepackage{xargs}                      % Use more than one optional parameter in a new commands
\usepackage{xcolor} % Coloured text etc.
\definecolor{OliveGreen}{rgb}{0,0.6,0}
 
\usepackage[colorinlistoftodos, prependcaption,textsize=tiny]{todonotes}
\newcommandx{\unsure}[2][1=]{\todo[linecolor=red,backgroundcolor=red!25,bordercolor=red,#1]{#2}}
\newcommandx{\change}[2][1=]{\todo[linecolor=blue,backgroundcolor=blue!25,bordercolor=blue,#1]{#2}}
\newcommandx{\info}[2][1=]{\todo[linecolor=OliveGreen,backgroundcolor=OliveGreen!25,bordercolor=OliveGreen,#1]{#2}}
\newcommand{\centre}[1]{\begin{array}{c} #1 \end{array}}

\usepackage{stackengine}

\usepackage{amsmath,amsthm,amssymb, amsfonts, amscd}
\usepackage{rotating,subcaption}
\usepackage{url}
\usepackage{graphicx,bm}
\usepackage{mathtools}
\usepackage[mathscr]{euscript}
\allowdisplaybreaks
\usepackage{lipsum}
\usepackage{bm}
\usepackage[capitalize]{cleveref} 

\usepackage{xpatch}
\makeatletter
\AtBeginDocument{\xpatchcmd{\@thm}{\thm@headpunct{.}}{\thm@headpunct{}}{}{}}
\makeatother

\newtheorem{theorem}{Theorem}[section]
\newtheorem{proposition}[theorem]{Proposition}
\newtheorem{lemma}[theorem]{Lemma}
\newtheorem{corollary}[theorem]{Corollary}

\theoremstyle{definition}

\newtheorem{example}[theorem]{Example}

\newtheorem{construction}[theorem]{Construction}

\theoremstyle{remark}

\newtheorem{remark}[theorem]{Remark}

\numberwithin{equation}{section}

\newcommand{\hooklongrightarrow}{\lhook\joinrel\longrightarrow}

\newcommand{\N}{\mathbb{N}}
\newcommand{\Z}{\mathbb{Z}}

\newcommand{\R}{\mathbb{R}}

\newcommand{\C}{\mathcal{C}}

\newcommand{\B}{\mathcal{B}}

\newcommand{\W}{\mathcal{W}}
\newcommand{\Cstr}{\C^{\mathrm{str}}}

\newcommand{\T}{\mathcal{T}}
\newcommand{\Tup}{\mathcal{T}^{\mathrm{up}}}
\newcommand{\id}{\mathrm{Id}}

\newcommand{\eend}[2]{\mathrm{End}_{#1}(#2)}
\let\hom\relax
\newcommand{\hom}[3]{\mathrm{Hom}_{#1}(#2,#3)}
\renewcommand{\to}{\longrightarrow}

\renewcommand{\SS}{\mathfrak{S}}
\DeclareMathOperator{\rot}{rot}
\DeclareMathOperator{\sign}{sign}

\newcommand{\toiso}{\overset{\cong}{\to}}

\usepackage{accents}
\newcommand{\uhat}{\underaccent{\check}}

\makeatletter
\newcommand{\cupr@tip}{\text{\raisebox{-0.1ex}{$\m@th\hat{}$}}}
\newcommand{\cupr}{\mathbin{\cup\cupr@}}

\newcommand{\cupr@}{%
  \mathchoice
  {\mkern-1.35mu\cupr@tip}
  {\mkern-1.35mu\cupr@tip}
  {\mkern-1.55mu\cupr@tip}
  {\mkern-1.875mu\cupr@tip}
}
\makeatother

\makeatletter
\newcommand{\capr@tip}{\text{\raisebox{0.47ex}{$\m@th\uhat{}$}}}
\newcommand{\capr}{\mathbin{\capr@\cap}}

\newcommand{\capr@}{%
  \mathchoice
  {\mkern11.6mu\capr@tip\mkern-11.6mu}
  {\mkern11.4mu\capr@tip\mkern-11.4mu}
  {\mkern11.1mu\capr@tip\mkern-11.1mu}
  {\mkern10.2mu\capr@tip\mkern-10.2mu}
}
\makeatother

\makeatletter
\newcommand{\capl@tip}{\text{\raisebox{0.47ex}{$\m@th\uhat{}$}}}
\newcommand{\capl}{\mathbin{\capl@\cap}}

\newcommand{\capl@}{%
  \mathchoice
  {\mkern2.1mu\capl@tip\mkern-2.1mu}
  {\mkern2.1mu\capl@tip\mkern-2.1mu}
  {\mkern2.3mu\capl@tip\mkern-2.3mu}
  {\mkern2.1mu\capl@tip\mkern-2.1mu}
}
\makeatother

\makeatletter
\newcommand{\cupl@tip}{\text{\raisebox{-0.1ex}{$\m@th\hat{}$}}}
\newcommand{\cupl}{\mathbin{\cupl@\cup}}

\newcommand{\cupl@}{%
  \mathchoice
  {\mkern1.35mu\cupl@tip\mkern-1.35mu}
  {\mkern1.35mu\cupl@tip\mkern-1.35mu}
  {\mkern1.55mu\cupl@tip\mkern-1.55mu}
  {\mkern1.875mu\cupl@tip\mkern-1.875mu}
}
\makeatother

\newcommand{\KPP}[1]{%
  \begin{tikzpicture}[scale=0.8, baseline=-\dimexpr\fontdimen22\textfont2\relax]
  #1
  \end{tikzpicture}%
}

\newcommand{\PC}{%
  \KPP{
    \draw[color=black] (0.2,0.2) -- (-0.2,-0.2);
    \draw[color=black] (0.2,-0.2) -- (0.05,-0.05);
    \draw[color=black] (-0.05,0.05) -- (-0.2, 0.2);
     \draw[color=black] (-0.05,0.05) -- (-0.2, 0.2);
      \draw[color=black] (0.1,0.2) -- (0.21, 0.2);
      \draw[color=black] (0.2,0.1) -- (0.2, 0.21);
   \draw[color=black] (-0.21,0.2) -- (-0.1, 0.2);
    \draw[color=black] (-0.2,0.21) -- (-0.20, 0.1);
  }%
}

\newcommand{\NC}{%
  \KPP{
    \draw[color=black] (-0.2,0.2) -- (0.2,-0.2);
    \draw[color=black] (-0.2,-0.2) -- (-0.05,-0.05);
    \draw[color=black] (0.05,0.05) -- (0.2, 0.2);
          \draw[color=black] (0.1,0.2) -- (0.21, 0.2);
      \draw[color=black] (0.2,0.1) -- (0.2, 0.21);
   \draw[color=black] (-0.21,0.2) -- (-0.1, 0.2);
    \draw[color=black] (-0.2,0.21) -- (-0.20, 0.1);
  }%
}

\DeclareFontFamily{U}{mathx}{}
\DeclareFontShape{U}{mathx}{m}{n}{ <-> mathx10 }{}
\DeclareSymbolFont{mathx}{U}{mathx}{m}{n}
\DeclareFontSubstitution{U}{mathx}{m}{n}
\DeclareMathAccent{\widecheck}{0}{mathx}{"71}

\usetikzlibrary{decorations.markings}
\tikzset{double line with arrow/.style args={#1,#2}{decorate,decoration={markings,%
mark=at position 0 with {\coordinate (ta-base-1) at (0,1pt);
\coordinate (ta-base-2) at (0,-1pt);},
mark=at position 1 with {\draw[#1] (ta-base-1) -- (0,1pt);
\draw[#2] (ta-base-2) -- (0,-1pt);
}}}}

\begin{document}

%%
%% The title of the paper goes here.  Edit to your title.
%%
%The 2-loop polynomial of genus 1 knots coming from Hopf algebras

\title{A refined functorial universal tangle invariant}

%\title{Canonicity of the functorial universal tangle invariant}
\date{\today}
%%
%% Now edit the following to give your name and address:
%% 

\author{Jorge Becerra}
\address{Université Bourgogne Europe, CNRS, IMB UMR 5584, 21000 Dijon, France}
\email{\href{mailto:Jorge.Becerra-Garrido@ube.fr}{Jorge.Becerra-Garrido@ube.fr}}
\urladdr{ \href{https://sites.google.com/view/becerra/}{https://sites.google.com/view/becerra/}} % Delete if not wanted.
%\thanks{A mi madre}
%%
%% If there is another author uncomment and edit the following.
%%

%\author{Second Author}
%\address{Department of Mathematics, University of South Carolina,
%Columbia, SC 29208}
%\email{second@math.sc.edu}
%\urladdr{www.math.sc.edu/$\sim$second}

%%
%% If there are three of more authors they are added in the obvious
%% way. 
%%

%%%
%%% The following is for the abstract.  The abstract is optional and
%%% if not used just delete, or comment out, the following.
%%%

\begin{abstract}
The universal invariant with respect to a given ribbon Hopf algebra $A$ is a tangle invariant that dominates all the Reshetikhin-Turaev invariants built from the representation theory of the algebra. In this paper, we construct a canonical  strict monoidal  functor $Z_A: \Tup \to \mathcal{E}(A)$, where $\Tup$ is the category of upwards tangles in a cube and $\mathcal{E}(A)$ is the so-called ``category of elements'' of $A$, that  encodes the universal invariant of upwards tangles. Moreover, this functor refines the Kerler-Kauffman-Radford functorial invariant $\mathsf{Dec}_A$ in the sense that it factors through $Z_A$ and, unlike $\mathsf{Dec}_A$, our functor $Z_A$  preserves the braiding, twist and the open trace of $\Tup$, the latter being a mild modification of Joyal-Street-Verity's notion of trace in a balanced category. We construct this functor $Z_A$ using the more flexible XC-algebras, a class which contains both ribbon Hopf algebras and endomorphism algebras of representation of these.
%The universal invariant with respect to a given ribbon Hopf algebra is a tangle invariant that dominates all the Reshetikhin-Turaev invariants built from the representation theory of the algebra. We construct a canonical strict monoidal  functor that encodes the universal invariant of upwards tangles and refines the Kerler-Kauffman-Radford functorial invariant. Moreover, this functor preserves the braiding, twist and the open trace, the latter being a mild modification of Joyal-Street-Verity's notion of trace in a balanced category. We construct this functor using the more flexible XC-algebras, a class which contains both ribbon Hopf algebras and endomorphism algebras of representation of these.
\end{abstract}

\keywords{universal invariant, ribbon Hopf algebra, XC-algebra, open-traced monoidal category}
\subjclass{16T05, 18M10, 18M15, 57K10}

%%
%%  LaTeX will not make the title for the paper unless told to do so.
%%  This is done by uncommenting the following.
%%

\maketitle

%%
%% LaTeX can automatically make a table of contents.  This is done by
%% uncommenting the following:
%%
\setcounter{tocdepth}{1}
\tableofcontents

%%
%%  To enter text is easy.  Just type it.  A blank line starts a new
%%  paragraph. 
%%

\section{Introduction}

Modern knot theory is tied to the area of \textit{quantum topology} that arose at the end of the 1980s after the work of the Fields medallists  Vaughan Jones, Vladimir Drinfeld and Edward Witten (the three of them  awarded in 1990). The point of view of quantum topology is different to that of most branches of mathematics in the following sense: generally,  one is interested in constructing invariants of a certain topological/geometrical object (e.g. an elliptic curve, a topological space, a knot) that are intrinsic, that is, invariants that are built using only data from the object itself (e.g. its rank, its homotopy groups, its Alexander polynomial, respectively).  On the contrary, quantum topology studies algebraic invariants of objects, typically from low-dimensional topology, using the  additional data of some algebraic gadget that somehow encodes or mimics properties that the topological object satisfies. Invariants built this way are usually called \textit{quantum invariants}. The word ``quantum'' is a remnant of  the fact that this theory was inspired by ideas from theoretical physics. For instance, one of the cornerstones of quantum topology, namely \textit{quantum groups} (roughly, one-parameter deformations of the universal enveloping algebra of semisimple Lie algebras),  arose as a mathematical formalisation of ideas from the Leningrad school of mathematical physics.% directed by Ludvig Faddeev.

This paper revolves around one of these invariants, namely the so-called \textit{universal invariant} of knots, which was defined by Lawrence \cite{lawrence}  and  was further developed by Reshetikhin \cite{reshetikhin}, Lee \cite{lee1,lee2}, Hennings \cite{hennings}, Ohtsuki \cite{ohtsuki1,ohtsukibook} and Habiro \cite{habiro}. The auxiliary algebraic gadget used to construct this invariant is a \textit{ribbon Hopf algebra}, that is, a Hopf algebra $A$ together with the choice of two invertible elements $R\in A \otimes A$ and $v \in A$ fulfilling certain axioms that somehow  mimic  algebraically properties that the positive crossing and the full twist satisfy geometrically. In its simplest form, the universal invariant of a (long) knot $K$, viewed as a $(1,1)$-tangle, is an element $\mathfrak{Z}_A(K) \in A$. The adjective ``universal'' is due to the fact that this invariant dominates all Reshetikhin-Turaev invariants of the knot $K$, these being built from the representation theory of $A$. More precisely, given a finite-dimensional $A$-module $\rho: A \to \eend{\Bbbk}{W}$, we have that $$RT_W (K)= \rho(\mathfrak{Z}_A(K)).$$
Here $RT_W: \T \to \mathsf{fMod}_A^{\mathrm{str}}$ is the functor canonically defined by the fact that the category of framed, oriented tangles in a cube $\T$ (with objects sequences of signs $+$ and $-$) is the free strict ribbon category generated by a single object, and the knot $K$ is viewed a morphism $K: + \to +$.

The aim of this paper is to construct a  strict monoidal functor $$Z_A: \Tup \to \mathcal{E}(A)$$ that encodes the universal invariant of knots. Kerler  \cite{kerler1}, Kauffman \cite{K1} and Kauffman-Radford \cite{KR3,KR2,KR1} have already proposed a version $\mathsf{Dec}_A : \T \to s\T (A)$ of such a functor, where the target $s\T (A)$ is the so-called ``category of singular $A$-tangles''. However, their functor $\mathsf{Dec}_A$ has a few downsides. Firstly, this functor is ill-defined for tangles with closed components. Secondly, their target category $s\T (A)$ does not even admit a braiding with the property that  $\mathsf{Dec}_A$  is a braided functor. Thirdly,  this functor does not arise canonically (that is, from a universal property), despite it dominates the Reshetikhin-Turaev functors $RT_W$, which do arise canonically. On the contrary, the functor  $Z_A: \Tup \to \mathcal{E}(A)$ that we construct in this paper preserves all the structure that $\Tup $ has, in particular $Z_A$ is braided. Besides, our functor $Z_A$   will arise canonically from a universal property, and will be shown to refine Kerler-Kauffman-Radford functor in the sense that (an appropriate version of) $\mathsf{Dec}_A$  factors through $Z_A$. Furthermore, when the algebra $A$ is endowed with a trace, our construction of $Z_A$ can be modified to incorporate closed components. A different approach to study the universal invariant functorially was given by Habiro in \cite{habiro} making use of bottom tangles.

\subsection{The category \texorpdfstring{$\Tup$}{Tup} of upwards tangles}

We will define the functorial universal invariant in a convenient, yet sufficiently general subcategory  of tangles that contains all (long) knots. Concretely, we write $\Tup \subset \T$ for the monoidal subcategory on the objects sequences of the sign $+$ (hence non-negative integers) and arrows tangles without closed components. Therefore $\Tup$ sits in between the monoidal subcategory $\mathcal{B} \subset \T$ of framed braids and the full monoidal subcategory $\T^+ \subset \T$ on sequences of the sign $+$,
$$\mathcal{B} \hooklongrightarrow \Tup \hooklongrightarrow \T^+ .$$
Upwards tangles have the advantage that they admit diagrams, that we call \textit{rotational}, which are made of the following five local pictures,
\begin{equation}\label{eq:crossings_and_spinners_INTRO}
\centre{
\labellist \small \hair 2pt
\endlabellist
\centering
\includegraphics[width=0.6\textwidth]{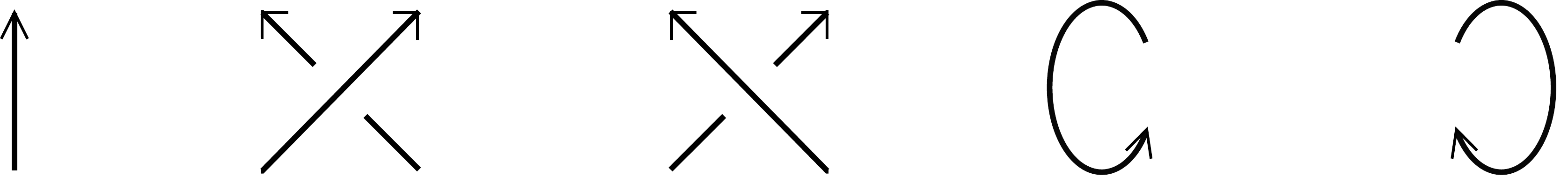}}
\end{equation}
%\vspace{0.1pt}
\noindent and moreover a family of Reidemeister moves in terms of these five pieces can be made explicit.

It turns out that the category $\Tup$, just like $\B$, $\T^+$ or $\T$, satisfies a universal property:

\begin{theorem}[\cref{thm:UP_opentraced}]
The category $\Tup$ of upwards tangles is the free strict open-traced monoidal category generated by a single object.
\end{theorem}

The notion of open-traced monoidal category is a minor modification of that of  traced monoidal category due to Joyal-Street-Verity \cite{joyal_street_traced}. Roughly, an open-traced monoidal category is a balanced monoidal category $\C$ (i.e. braided and with a twist) together with a family of ``partial trace'' operations
$$\mathrm{tr}_{X,Y}^U : \mathrm{Hom}_\C^{ad} (X \otimes U, Y \otimes U) \to \hom{\C}{X}{Y}$$
defined for  ``admissible'' morphisms. Informally, this means that the resulting arrow $\mathrm{tr}_{X,Y}^U (f)$ has no closed components in the graphical calculus. Formally, admissibility of morphisms is defined by means of a balancing-preserving strict monoidal functor $P: \C \to \SS$, which is part of the data. Here $\SS$ denotes the free strict symmetric monoidal category generated by a single object.

For instance, the category $\T^+$ (which is traced in the sense of Joyal-Street-Verity) is also open-traced, and more generally if $\C$ is a traced monoidal category and $X$ is an object of $\C$ such that all tensor powers $X^{\otimes n}$ are different objects, then one can construct a open-traced monoidal category $\C_X$ which is a monoidal subcategory of $\C$, see \cref{constr:open_tr_from_tr}.

\subsection{XC-algebras}

Classically, the universal tangle invariant is defined with the extra data of a ribbon Hopf algebra $(A, R, v)$. However, the comultiplication, counit and antipode from the Hopf algebra are not really used to define $\mathfrak{Z}_A(K)$ (although they satisfy some naturality properties with respect to tangle operations, see e.g. \cite{ohtsukibook,habiro,becerra_gaussians}). Our construction of the functorial universal invariant uses the minimal algebraic data needed to produce an isotopy invariant, that we call an \textit{XC-algebra}.

More precisely, an XC-algebra structure over a  $\Bbbk$-algebra $A$ is pair $(R,\kappa)$  where $R\in A \otimes A$ and $\kappa \in A$ are invertible elements that are the algebraic counterparts of the positive crossing and the clock-wise full rotation from  \eqref{eq:crossings_and_spinners_INTRO}. As expected, any ribbon Hopf algebra is an example of XC-algebra, see \cref{prop:ribbon->XC}. Moreover, any XC-algebra structure on an algebra $A$ induces an XC-algebra structure on the endomorphism algebra $\eend{\Bbbk}{W}$ for any finite-dimensional $A$-module $W$. The endomorphism XC-algebra has the additional property of being \textit{traced} (as an algebra).

The main construction of this paper is a strict open-traced monoidal category associated to any XC-algebra:

\begin{theorem}[\cref{thm:E(A)_opentraced} and \cref{thm:E(A)_traced}]
For any XC-algebra (resp. traced XC-algebra) $A$, we can construct an open-traced (resp. traced) monoidal category $\mathcal{E}(A)$ with objects non-negative integers and only endomorphisms with $$ \eend{\mathcal{E}(A)}{n} \subset A^{\otimes n} \times \SS_n  $$ (this is a set-theoretical inclusion).
\end{theorem}

We call $\mathcal{E}(A)$ the \textit{category of elements} of $A$. This category is constructed in such a way that, if $\sigma_T$ denotes the permutation induced by an upwards tangle $T$, then we have a set-theoretical equality
\begin{equation}\label{eq:set-theor_1}
 \eend{\mathcal{E}(A)}{n} = \{ (\mathfrak{Z}_A(T), \sigma_T) : T \in \eend{\Tup}{+^n}   \},
\end{equation}
which in the traced case becomes
\begin{equation}\label{eq:set-theor_2}
\eend{\mathcal{E}(A)}{n} = \{ (\mathfrak{Z}_A(T), \sigma_T) : T \in \eend{\T^+}{+^n}   \}.
\end{equation}

The unique strict monoidal functor $$Z_A: \Tup \to  \mathcal{E}(A)$$ that sends $+$ to $1$ and that preserves the braiding, twist and open trace will be called the \textit{functorial universal invariant} associated to $A$. In the traced case, this becomes a traced functor $Z_A: \T^+ \to  \mathcal{E}(A)$. In particular, the previous equalities \eqref{eq:set-theor_1} and \eqref{eq:set-theor_2} express that $Z_A$ is full both in the non-traced and traced case.

We emphasise that one advantage of this functorial version of the universal  is that it arises canonically  from the universal property of $\Tup$.

\subsection{Comparison with Kerler-Kauffman-Radford functorial invariant}

The functorial version of universal invariant due to Kerler-Kauffman-Radford can be expressed as a well-defined ``decoration functor'' $$\mathsf{Dec}_A : \Tup \to s\Tup (A)$$ from $\Tup$ to the so-called category of ``singular upwards $A$-tangles''  \cite{KR3, kerler1}.  This is a strict monoidal functor, but it does not preserve any other structure of $\Tup$, in particular $s\Tup (A)$ does not admit a braiding such that $\mathsf{Dec}_A$ is a braided functor.

\begin{theorem}[\cref{thm:comparison_KKR}]
The functor $\mathsf{Dec}_A $ factors through $Z_A$, that is, there is a commutative  diagram of strict monoidal functors as below :
$$
\begin{tikzcd}
\Tup \rar{\mathsf{Dec}_A} \dar[swap]{Z_A} & s\Tup (A) \\
\mathcal{E}(A) \arrow[hook]{ur} &
\end{tikzcd}
$$
\end{theorem}

According to the previous theorem, our functorial universal invariant $Z_A: \Tup \to  \mathcal{E}(A)$ can be seen as a refinement of Kerler-Kauffman-Radford functor, in the sense that  $\mathsf{Dec}_A$ factors through a functor which does preserve the braiding, twist and open trace of $\Tup$.

\subsection{Comparison with the Reshetikhin-Turaev invariant}

We would like to explain how the functorial universal invariant $Z_A$  relates to several other functors that we can also construct when $A$ is a ribbon Hopf algebra and $W$ is a given finite-dimensional $A$-module. We can assemble all relations in the following commutative prism (\cref{cor:prism}), that we explain below:

$$
\begin{tikzcd}
\Tup \arrow[rrrr,"RT_W"] \arrow[rrd,"Z_A"] \arrow[dd,hook] &  &               &  & (\mathsf{fMod}_A^{\mathrm{str}})_W  \arrow[dd,hook] \\
                                       &  & \mathcal{E}(A) \arrow{rru}{\rho_W} &  &                           \\
\T^+ \arrow[rrrr,pos=0.6,"RT_W"] \arrow[rrd,"Z_{\eend{\Bbbk}{W}}"]            &  &               &  & \mathsf{fMod}_A^{\mathrm{str}}            \\
                                       &  & \mathcal{E}(\eend{\Bbbk}{W}) \arrow[hook]{rru}{\iota_W} \arrow[from=uu,pos=0.3,crossing over]        &  &                          
\end{tikzcd}
$$

The back face of the prism expresses that the unique strict monoidal functor $RT_W: \Tup \to (\mathsf{fMod}_A^{\mathrm{str}})_W $ obtained by applying the universal property of $\Tup$ to the open-traced monoidal category $(\mathsf{fMod}_A^{\mathrm{str}})_W $ obtained from $\mathsf{fMod}_A^{\mathrm{str}}$ is simply the restriction of the usual Reshetikhin-Turaev invariant to the category of upwards tangles (\cref{lemma:restriction}).

The commuting top face of the prism is the universality of $Z_A$ with respect to the Reshetikhin-Turaev invariants: the functor $RT_W$ factors through $Z_A$ (\cref{thm:ZA_universal}). The functor $\rho_W:\mathcal{E}(A) \to (\mathsf{fMod}_A^{\mathrm{str}})_W $ is induced by the $A$-module structure morphism $\rho: A \to \eend{\Bbbk}{W}$, hence the name. Furthermore, all functors in this diagram preserve the open-traced structure.

The commutativity of the bottom face of the prism expresses that, if we consider over the endomorphism algebra $\eend{\Bbbk}{W}$ the XC-algebra structure inherited from $A$, then the functorial universal invariant $Z_{\eend{\Bbbk}{W}}$ is essentially the same as the Reshetikhin-Turaev invariant (\cref{thm:RTV_ZEndV}).  The functor $\iota_W:  \mathcal{E}(\eend{\Bbbk}{W}) \hooklongrightarrow  \mathsf{fMod}_A^{\mathrm{str}} $ is a monoidal embedding, which is induced by the canonical isomorphism $\eend{\Bbbk}{W}\otimes \eend{\Bbbk}{W} \cong \eend{\Bbbk}{W\otimes W}$. Even more, all functors in the bottom face of the prism preserve the traced structure.

The structure morphism  $\rho: A \to \eend{\Bbbk}{W}$ clearly induces a strict monoidal functor $ \mathcal{E}(A) \to \mathcal{E}( \eend{\Bbbk}{W})$. The commutativity of the left-hand face of the prism affirms that composing $Z_A$ with this functor is essentially the functor $Z_{\eend{\Bbbk}{W}}$, and the commutativity of the right-hand face says that the composite of the functor $ \mathcal{E}(A) \to \mathcal{E}( \eend{\Bbbk}{W})$ with $\iota_W$ is essentially the functor $\rho_W$.

\subsection*{Organisation of the paper} 

In \cref{sec:upwards_tangles}, we recall basics on tangle categories and ribbon Hopf algebras, putting emphasis on the universal properties that these categories satisfy. In \cref{sec:upwards}, we study the category $\Tup$ of upwards tangles, making use of rotational diagrams. We also define rigorously open traced monoidal categories and show that $\Tup$ is the free such category generated by a single object. Next in \cref{sec:universal_inv} we define XC-algebras, give several examples and construct an open-traced monoidal category $\mathcal{E}(A)$ for every XC-algebra $A$, as well as a traced monoidal category $\mathcal{E}(A)$ for every traced XC-algebra $A$. Lastly we compare the functorial universal tangle invariant $Z_A$ with Kerler-Kauffman-Radford ``decoration functor''. Finally in \cref{sec:universality} we discuss a functorial statement about the universality of $Z_A$ with respect to Reshetikhin-Turaev invariants and show that for a representation $W$ of a ribbon Hopf algebra, the functorial universal invariant $Z_{\eend{\Bbbk}{W}}$ is essentially the same as the Reshetikhin-Turaev invariant $RT_W$.

\subsection*{Acknowledgments} The author would like to thank  Roland van der Veen for many valuable discussions about the content of this paper, as well as to Sofia Lambropoulou, Luis Paris  and Lukas Woike for helpful conversations.   Most of the content of this paper is taken from the author's PhD thesis \textit{Universal quantum knot invariants}, written at the University of Groningen.  The author was supported by the ARN project CPJ number ANR-22-CPJ1-0001-0  at the Institut de Mathématiques de Bourgogne (IMB). The IMB receives support from the EIPHI Graduate School (contract ANR-17-EURE-0002).

%%%%%%%%%%%%%%%%%%%%%%%%%%%%%%%%%%%%

\section{Categories of tangles}\label{sec:upwards_tangles}

In this section we will recollect basic definitions and facts about several categories of braids and tangles that will heavily perform in the rest of the paper.

\subsection{Tangles in a cube}

Let $n,m \geq 0$ be  non-negative integers. A  \textit{(framed, oriented)  tangle} \index{tangle} is an isotopy class of an embedding $$T:  \left( \coprod_n D^1 \times D^1 \right) \amalg  \left( \coprod_m D^1 \times S^1 \right)  \hooklongrightarrow (D^1)^{\times 3}$$
with the property that it restricts to a orientation-preserving homeomorphism $$\coprod_n   D^1 \times \{-1, 1 \}       \toiso \Big(   \bigcup_{i=1}^{n_1} F_i \Big)  \cup  \Big(   \bigcup_{i=1}^{n_2} H_i \Big)  $$
where $n_1, n_2  \geq 0$, $ n_1+n_2=2n$,  $F_i := \left[ \frac{2i-1}{2n_1+1} , \frac{2i}{2n_1+1} \right] \times 0 \times -1 \subset  (D^1)^{\times 3}$ and $H_i := \left[ \frac{2i-1}{2n_2+1} , \frac{2i}{2n_2+1} \right] \times 0 \times 1 \subset  (D^1)^{\times 3}$, with the orientations on $D^1 \times \pm 1$ induced by the usual one in $D^1 \times D^1$ and the orientations on $F_i$ and $H_i$ are induced by the ones of the positive and negative direction, respectively.  The isotopy is understood to be relative to $\coprod_n (D^1 \times \pm 1  )$. Moreover, the cores of the strips $ 0 \times D^1 $ and annuli $0 \times S^1$ are also endowed with an orientation. If $m=0$, we say that the tangle is \textit{open}; and if  $n=0$ then we talk of a framed oriented link.

For a tangle $T$, we can talk about how twisted each component is. More precisely, the \textit{framing}\index{framing} of a closed component $T(D^1 \times S^1)$ is the linking number of the two-component link $T(-1 \times S^1) \cup T(1\times S^1)$. For an open component, its framing is the framing of the closed one obtained by closing up the component with a strip which lies in the plane $xz$.

Since it would be rather cumbersome to draw strips and annuli in pictures, it is customary to keep only the cores of these, as follows: first, isotope the strips and annuli  of the tangle so that each of the embedded segments $D^1 \times t$ is parallel to the plane $xz$ (i.e., so that the same side always faces the reader) and the cores of these are in general position with respect to the projection onto the plane $xz$. The projection of these cores on the square $D^1 \times 0 \times D^1$ is called a \textit{tangle diagram}. Below is an example:

\begin{equation*} 
\centering
\includegraphics[width=0.65\textwidth]{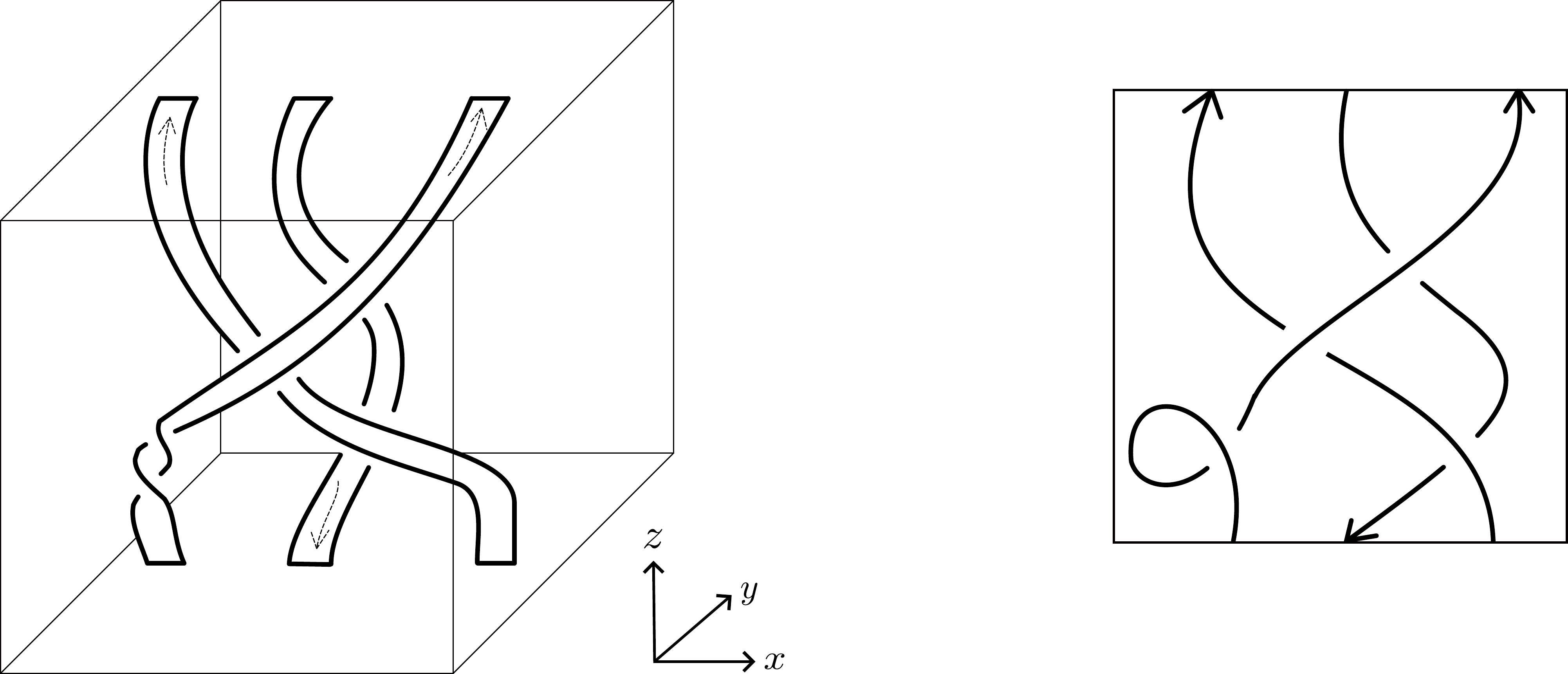}
\end{equation*}
%\vspace{0.1cm}\noindent

Conversely, given a tangle diagram of a framed tangle, we will always assume that its framing is given by the \textit{blackboard framing}, that is, that the strands of the diagrams are the cores of a framed tangle whose strips are parallel to the plane $xz$. The classical Reidemeister theorem says, in the framed case, that passing to tangle diagrams induces a bijection
$$\begin{tikzcd}[column sep=3em]
\frac{ \left\{ \parbox[c]{6em}{\centering     {\small  \textnormal{tangles in $(D^1)^{\times 3}$ }}} \right\} }{  \parbox[c]{6em}{\centering \vspace*{2pt} {\small  \textnormal{isotopy rel. endpoints} }}} \ar[-,double line with arrow={-,-}]{r} &    \frac{ \left\{ \parbox[c]{8em}{\centering     {\small  \textnormal{tangle diagrams in $D^2$ }}} \right\} }{ \parbox[c]{8em}{\centering  \vspace*{2pt} {\small  \textnormal{framed Reidemeister moves  and planar isotopy} }}}
\end{tikzcd}$$

\subsection{The category \texorpdfstring{$\T$}{T} of tangles} 

It is well-known that the set of tangles can be organised into a strict monoidal category. Let $\mathrm{Mon}(+,-)$ be the free monoid on the set $\{+,- \}$. Given a tangle $T$, assign to every $F_i$ and $H_j$  the symbol $+$ or $-$  depending on whether $T$ points upwards or downwards, respectively. This assignment defines two elements $s(T), t(T) \in \mathrm{Mon}(+,-)$ of lengths $n_1 $ and $n_2$  called the \textit{source} \index{source of a tangle} and the \textit{target} \index{target of a tangle} of $T$.

The category $\mathcal{T}$  of tangles is defined to have objects $\mathrm{Mon}(+,-)$ and morphisms $\hom{\mathcal{T}}{s}{t}$ the set of isotopy rel. endpoint classes of tangles $T$ such that $s= s(T)$ and $t=t(T)$. The composite $T_2 \circ T_1$ of tangles $T_1, T_2$ is the tangle resulting from stacking $T_2$ on top of $T_1$, and the identity of a word $w \in \mathrm{Mon}(+,-)$  is the tangle $\uparrow_w$ given by a number of parallel, vertical strands with orientations determined by $w$. The monoidal product is given by concatenation of words at the level of the object, and at the level of morphisms $T_1 \otimes T_2$ is the tangle resulting from placing $T_2$ to the right of $T_1$ (and normalising the length of the cube),
\begin{equation*}
T_2 \circ T_1 = 
\centre{
\labellist \small \hair 2pt
\pinlabel{$ T_1$}  at 58 65
\pinlabel{$ T_2$}  at 58 188
\endlabellist
\centering
\includegraphics[width=0.05\textwidth]{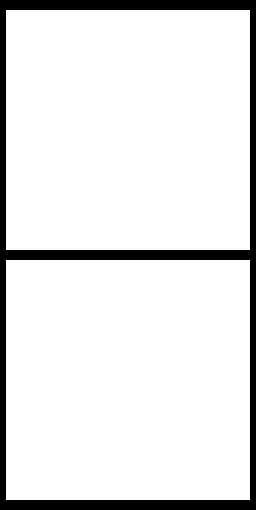}} \qquad , \qquad 
T_1  \otimes T_2 = \centre{
\labellist \small \hair 2pt
\pinlabel{$ T_1$}  at 58 64
\pinlabel{$ T_2$}  at 177 64
\endlabellist
\centering
\includegraphics[width=0.1\textwidth]{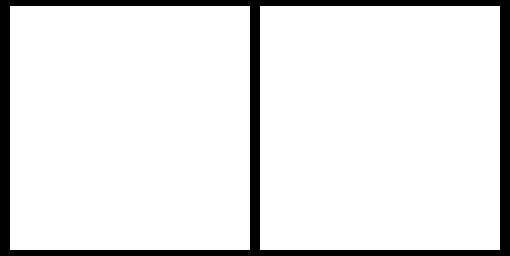}}
\end{equation*}
The unit object is the empty word, that is, the unit of $\mathrm{Mon}(+,-)$. Note that $\hom{\T}{\emptyset}{\emptyset}$ is precisely the set of framed, oriented links.

The following is a folklore result:

\begin{theorem}[\cite{turaev89,kassel, KRT,ohtsukibook}]\label{thm:T_presentation}
The category $\mathcal{T}$ is monoidally generated by the objects $+,-$ and the morphisms $  \PC$ , $\NC$ , $\cupr$ , $\cupl$ ,  $\capr$ , $\capl$ shown below:
\begin{equation*} 
\centre{
\centering
\includegraphics[width=0.6\textwidth]{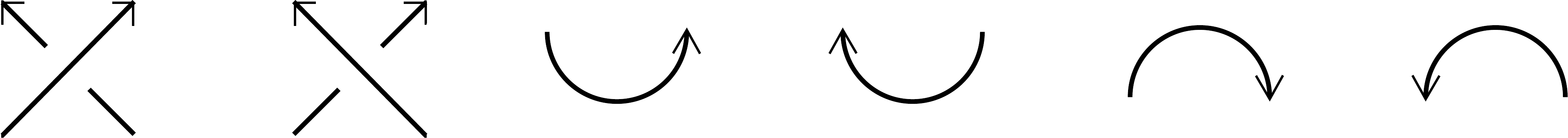}}
\end{equation*}
\end{theorem}
\vspace{0.3cm}\noindent
%in \cref{fig generating tangles}.

The category $\mathcal{T}$ enjoys a remarkable universal property, that we now discuss. A (strict) \textit{ribbon} \index{ribbon category} or \textit{tortile category} is a  strict monoidal category $(\mathcal{C}, \otimes, \bm{1})$ with the additional data of
\begin{enumerate}
\item A \textit{braiding}\index{braiding}, that is, a family of natural isomorphisms $\tau_{X,Y}: X \otimes Y \toiso Y \otimes X$ satisfying $$ \tau_{X, Y \otimes Z}= (\id_Y \otimes \tau_{X,Z}) (\tau_{X,Y} \otimes \id_Z) \quad , \quad  \tau_{X \otimes Y , Z}= (\tau_{X,Y} \otimes \id_Y)(\id_X \otimes \tau_{Y,Z}) .$$
\item A \textit{(left) rigid structure} \index{rigid structure} or duality $X\rightsquigarrow X^*$ with a morphism or \textit{pairing}\index{pairing} $$ \mathrm{ev}_X : X^* \otimes X \to \bm{1} ,$$ called the \textit{left-evaluation}\index{left-evaluation}, which is \textit{non-degenerate}\index{non-degenerate} in the sense that there exists another morphism
$\mathrm{coev}_X : \bm{1} \to X \otimes X^* $, called the \textit{left-coevaluation}\index{left-coevaluation},
satisfying that
$$ (\id_X \otimes \mathrm{ev}_X ) (\mathrm{coev}_X \otimes \id_X) = \id_X  \ , \  (\mathrm{ev}_{X^*} \otimes \id_{X^*} ) (\id_{X^*} \otimes \mathrm{coev}_{X^*}) = \id_{X^*}. $$
\item A \textit{twist}\index{twist}, that is, a family of natural isomorphisms $\theta_X : X\toiso X$ such that $$\theta_{\bm{1}}= \id_{\bm{1}} \quad , \quad   \theta_{X \otimes Y} =    (\theta_Y \otimes \theta_X) \tau_{Y, X}\tau_{X,Y} ,$$
\end{enumerate}
such that the twist and rigid structure are compatible in the sense that $\theta_{X^*} = \theta_X^*$.  The braiding and the twist imply that a ribbon category is in fact \textit{pivotal}\index{pivotal structure}, that is, it is also endowed with a right duality and the left and right dual functors coincide. More precisely, the right rigid structure is determined by considering the same right dual objects as the left duals and as right evaluation of an object $X \in \C$
\begin{equation}\label{eq:right_evaluation}
\widetilde{\mathrm{ev}}_X := \mathrm{ev}_X  \tau_{X,X^*} (\theta_X \otimes \id_{X^*})     : X \otimes X^* \to \bm{1}.
\end{equation}

It is well-known that the tangle category $\mathcal{T}$ is a ribbon category where the braiding is determined by $\tau_{+,+}= \PC$, the left rigid structure is determined by $+^* =-$ and $\mathrm{ev}_+= \capl$ and $\mathrm{coev}_+ = \cupl$, and the twist is determined by  $\theta_+ =  \begin{array}{c}
\centering
\includegraphics[scale=0.035]{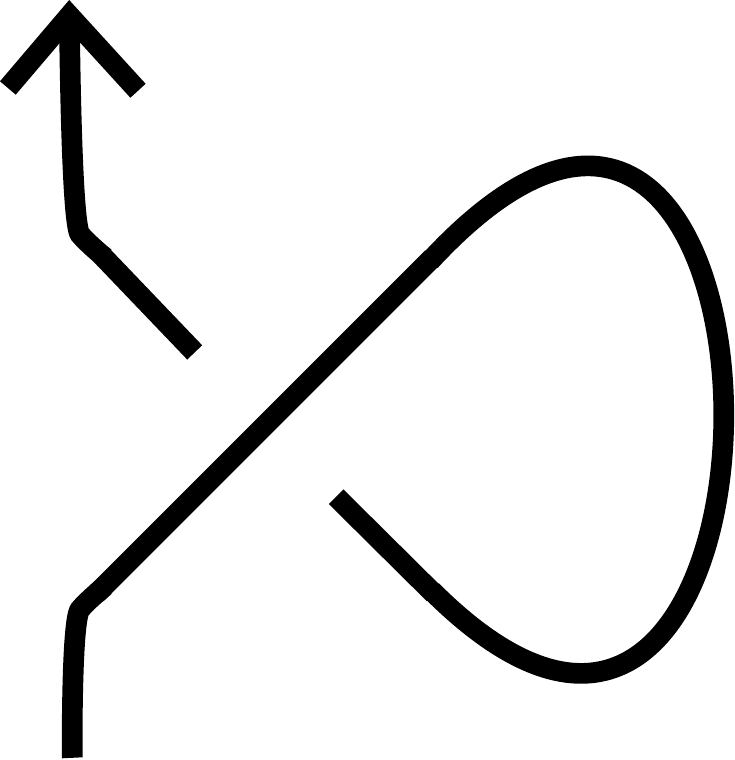} \end{array}$.

\begin{theorem}[\cite{kassel, KRT,turaev}]\label{thm:turaev}
Let $\C$ be a  strict ribbon category and let $X \in \C$ be an object. Then there exists a unique strict monoidal functor $$RT_X: \mathcal{T} \to \C$$ such that $RT_X(+)=X$ and $RT_X$ preserves the ribbon structure, that is, $RT_X(-)= X^*$ and $$RT_X(\, \PC \,)= \tau_{X,X} \ , \ RT_X(\begin{array}{c}
\centering
\includegraphics[scale=0.035]{figures/twist_gen} \end{array})= \theta_X  \ , \ RT_X(\, \capl \,)= \mathrm{ev}_X \ , \ RT_X(\, \cupl \, )= \mathrm{coev}_X.$$
In other words, $\mathcal{T}$ is the ``free strict ribbon category generated by a single object''.
\end{theorem}

We call the functor of the previous theorem the \textit{Reshetikhin-Turaev invariant} associated to $X \in \C$. This functor also provides a graphical calculus for ribbon categories, see \cite[Ch. 3]{turaevvirelizier}. The only difference with the conventions therein is that we orient the strands naturally from bottom to top, so that the source object $X$ of a morphism $f: X\to Y$ corresponds with the tail of the arrow and the target object $Y$ with the head.

\subsection{Ribbon Hopf algebras}\label{subsec:ribbon_algebras}

\cref{thm:turaev} asserts that in order to get a functorial tangle invariant (in particular a link invariant) that can be computed ``by pieces'' all we need is to find a ribbon category. The archetypal example of such a  category is the category of finite-free left-modules of a \textit{ribbon Hopf algebra}, that we recall next.

Let $(A, \mu, \eta, \Delta, \varepsilon, S)$ be  a Hopf algebra over some commutative ring $\Bbbk$. We will always assume that the antipode $S$ is invertible (this is automatic if $A$ is finite-free as a $\Bbbk$-module, that is, free and of finite rank, see \cite{pareigis}).

A \textit{quasi-triangular structure} over $A$ is a choice of an invertible element $R \in A \otimes A$, called the \textit{universal $R$-matrix}, satisfying the following properties:
\begin{enumerate}[leftmargin=4\parindent, itemsep=2mm]
\item[(QT1)] $(\Delta \otimes \id)(R) = R_{13}\cdot R_{23}$,
\item[(QT2)] $(\id \otimes \Delta)(R)=R_{13} \cdot R_{12}$,
\item[(QT3)] $  \Delta^{\mathrm{op}}  =   R \cdot \Delta(-) \cdot  R^{-1} , $
\end{enumerate}
where we have written $R_{12}=R\otimes 1$, $R_{13}=(\id_A \otimes \mathrm{flip}_{A,A})R_{12}$ and $R_{23}=1 \otimes R$. Here flip denotes the symmetric braiding of the category of $\Bbbk$-modules.
A Hopf algebra endowed with a quasi-triangular structure is called a \textit{quasi-triangular Hopf algebra}. 

It is readily seen that the universal $R$-matrix satisfies the \textit{Yang-Baxter equation }
\begin{equation}\label{eq:YB}
R_{12}R_{13}R_{23}=R_{23}R_{13}R_{12}
\end{equation}
and besides we have
\begin{equation}\label{eq:(S otimes id)R}
(S \otimes \id)(R)=R^{-1} \quad , \quad (\id \otimes S)(R^{-1})=R \quad , \quad (S \otimes S)(R^{\pm 1})=R^{ \pm 1}
\end{equation}
(the third is a consequence of the other two). We will typically write $$R= \sum_i \alpha_i \otimes \beta_i \qquad , \qquad R^{-1}= \sum_i  \bar{\alpha}_i \otimes \bar{\beta}_i $$ for the universal $R$-matrix and its inverse.

Let $(A,R)$ be a quasi-triangular Hopf algebra. The \textit{Drinfeld element} of the structure is the element $$u := \mu^{\mathrm{op}}(\id \otimes S)(R) = \sum_i S(\beta_i)\alpha_i.   $$ This element is known to be invertible with inverse $$u^{-1}= \mu^{\mathrm{op}}(S^2 \otimes \id)(R) = \sum_i \beta_i S^2(\alpha_i).    $$
 A \textit{ribbon structure} over $(A,R)$ is a choice of a grouplike square root of  $u S(u^{-1}) $ that implements $S^{2}$ by conjugation. More precisely, a ribbon structure is a choice of an element $\kappa \in A$ satisfying
\begin{enumerate}[leftmargin=4\parindent, itemsep=2mm]
\item[(R1)] $\kappa^2 = uS(u^{-1}) $,
\item[(R2)] $\Delta (\kappa) = \kappa \otimes  \kappa$,
\item[(R3)] $S^2 = \kappa\cdot (-) \cdot \kappa^{-1}$.
\end{enumerate}
Note that (R2) implies that $\varepsilon(\kappa)=1$ and that $\kappa$ is invertible with $S(\kappa)=\kappa^{-1}$. The element $\kappa$ will be called the \textit{balancing element}\index{balancing element} of the ribbon structure, and the triple $(A,R,\kappa)$ a \textit{ribbon Hopf algebra}\index{ribbon Hopf algebra}. It follows from \eqref{eq:(S otimes id)R} and (R3) that
\begin{equation}
(\id \otimes S)(R)= (S \otimes S^2)(R)= (\id \otimes S^2)(R^{-1})= \sum_i \bar{\alpha}_i \otimes \kappa \bar{\beta}_i \kappa^{-1},   
\end{equation}
and similarly
\begin{equation}
(S \otimes \id)(R^{-1})=  \sum_i \kappa \alpha_i \kappa^{-1}\otimes \beta_i,   
\end{equation}

The element $v :=  \kappa^{-1} u$ is classically called the \textit{ribbon element} (hence the name of the algebraic structure).  It can be shown \cite{kauffman_radford} that, for a quasi-triangular Hopf algebra, the set of axioms (R1)--(R3) for the balancing element are equivalent to the usual set of axioms
\begin{equation}\label{eq:ribbon_elmt_axioms}
 v \in \mathcal{Z}(A) \ \ , \ \ v^2 = u S(u)  \ \ , \ \  \Delta(v)=(R_{21}R)^{-1}(v \otimes v)  \ \ , \ \  \epsilon(v)=1  \ \ , \ \  S(v)=v
\end{equation}
for the ribbon element, where $\mathcal{Z}(A)$ denotes the centre of $A$.

Given a $\Bbbk$-algebra $A$, we write $\mathsf{fMod}_A$ for the category of finite-free left $A$-modules. It is  folklore that a ribbon structure $(R,\kappa)$ over $A$ induces structure of ribbon category over $\mathsf{fMod}_A$, as follows: given $A$-modules $V,W$, the $\Bbbk$-linear tensor product $V \otimes W$ is an $A$-module with product given by $a \cdot (v \otimes w):= \Delta (a) (v \otimes w)$. The base ring $\Bbbk$ is viewed as an $A$-module via $a \cdot \lambda := \varepsilon(a)  \lambda$, and it is the unit of the monoidal structure. The axioms (QT1)--(QT3) imply that for $A$-modules $V,W$, the map
\begin{equation}\label{eq:braiding_fMod}
\tau_{V,W}: V \otimes W \to W \otimes V \qquad , \qquad \tau_{V,W}(v \otimes w) := R_{21}(w \otimes v) 
\end{equation}
is $A$-linear and defines a braiding on $\mathsf{fMod}_A$. On the other hand, the $\Bbbk$-linear dual $V^*:= \hom{\Bbbk}{V}{\Bbbk}$ of $V$ can be endowed with a left $A$-module structure given by $(a\cdot \omega)(v):= \omega (S(a)v)$. Lastly, there is a twist $\theta_W: W \to W$ given by multiplication by the inverse of the ribbon element,
\begin{equation}\label{eq:twist_fMod}
\theta_W(w):=v^{-1}w.
\end{equation} 
 In fact, for a finite-free $\Bbbk$-algebra $A$, a ribbon category structure on $\mathsf{fMod}_A$ uniquely determines a ribbon Hopf algebra on $A$, see e.g. \cite{becerra_thesis} for a detailed discussion.
%\mathrm{flip}_{V,W}(R(v \otimes w)) 

We would like to emphasise that the category $\mathsf{fMod}_A$ is not strict (because the canonical $\Bbbk$-linear isomorphism $(V_1 \otimes V_2) \otimes V_3 \cong V_1 \otimes (V_2 \otimes V_3)$ is not the identity map)  and therefore, stricto senso, $\mathsf{fMod}_A$ cannot be used directly in \cref{thm:turaev}. However, this is a non-issue due to the celebrated Mac Lane's  strictness theorem \cite{maclane}, which ensures that there exists a strict monoidal category $\mathsf{fMod}_A^{\mathrm{str}}$  together with a  monoidal equivalence $ \mathsf{fMod}_A \overset{\simeq}{\to} \mathsf{fMod}_A^{\mathrm{str}}$ (see also \cite{becerra_strictification} for a more modern monograph). In particular, this means that $\mathsf{fMod}_A^{\mathrm{str}}$ admits a  ribbon structure such that the previous is an equivalence of ribbon categories. The strict ribbon category $\mathsf{fMod}_A^{\mathrm{str}}$  is the desired category.

We will always assume the  model for the strictification $\mathsf{fMod}_A^{\mathrm{str}}$ based on sequences of objects \cite{maclane,kassel,becerra_strictification}: the objects of $\mathsf{fMod}_A^{\mathrm{str}}$ are sequences $S=(V_1, \ldots, V_n)$ of finite-free $A$-modules , $n \geq 0$ (this includes the empty sequence $\emptyset$). If the \textit{parenthesisation} of a sequence $S$ is the $A$-module
\begin{equation*}
Par(S) := ( \cdots (V_1 \otimes V_2) \otimes V_3) \otimes \cdots  ) \otimes V_n
\end{equation*}
for any sequence $S$ of length $n >0$ and $Par(\emptyset): =\bm{1}$, define 
\begin{equation}\label{eq:1}
\hom{\mathsf{fMod}_A^{\mathrm{str}}}{S}{S'} := \hom{\mathsf{fMod}_A}{Par(S)}{Par(S')},
\end{equation}
with identities and composition taken from those of $\mathsf{fMod}_A$. The monoidal structure is given by concatenation of sequences, with unit the empty sequence. We view any object $X$ of $\C$ in  $\Cstr$ as a one-object sequence.

\subsection{The category \texorpdfstring{$\T^+$}{T+} }\label{subsec:T+}

The category $\T$ of tangles contains important subcategories that we would like to recall next. Let us write $\T^+$ for the full subcategory of $\T$ on the objects $\mathrm{Mon}(+) \subset \mathrm{Mon}(+,-)$. The main difference of this category with respect to $\T$ is that isolated cups and caps $\cupr$ , $\cupl$ ,  $\capr$ , $\capl$ are not allowed, although they can still appear in a decomposition of an arrow in $\T^+$. In particular, closed components are allowed and $\hom{\T^+}{\emptyset}{\emptyset}$ is the set of framed, oriented links.

The category $\T^+$ satisfies a universal property, that we now discuss. First recall that a monoidal category is said to be \textit{balanced} if it is endowed with a braiding and a twist. Let $\C$ be a strict balanced category with monoidal product $ \otimes$, unit $\bm{1}$, braiding $\tau$ and twist $ \theta$. A \textit{trace} \index{trace} for $\C$ is a family of natural set-theoretical maps
\begin{equation}
\mathrm{tr}_{X,Y}^U : \hom{\C}{X \otimes U}{Y \otimes U} \to \hom{\C}{X }{Y}
\end{equation}
satisfying the following axioms:
\begin{enumerate}[leftmargin=4\parindent]
\item[(TC1)] For every $f:X \to Y$,  $$\mathrm{tr}_{X,Y}^{\bm{1}}(f)=f.$$
\item[(TC2)] For every $g: X \otimes U \otimes V  \to  Y \otimes U \otimes V  $,  $$\mathrm{tr}_{X,Y}^{U \otimes V}(g)= \mathrm{tr}_{X,Y}^U (\mathrm{tr}_{X \otimes U,Y \otimes U}^V(g)).$$
\item[(TC3)] For every $f: X_1 \otimes U \to Y_1 \otimes U$  and every $g: X_2  \to Y_2 $,  
\begin{align*}
\mathrm{tr}_{X_1,Y_1}^U(f) \otimes g &= \mathrm{tr}_{X_1 \otimes X_2,Y_1 \otimes Y_2}^U \Big( (\id_{Y_1} \otimes \tau^{-1}_{Y_2,U})(f \otimes g)(\id_{X_1} \otimes\tau_{X_2,U})  \Big) \\
&= \mathrm{tr}_{X_1 \otimes X_2,Y_1 \otimes Y_2}^U \Big( (\id_{Y_1} \otimes \tau_{U,Y_2})(f \otimes g)(\id_{X_1} \otimes\tau^{-1}_{U,X_2})  \Big).
\end{align*}
\item[(TC4)] For every object $U$ in $\C$, 
$$ \mathrm{tr}_{U,U}^U(\tau_{U,U}) = \theta_U \qquad , \qquad  \mathrm{tr}_{U,U}^U(\tau^{-1}_{U,U}) = \theta^{-1}_U.$$
\end{enumerate}
A strict balanced monoidal category endowed with a trace is called a \textit{traced monoidal category}. This notion was introduced first by Joyal, Street and Verity \cite{joyal_street_traced}.

The authors show in  \cite{joyal_street_traced} that every ribbon category $\C$ is canonically traced:  given objects $X,Y,U$ in $\C$ and an arrow $f: X \otimes U \to Y \otimes U$, the trace $\mathrm{tr}_{X,Y}^U(f)$ is defined as the composite
\begin{equation}\label{eq:trace_for_ribbon}
\begin{tikzcd}[column sep=4em]
X \rar{\id_X \otimes \mathrm{coev}_U } & X \otimes U \otimes U^* \rar{f \otimes \id_{U^*}} & Y \otimes U \otimes U^* \rar{\id_Y \otimes \widetilde{\mathrm{ev}}_U} & Y
\end{tikzcd}
\end{equation}
%In the above composite, $\widetilde{\mathrm{coev}}_U $ stands for the coevaluation for the right rigid structure.
They also prove that the free traced monoidal category generated by one object is precisely  $\T^+$:
\begin{theorem}[\cite{joyal_street_traced}]\label{thm:UP_traced}
Let $\C$ be a  traced monoidal category, and let $X$ be an object of $\C$. Then there exists a unique strict monoidal functor $$F: \T^+ \to \C$$ such that $F(+)=X$ and $F$ preserves the braiding, twist and the trace in the sense that $F(\mathrm{tr}_{+^n , +^n}^{+^m}(f))=\mathrm{tr}_{X^{\otimes n}, X^{\otimes n}}^{X^{\otimes m}}(Ff) $.
\end{theorem}

Just as before, this theorem provides a graphical calculus for traced monoidal categories,  see \cite[\S 2]{joyal_street_traced}.

\subsection{Bundle monoidal categories}\label{subsec:constr_M_n} We now briefly recall a construction that will feature a few times in what follows. 
Let $\mathcal{M}=(M_i)_{i \in \N}$ be a family of monoids indexed by nonnegative integers, and suppose that there is a family of monoid homomorphisms $$\rho_{n,m}: M_n \times M_m \to M_{n+m} \qquad , \qquad  n,m \in \N$$ satisfying
\begin{equation}\label{eq:M_n1}
 \rho_{0,n} \circ (1_{M_0}\times\id_{M_n}  ) = \id_{M_n} = \rho_{n,0}\circ (\id_{M_n} \times 1_{M_0})
\end{equation}
and  
\begin{equation}\label{eq:M_n2}
\rho_{n+m,r} \circ (\rho_{n,m} \times \id_{M_r}) = \rho_{n,m+r}   \circ (\id_{M_n} \times \rho_{m,r})  
\end{equation}
for all $n,m \geq 0$, where $1_{M_0}$ is the unit of $M_0$.  Then it is easy to see that this family of monoids gives rise to a strict monoidal category $\mathcal{M}$ defined as follows: its objects are the nonnegative integers, and given $n,m \in \N$ declare
$$ \hom{\mathcal{M}}{n}{m}:= \begin{cases} M_n, & n=m \\ \emptyset, & n \neq m \end{cases} ,$$
with composite determined by the multiplication law of the monoids $M_n$, $x \circ y := xy$, $x, y \in M_n$, and identity given by the unit  of $M_n$. The  monoidal product on this category is given by addition of integers for objects and the structure maps $\rho_{n,m}$ for the arrows, with unit object given by $0 \in \N$. If the monoids $(M_n)$ are in fact groups, then the resulting category $\mathcal{M}$ is a groupoid.

\begin{example}
Let $\SS_n$ be the symmetric group of $n$ elements, and for every $n,m \in \N$, let $\rho_{n,m}: \SS_n \times \SS_m \to \SS_{n+m}$ denote the block product of permutations, which satisfies the conditions above. The strict monoidal category resulting from applying the previous construction is called the \textit{permutation category}  and will be denoted by $\SS$. Furthermore, we can actually endow this category with a balanced structure: indeed a symmetric braiding $\tau_{n,m}: n+m \to n+m$ is given by the permutation that maps $i$ to $i+m \pmod {n+m}$, and a twist is given by the identity natural transformation. In fact, $\SS$ is the free symmetric monoidal category generated by one object.
\end{example}

\begin{remark}\label{rmk:functor_M->C}
It directly follows from the definitions that if $\C$ is a strict monoidal category, then the data of a strict monoidal functor $\mathcal{M} \to \C$ is fully determined by a choice of an object $X$ in $\C$ and a family of monoid homomorphisms $$ \varphi_n: M_n \to \mathrm{End}_\C(X^{\otimes n}) \qquad , \qquad n \geq 0   $$ such that the following diagram commutes for all $n,m \geq 0$:
\begin{equation}
\begin{tikzcd}
M_n \times M_n \dar{\varphi_n \times \varphi_m}  \rar{\rho_{n,m}} & M_{n+m} \dar{\varphi_{n+m}}\\
\mathrm{End}_\C(X^{\otimes n}) \times \mathrm{End}_\C(X^{\otimes n}) \rar{\otimes } & \mathrm{End}_\C(X^{\otimes n +m})
\end{tikzcd}
\end{equation}
\end{remark}

\subsection{The categories \texorpdfstring{$\B$}{B} and \texorpdfstring{$\B^0$}{B0}}

Now we want to introduce two more remarkable subcategories of $\T$. We write $\B$ for the subcategory of $\T$ on the objects $\mathrm{Mon}(+) \subset \mathrm{Mon}(+,-)$ and arrows open tangles $T$ with the property that for some representative $\gamma$ of the isotopy class $T$, we have that the intersection of the union of the cores of $\gamma$ with the hyperplanes $D^1 \times D^1 \times \{ z \} \subset (D^1)^{\times 3}$ is constant for all $z \in D^1$.   We call $\B$ the category of \textit{framed braids} in a cube.

Similarly, the category $\B^0$ is the wide subcategory of $\B$ on the arrows 0-framed tangles, that is, tangles whose all components have framing zero. We call $\B$ the category of \textit{braids} in a cube. It is readily seen that both are monoidal subcategories of $\T$.

These two categories admit a description in terms of the construction from \cref{subsec:constr_M_n}. Let us write $B_n$ for the braid group in $n$ strands (also known as the fundamental group $\pi_1 (\mathrm{UConf}_n(\mathbb{C}))$ of the $n$-th unordered configuration space of $\mathbb{C}$, or alternatively the mapping class group $MCG(D^2 - \{n \text{ points}\})$ of the $n$-punctured disc). There is a surjective group homomorphism $\pi: B_n \to \SS_n$ which maps each of the Artin generators $\sigma_i$ to the transposition $s_i:=(i,i+1)$. We call the \textit{framed braid group} in $n$ strands to the semidirect product $$B_n^{fr} := B_n \ltimes \Z^n  $$ where $B_n$ acts on $\Z^n$ permuting the components via $\pi$, that is $\sigma (k_1, \ldots , k_n) := (k_{\pi\sigma(1)}, \ldots , k_{\pi\sigma(n)})$. There are families of group homomorphisms $$ \rho_{n,m}: B_n \times B_m \to B_{n+m} \qquad , \qquad  \rho_{n,m}^{fr}:B_n^{fr} \times B_m^{fr} \to B_{n+m}^{fr} $$ determined by $\rho_{n,m}(\sigma_i, \sigma_j)=\sigma_i\sigma_{n+j}$ and similarly for  $\rho_{n,m}^{fr}$. It is easy to check that they satisfy \eqref{eq:M_n1} and \eqref{eq:M_n2}, and the monoidal categories  resulting from applying the construction from \cref{subsec:constr_M_n} are (isomorphic to) $\B^0$ and $\B$, respectively.

The categories $\B^0$ and $\B$ also enjoy some universal properties:  $\B^0$ is the free braided monoidal category generated by one object, and  $\B$ is the free balanced category generated by one object:

\begin{theorem}
Let $\C$ be a strict braided (resp. balanced) monoidal category, and let $X$ be an object of $\C$. There there exists a unique strict monoidal functor $$F: \B^0 \to \C \qquad \qquad \text{(resp.  } F: \B \to \C \text{)}$$
such that $F(+)=X$ and $F$ preserves the braiding (resp. the braiding and the twist). 
\end{theorem}

In summary, we have the following

\begin{corollary}\label{cor:chain_embeddings}
We have the following chain of strict monoidal embeddings,
$$\mathcal{B}^0 \hooklongrightarrow \mathcal{B} \hooklongrightarrow \T^+ \hooklongrightarrow \T,$$
where $\mathcal{B}^0$, $\mathcal{B}$, $\Tup$, $\T^+$ and $\T$ are the free braided, free balanced, free traced and  free ribbon monoidal categories generated by a single object, respectively.
\end{corollary}

\section{Upwards tangles}\label{sec:upwards}

So far we have introduced four different categories of braids and tangles, some being more restrictive and some being more general in the classes of isotopy classes allowed. The universal balanced category $\B$ has the advantage that one does not have to deal with closed components, which sometimes might pose problems in certain constructions, but it is rather restrictive as it only allows framed braids. On the other hand, the universal traced category $\T^+$ has the advantage that it allows very general classes of tangles (it is a full subcategory of $\T$) without requiring duality, but it has as downside  that one still has to deal with closed components. In this section we will introduce a class of tangles that retains the best features of both $\B$ and $\T^+$.

\subsection{The category \texorpdfstring{$\Tup$}{Tup} }

Here is the relevant definition: the category of \textit{upwards tangles} $\Tup$ is the subcategory of $\T$ on the objects $\mathrm{Mon}(+) \subset \mathrm{Mon}(+,-)$ and arrows open tangles. It is clear that $\Tup$ is a monoidal subcategory of $\T$ that sits in between $\B$ and $\T^+$: there are monoidal embeddings $$\mathcal{B} \hooklongrightarrow \Tup \hooklongrightarrow \T^+ ,$$ in particular $\Tup$ is balanced. Besides, this category is sufficiently general to encode knot theory: the canonical closure map
\begin{equation*}
\centre{
\labellist \small \hair 2pt
\pinlabel{$K$}  at 58 165
\pinlabel{$K$}  at 423 165
\pinlabel{$\mapsto$}  at 240 165
\endlabellist
\centering
\includegraphics[width=0.25\textwidth]{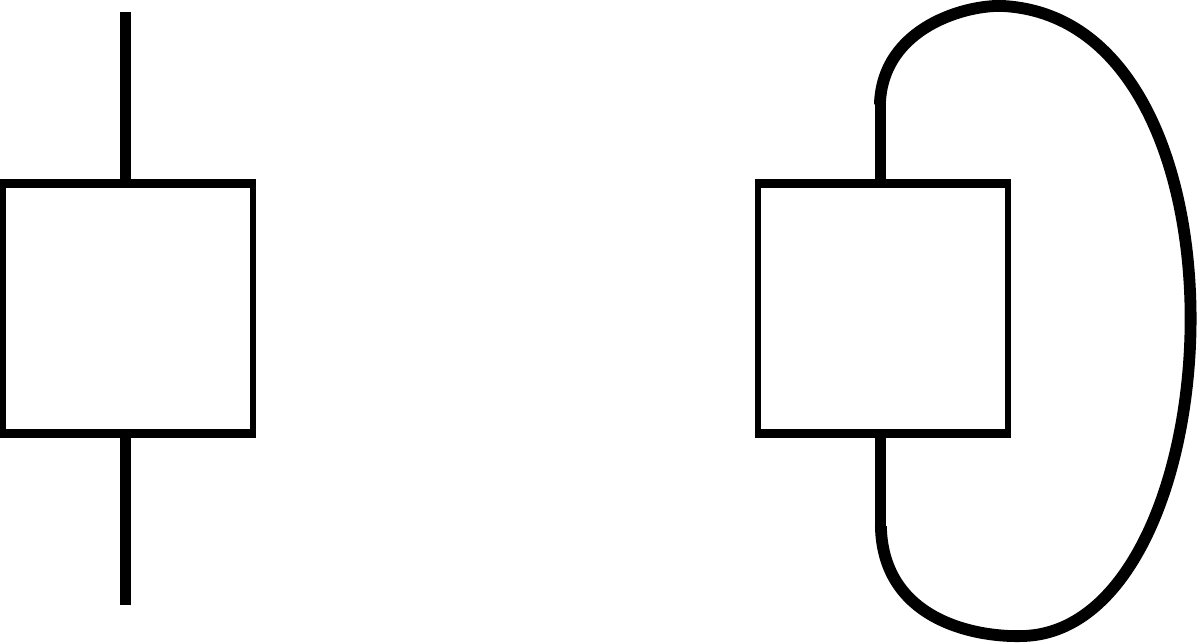}}
\end{equation*}
establishes a bijection
\begin{equation}\label{eq:bijection_open_closed_knots}
\begin{tikzcd}[column sep=3em]
\frac{ \left\{ \parbox[c]{9em}{\centering     {\small  \textnormal{one-component upwards tangles in $(D^1)^{\times 3}$ }}} \right\} }{ \parbox[c]{9em}{\centering \vspace*{2pt} {\small  \textnormal{isotopy rel. endpoints} }}} \ar[-,double line with arrow={-,-}]{r} &    \frac{ \left\{ \parbox[c]{6em}{\centering     {\small  \textnormal{framed knots in $(D^1)^{\times 3}$ }}} \right\} }{ \parbox[c]{8em}{\centering \vspace*{2pt} {\small  \textnormal{isotopy} }}}
\end{tikzcd}
\end{equation}

The category $\Tup$ can also be seen as arising from the bundle construction from \cref{subsec:constr_M_n}. For every $n \geq 0$, the endomorphism monoid $\mathrm{End}_{\Tup}(+^n)$ will be called the monoid of $n$-components upwards tangles. By the monoidal category axioms, \eqref{eq:M_n1} and \eqref{eq:M_n2} hold, and the resulting category is (trivially) $\Tup$. Note that by the same reasoning, we could view the category $\T^+$ also arising from the same construction.

For every $n$-component upwards tangle, let  $\sigma_T \in \SS_n$ be the permutation given by sending $i$ to the position in $t(T)$ (reading from left to right) where the component which starts at the $i$-th position in $s(T)$ finishes at. This assignment promotes to a monoid homomorphism
\begin{equation}\label{eq:pi_Tup->S}
\pi: \mathrm{End}_{\Tup}(+^n)  \to \SS_n
\end{equation}
whose kernel (that is, those upwards tangles whose components start and end at the same position) consists of the so-called \textit{$n$-string links} \cite{BN_vassiliev,meilhan,MWY}. Here we prefer keep the word \textit{link} for closed components, so we call those elements in the kernel \textit{pure} upwards tangles as in \cite{BBK}, in analogy to the pure braid group $PB_n = \ker (B_n \to \SS_n)$.

The strands of an upwards tangle are canonically ordered, labelling them at the tails from left to right. We will always assume this order without further mention.

We record here two important properties:

\begin{proposition}[\cite{schubert, krebes,BBK}]
The upwards tangle monoids satisfy the following properties:
\begin{enumerate}
\item The invertible upwards tangles are exactly the framed braids, $$(\mathrm{End}_{\Tup}(+^n))^\times = B_n^{fr}.$$
\item The monoids $\mathrm{End}_{\Tup}(+^n)$, $n \geq 1$, are infinitely-generated.
\end{enumerate}
\end{proposition}

As an immediate consequence, we find that $\Tup$ is infinitely generated as a monoidal category, and  isomorphisms in $\Tup$ are exactly framed braids.

\subsection{Rotational diagrams}

We will mostly focus on a convenient class of diagrams of upwards tangles. A diagram $D$ of an upwards tangle is said to be  \textit{rotational}\index{rotational diagram} if all crossings of $D$ point upwards and all maxima and minima appear in pairs of the following two forms,
\begin{equation}\label{eq:full_rotations}
\centre{
\centering
\includegraphics[width=0.27\textwidth]{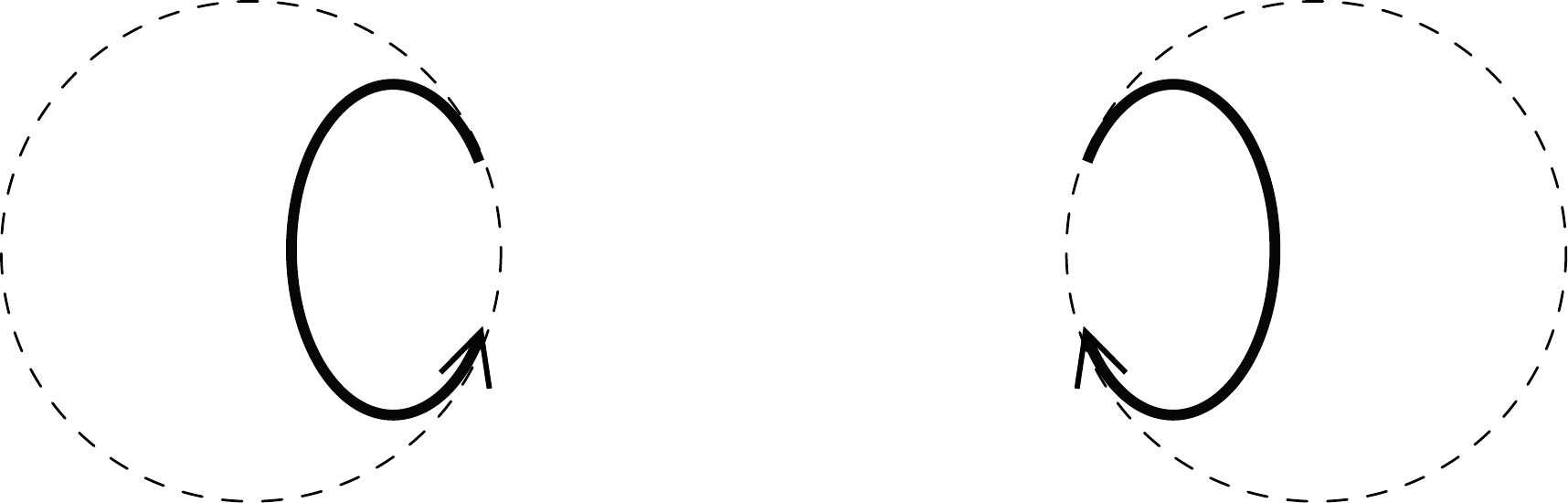}}
\end{equation}
where the dashed discs denote some neighbourhoods of that piece of strand in the tangle diagram. The idea of considering digrams where only full rotations are allowed first appeared in \cite{barnatanveenpolytime,barnatanveengaussians} and was further developed by the author in \cite{becerra_gaussians}. 
We regard rotational tangle diagrams up to \textit{Morse isotopy}\index{Morse isotopy}, that is, planar isotopy that preserve all maxima and minima. In other words, we do not allow isolated cups and caps (``half rotations''), instead they must appear in pairs, either  $\capl$ and  $\cupr$ or  $\capr$ and  $\cupl$,  forming full rotations, as depicted in \eqref{eq:full_rotations}.

\begin{lemma}[\cite{becerra_gaussians}]\label{lemma:every_tangle_has_rot_diag}
Any upwards tangle has a rotational diagram.
\end{lemma}
\begin{proof}
Given a tangle diagram $D$ of an upwards tangle, we will construct another diagram in rotational form which is related to the former only by planar isotopy.

For each tangle component, label the edges of the strand according to the orientation. Let the pair $(i,j)$ denote the $j$-th edge of the strand $D_i$. By an edge of a tangle diagram we mean an edge of the underlying uni-tetravalent graph of $D$. Write $X_{(i,j), (i',j')}^\pm$ for the crossing that has the edge $(i,j)$ as the foot of the overstrand and $(i',j')$ as the foot of the understrand, where $ \pm$ indicates whether the crossing is positive or negative. We define a total order in the set of labels $(i,j)$ as follows: $$(i,j)<(i',j') \text{ if } i<i' \text{ or } i=i'  \text{ and } j<j'.$$
This induces the following total order in the set $\mathrm{cross}(D)$ of crossings of the digram:
$$X_{(i_1,j_1), (i'_1,j'_1)}^\pm   < X_{(i_2,j_2), (i'_2,j'_2)}^\pm  \text{ if }  \min ((i_1,j_1), (i'_1,j'_1))< \min ((i_2,j_2), (i'_2,j'_2)).$$

Now, let us construct a new diagram for the given upwards tangle. First, start by placing the feet of all components of $D$, from left to right following the order of the components. Now,  according to the order in $\mathrm{cross}(D)$,  place the crossings in the bands $\R  \times [k,k+1]$ of the plane  in an upward fashion, placing a cup at the end of the foot of the edge with the greatest pair. We illustrate this below with $(i,j) < (i',j')$:

\vspace{0.2cm}
\begin{equation*} 
\labellist \tiny \hair 2pt
\pinlabel{$(i,j)$}  at 44 -35
\pinlabel{$(i',j'+1)$}  at -15 195
\pinlabel{$(i,j+1)$}  at 230 195
\pinlabel{$(i',j')$}  at 345 65
\pinlabel{$(i,j)$}  at 655 -35
\pinlabel{$(i',j'+1)$}  at 620 195
\pinlabel{$(i,j+1)$} [l] at 749 195
\pinlabel{$(i',j')$}  at 960 65
\pinlabel{$(i,j)$}  at 1536 -35
\pinlabel{$(i',j'+1)$}  at 1630 195
\pinlabel{$(i,j+1)$} [r] at 1500 195
\pinlabel{$(i',j')$}  at 1230 65
\pinlabel{$(i,j)$}  at 2150 -35
\pinlabel{$(i',j'+1)$}  at 2250 195
\pinlabel{$(i,j+1)$} [r] at 2120 195
\pinlabel{$(i',j')$}  at 1850 65
\pinlabel{\normalsize{$X_{(i,j), (i',j')}^+$}} at 200 -140
\pinlabel{\normalsize{$X_{(i',j'),(i,j) }^-$}} at 790 -140
\pinlabel{\normalsize{$X_{(i',j'),(i,j) }^+$}} at 1450 -140
\pinlabel{\normalsize{$X_{(i,j), (i',j')}^-$}} at 2080 -140
\endlabellist
\centering
\includegraphics[width=0.87\textwidth]{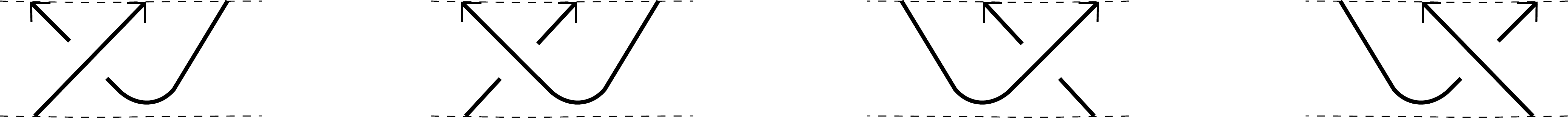}
\end{equation*}
\vspace*{15pt}

\noindent
For each $i$, we connect the edges $(i,j)$ according to the order and extend all edges up if they have not been connected yet.  In doing so, we must place some caps when we have to merge with an edge that is already on the diagram. The resulting diagram has cups and caps appearing in pairs as in  \eqref{eq:full_rotations}, but also isolated maxima and zig-zag curves, as shown below to left and right respectively:
\begin{equation*}
\centre{
\centering
\includegraphics[width=0.45\textwidth]{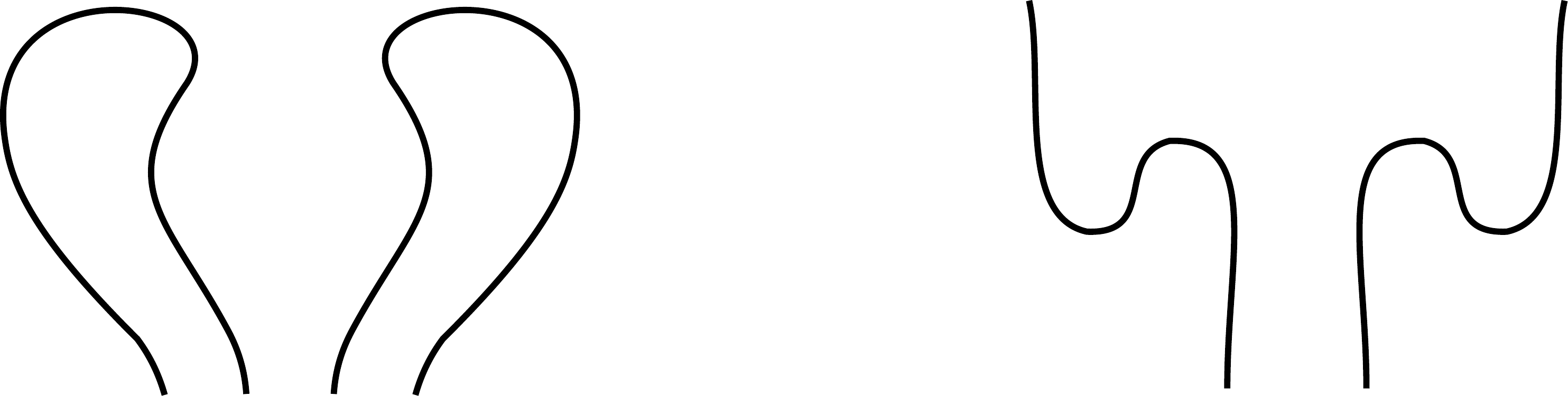}}
\end{equation*}
However these two can be removed by a planar isotopy. By construction, the resulting diagram is planar isotopic to the original one.
\end{proof}

\begin{example}
Let us illustrate the proof of the  previous lemma. Suppose we start with the 2-component tangle showed below:
\begin{equation*}
\labellist \tiny \hair 2pt
\pinlabel{$1$}  at 45 50
\pinlabel{$2$}  at 195 321
\pinlabel{$2'$}  at 394 321
\pinlabel{$1'$}  at 538 50
\pinlabel{$3'$}  at 195 58
\pinlabel{$3$}  at 400 58
\pinlabel{$4'$}  at 55 313
\pinlabel{$4$}  at 545 313
\endlabellist
\centering
\includegraphics[width=0.2\textwidth]{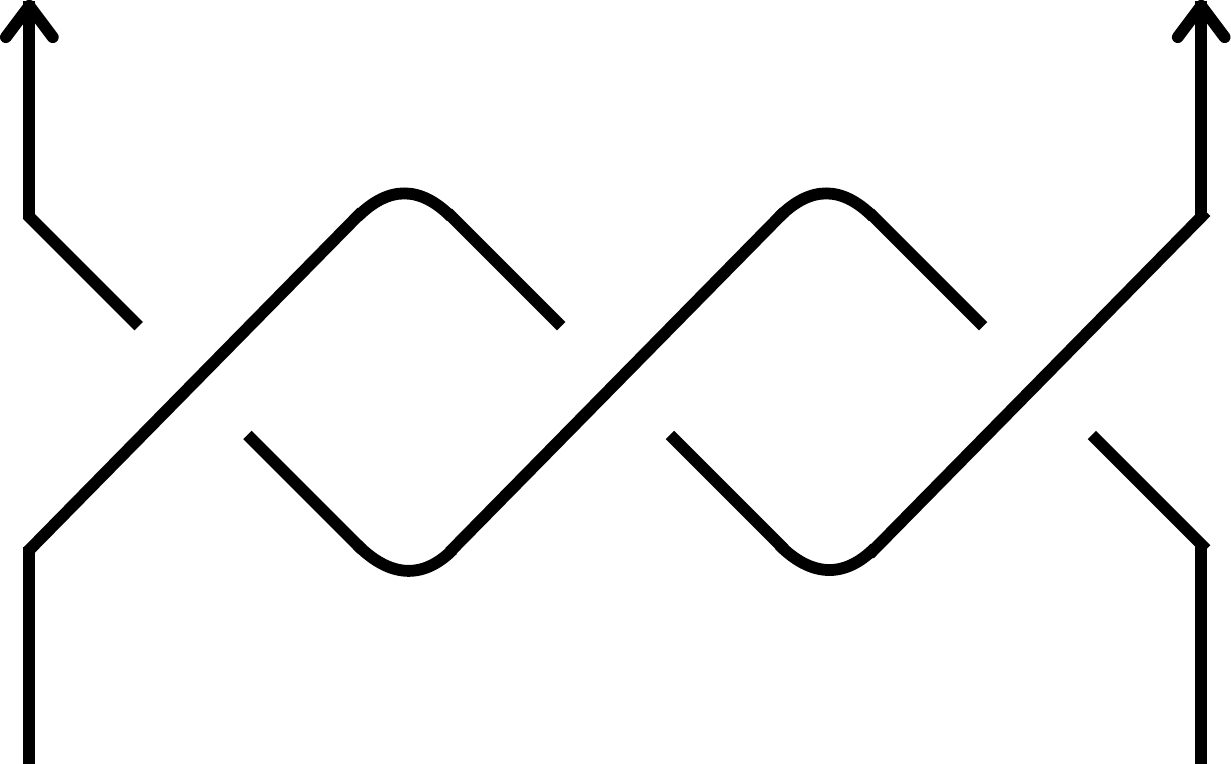}
\end{equation*}
For simplicity we  have labelled the edges  as $j=(1,j)$ and $j'=(2,j)$. The algorithm described produces the following tangle diagram:
\begin{equation*}
\labellist \tiny \hair 2pt
\pinlabel{$1$}  at  228 52
\pinlabel{$2$}  at 406 214
\pinlabel{$3$}  at 236 400
\pinlabel{$4$}  at 355 550
\pinlabel{$1'$}  at 718 64
\pinlabel{$2'$}  at 200 600
\pinlabel{$3'$}  at 555 290
\pinlabel{$4'$}  at 41 432
\endlabellist
\centering
\includegraphics[width=0.26\textwidth]{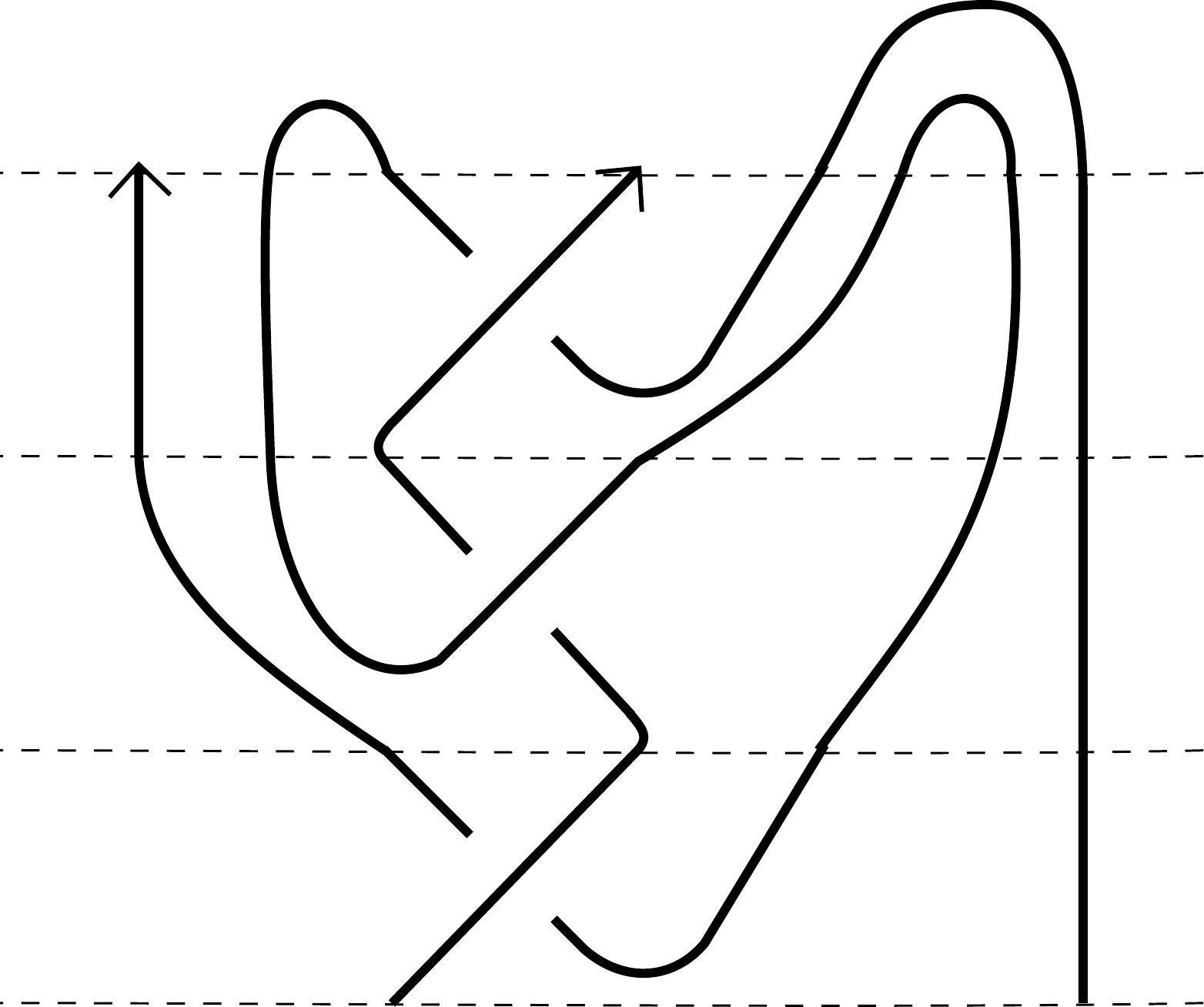}
\end{equation*}
Removing isolated maxima, we obtain the following rotational diagram of the original tangle:
\begin{equation*}
\centering
\includegraphics[width=0.2\textwidth]{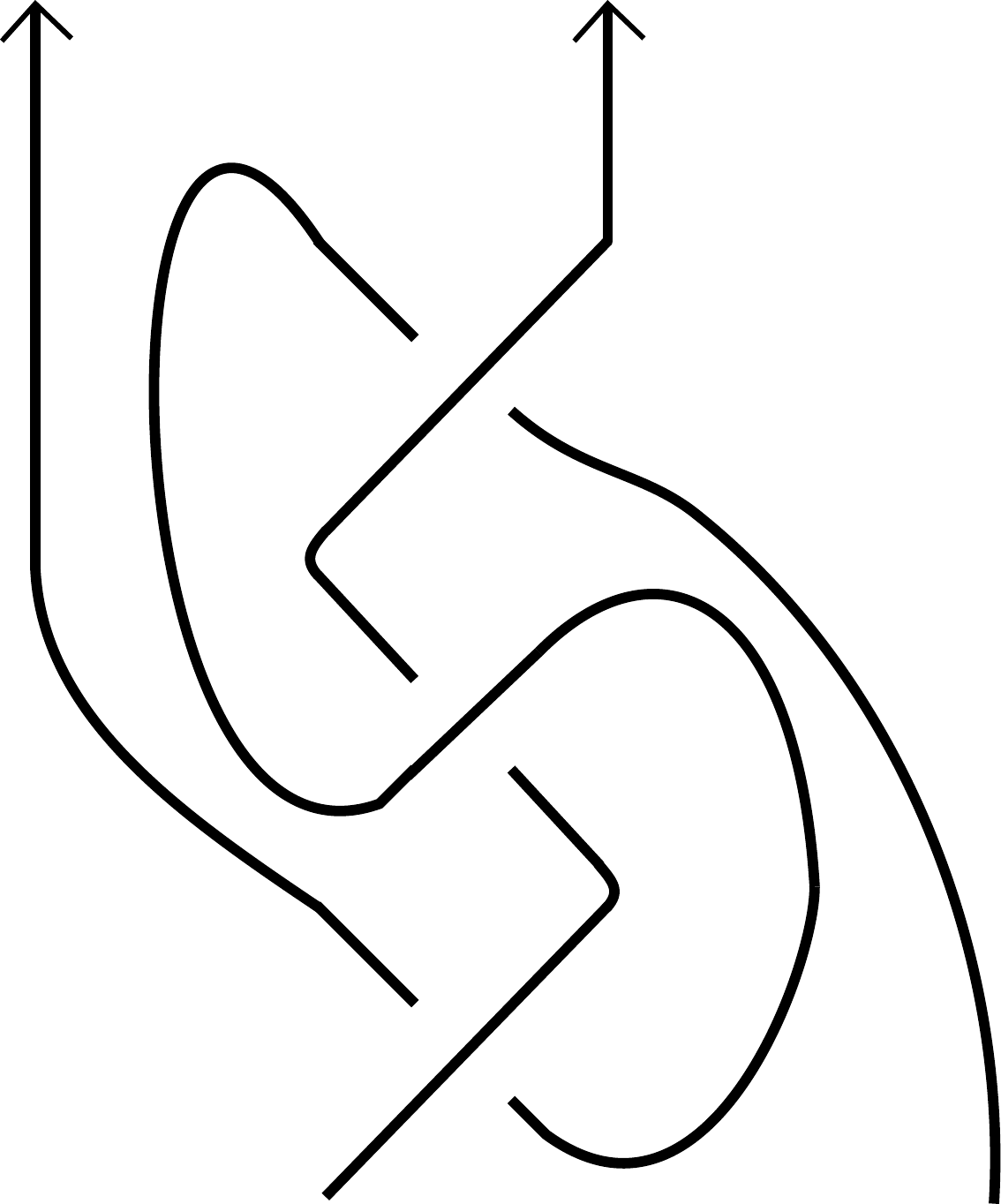}
\end{equation*}
\end{example}

We can hence restrict ourselves to study only tangle diagrams in rotational form. It was shown in \cite{becerra_reidemeister} that the following relations form a complete set of \textit{rotational Reidemeister moves}\index{rotational Reidemeister moves} for rotational tangle diagrams:
\begin{equation}\label{eq:R0}
\centre{
\labellist \small \hair 2pt
\pinlabel{$\overset{(\Omega 0a)}{=}$}  at 360 210
\pinlabel{$\overset{(\Omega 0b)}{=}$}  at 570 210
\pinlabel{$\overset{(\Omega 0c)}{=}$}  at 1725 210
\pinlabel{$,$}  at 1150 180
\endlabellist
\centering
\includegraphics[width=0.75\textwidth]{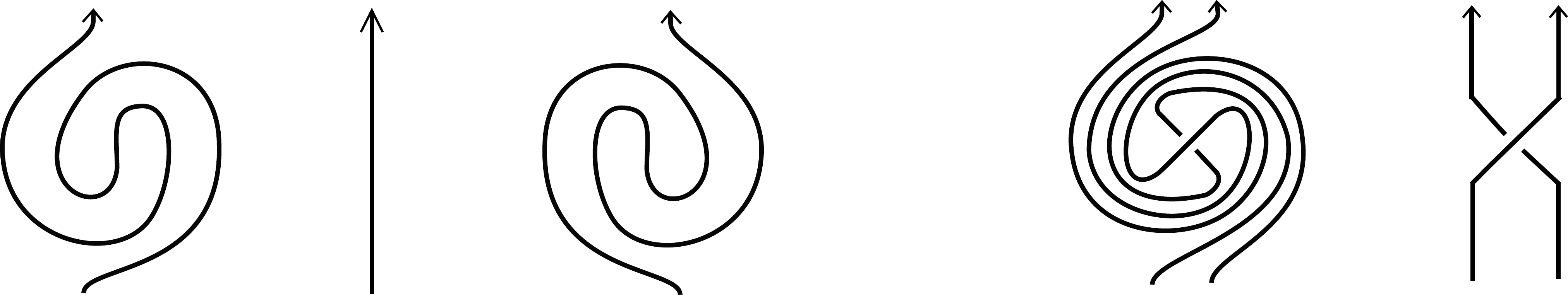}}
\end{equation}
\vspace*{7pt}
\begin{equation}\label{eq:R1_andR2a}
\centre{
\labellist \small \hair 2pt
\pinlabel{$\overset{(\Omega 0d)}{=}$}  at 375 210
\pinlabel{$\overset{(\Omega 1f)}{=}$}  at 1470 210
\pinlabel{$,$}  at 875 180
\endlabellist
\centering
\includegraphics[width=0.7\textwidth]{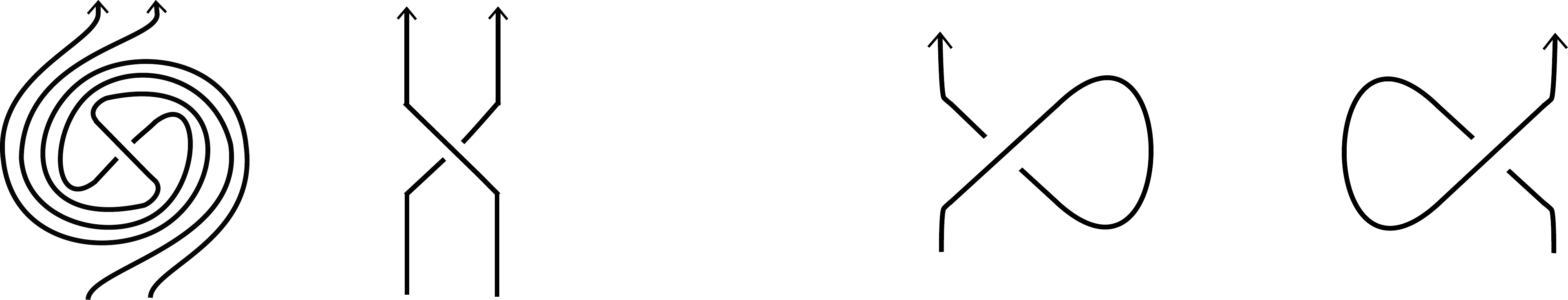}}
\end{equation}
%\vspace*{10pt}
\begin{equation}
\centre{
\labellist \small \hair 2pt
\pinlabel{$\overset{(\Omega 2a)}{=}$}  at 280 200
\pinlabel{$\overset{(\Omega 2b)}{=}$}  at 655 200
\pinlabel{$\overset{(\Omega 2c)}{=}$}  at 1640 200
\pinlabel{$,$}  at 1130 170
\endlabellist
\centering
\includegraphics[width=0.75\textwidth]{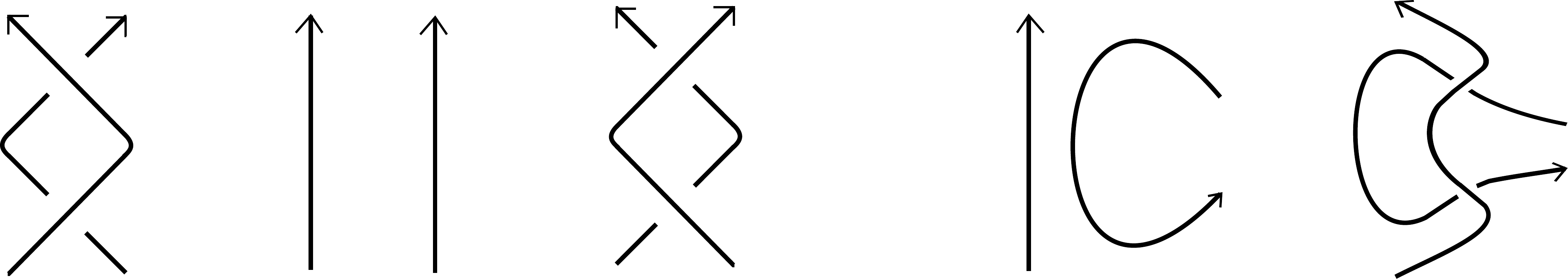}}
\end{equation}
\vspace*{7pt}
\begin{equation}\label{eq:R2c_and_R3}
\centre{
\labellist \small \hair 2pt
\pinlabel{$\overset{(\Omega 2d)}{=}$}  at 330 200
\pinlabel{$\overset{(\Omega 3)}{=}$}  at 1450 200
\pinlabel{$,$}  at 890 170
\endlabellist
\centering
\includegraphics[width=0.75\textwidth]{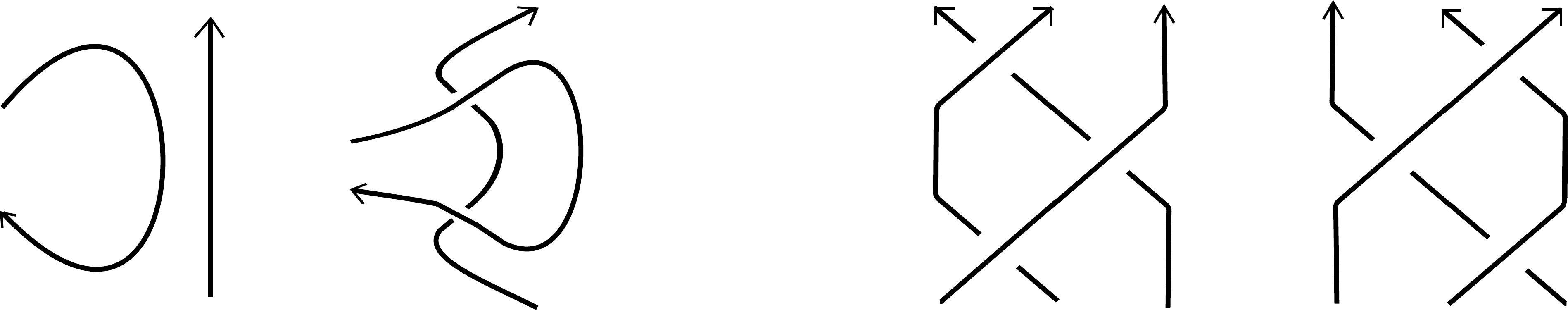}}
\end{equation}
\vspace*{5pt}

\noindent As usual, the above pictures are to be understood as identifying two tangle diagrams that are identical except in an open neighbourhood, where they look like as shown.

From this it follows

\begin{corollary}\label{cor:uptangles_rotdiag}
There are bijections
\vspace*{5pt}
\begin{center}
\begin{tikzcd}[column sep=2em,row sep=2ex]
 \frac{ \left\{ \parbox[c][2.5em]{7em}{\centering
                      {\small  \textnormal{upwards tangles \\ in $(D^1)^{\times 3}$ }}} \right\} }{  \parbox[c][1.5em]{8em}{\centering  {\small  \textnormal{isotopy} }}} \arrow[equals]{r} &   \frac{ \left\{ \parbox[c][2.5em]{8em}{\centering
                      {\small  \textnormal{upwards  tangle diagrams in $(D^1)^{\times 2}$ }}} \right\} }{ \parbox[c][2.5em]{8em}{\centering  {\small \textnormal{  planar isotopy and Reidemeister moves }}}} \arrow[equals]{r} &   \frac{ \left\{ \parbox[c][2.5em]{8em}{\centering
                      {\small  \textnormal{rotational tangle diagrams in $(D^1)^{\times 2}$ }}} \right\} }{ \parbox[c][3.8em]{8em}{\centering  {\small  \textnormal{Morse isotopy and rotational Reidemeister moves }}}}
\end{tikzcd}
\end{center}
\end{corollary}

It should be clear that similar statement can also be made for tangles in $\T^+$.

\begin{remark}
Observe that rotational diagrams do not fit very well with the ribbon category structure of $\mathcal{T}$, as the full rotations of \eqref{eq:full_rotations} cannot be seen as morphisms. However, any rotational diagram of an upwards tangle can be decomposed as the merging in the plane of the elementary building blocks depicted below: the single, unknotted strand (denoted by $I$), the positive and negative crossings (denoted $X$ and $X^-$), and the anticlockwise and clockwise full rotations (denoted $C$ and $C^-$).

\begin{equation}\label{eq:crossings_and_spinners}
\centre{
\labellist \small \hair 2pt
\pinlabel{$I$}  at 20 -78
\pinlabel{$X$}  at 350 -78
 \pinlabel{$ X^- $}  at 800 -78
 \pinlabel{$ C$}  at 1140 -78
  \pinlabel{$ C^- $}  at 1560 -78
\endlabellist
\centering
\includegraphics[width=0.6\textwidth]{figures/building_blockss}}
\end{equation}
\vspace{10pt}

\noindent One of the main goals of this paper is to  introduce a more suitable categorical framework for upwards tangles using rotational diagrams, based on Joyal-Street-Verity's traced monoidal categories.
\end{remark}

To finish this subsection, let us focus on knots. There are two quantities that we can associate to the rotational diagram $D$ of a given knot. On the one hand, the \textit{writhe} of $D$ is defined as
\begin{equation*}
w(D) := \sum_c \mathrm{sign}(c)
\end{equation*}
where the sum is taken over all crossings $c$  of $D$. This value is exactly the framing of the knot. On the other hand, the \textit{rotation number} $\rot(D)$ of  diagram $D$ is the number of positive full rotations $C$ minus the number of negative full rotations $C^-$ that appear in the diagram (note however that this is \textit{not} an isotopy invariant of the knot). For a general upwards tangle, we can similarly talk of the rotation number of a given strand. This integer is in fact determined by the crossings of the diagram:

\begin{lemma}\label{lemma:rot_number}
Let $D$ be a rotational diagram of a knot. Then we have $$\rot (D) =  \sum_{ \substack{c \\ \mathrm{under \ first}}}    \sign (c) - \sum_{ \substack{c \\ \mathrm{over \ first}}}   \sign (c) ,$$
where the first (resp. second) summation is taken over all crossings $c$ in $D$ where the knot traverses the understrand (resp. overstrand) in the first place.
\end{lemma}
\begin{proof}
We claim that  $$\rot (D) = \#  \left( \parbox[c]{8em}{\centering
 {\small  crossings with bottom right first }}
            \right) - \#  \left( \parbox[c]{8em}{\centering
 {\small  crossings with bottom left first }}
            \right),$$
where each of the summands refers to the number of crossings (regardless the sign) that the knot hits first by the bottom right or left endpoint. The equality of the statement follows from the claim as 
$$
\#  \left( \parbox[c]{8em}{\centering
 {\small  crossings with bottom right first }}
            \right) =  \sum_{ \substack{c \mathrm{ \ positive} \\ \mathrm{under \ first}}}    \sign (c) - \sum_{ \substack{c \mathrm{ \ negative} \\ \mathrm{over \ first}}}    \sign (c)  $$
and
$$
\#  \left( \parbox[c]{8em}{\centering
 {\small  crossings with bottom left first }}
            \right) =  \sum_{ \substack{c \mathrm{ \ positive} \\ \mathrm{over \ first}}}    \sign (c) - \sum_{ \substack{c \mathrm{ \ negative} \\ \mathrm{under \ first}}}    \sign (c).
$$

Let us now prove the claim. The first observation is that the all rotational Reidemeister moves  except the $(\Omega 1f)$ move on the left in \eqref{eq:R1_andR2a} preserve both the rotation diagram and the difference $d(D):=$ \#(bottom right first) $-$  \#(bottom left first). Using these moves, we can isotope a rotational knot diagram into a braid closure, where the sign of every original rotation is unchanged, so that the closure is taken both with positive and negative rotations, let us say $n$ and $m$ respectively.  We can then turn every $C^-$ into $C^+$, as follows:
\begin{equation*}
\centre{
\labellist \small \hair 2pt
\pinlabel{$ b$}  at 527 277
\pinlabel{\tiny{$ \cdots$}}  at 100 441
\pinlabel{\tiny{$ \cdots$}}  at 325 441
\pinlabel{\tiny{$ \cdots$}}  at 731 441
\pinlabel{\tiny{$ \cdots$}}  at 990 441
\endlabellist
\centering
\includegraphics[width=0.30\textwidth]{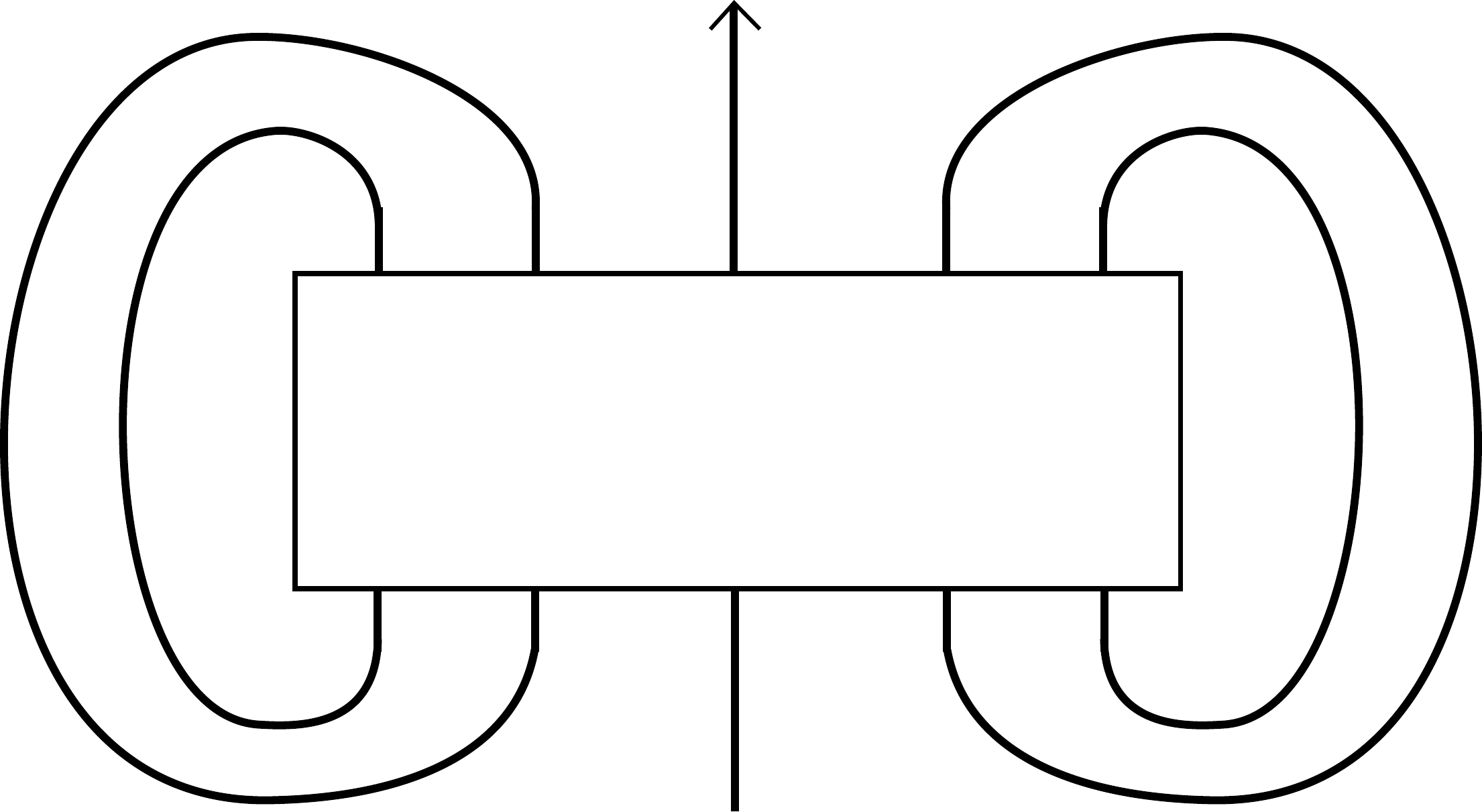}}  = 
\centre{
\labellist \small \hair 2pt
\pinlabel{$ b$}  at 672 483
\pinlabel{\tiny{$ \cdots$}}  at 902 623
\pinlabel{\tiny{$ \cdots$}}  at 500 623
\endlabellist
\centering
\includegraphics[width=0.35\textwidth]{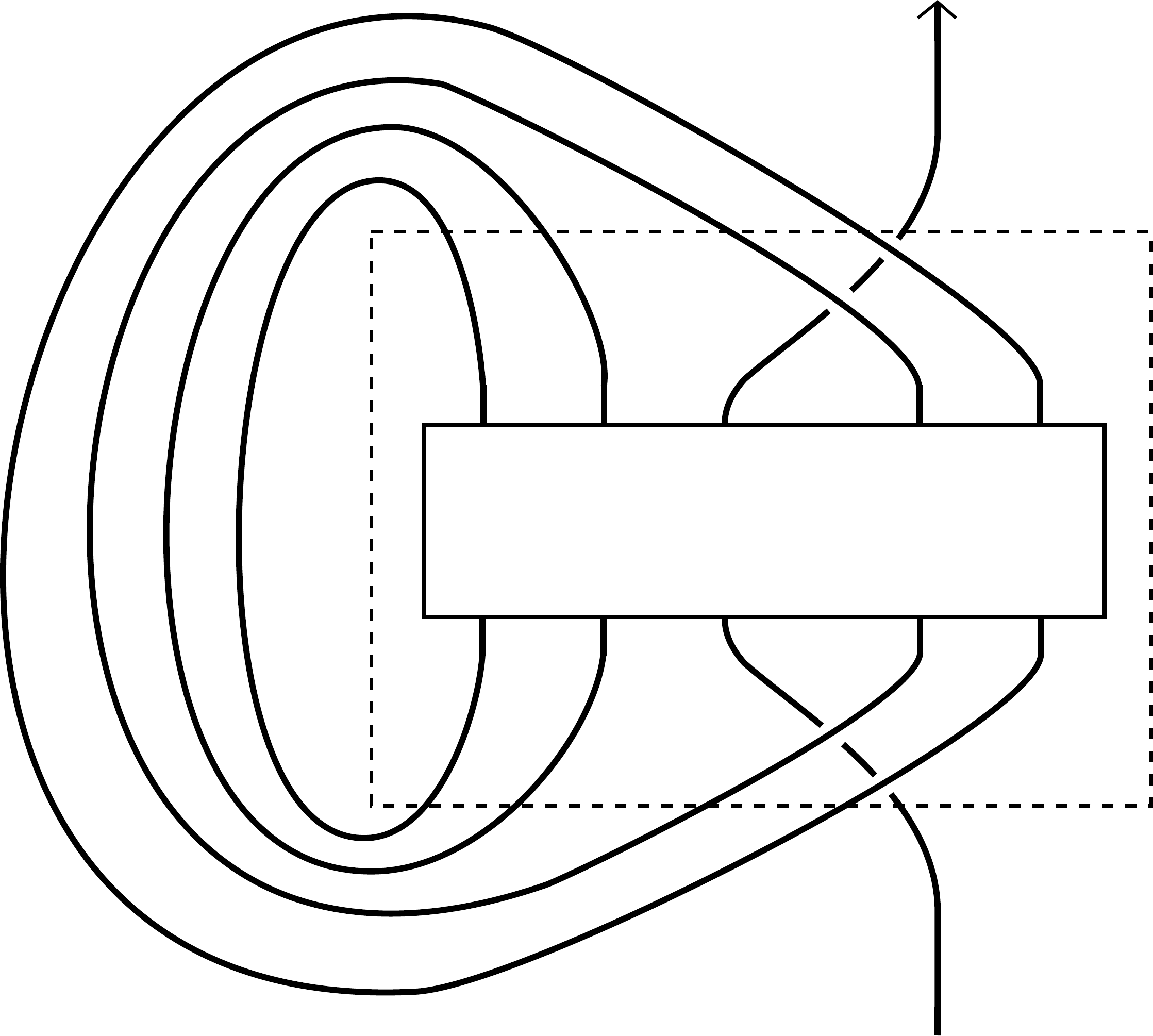}}
\end{equation*}
Note that the rotation number is increased by $2m$, and so is the difference $d$. Therefore, it suffices to show that $\rot (D)=d(D)$ for $D= \mathrm{cl}(b)$ a braid closure only with positive rotations, as depicted below:
\begin{equation*}
\centre{
\labellist \small \hair 2pt
\pinlabel{$ b$}  at 420 310
\pinlabel{\tiny{$ \cdots$}}  at 80 441
\pinlabel{\tiny{$ \cdots$}}  at 325 441
\endlabellist
\centering
\includegraphics[width=0.20\textwidth]{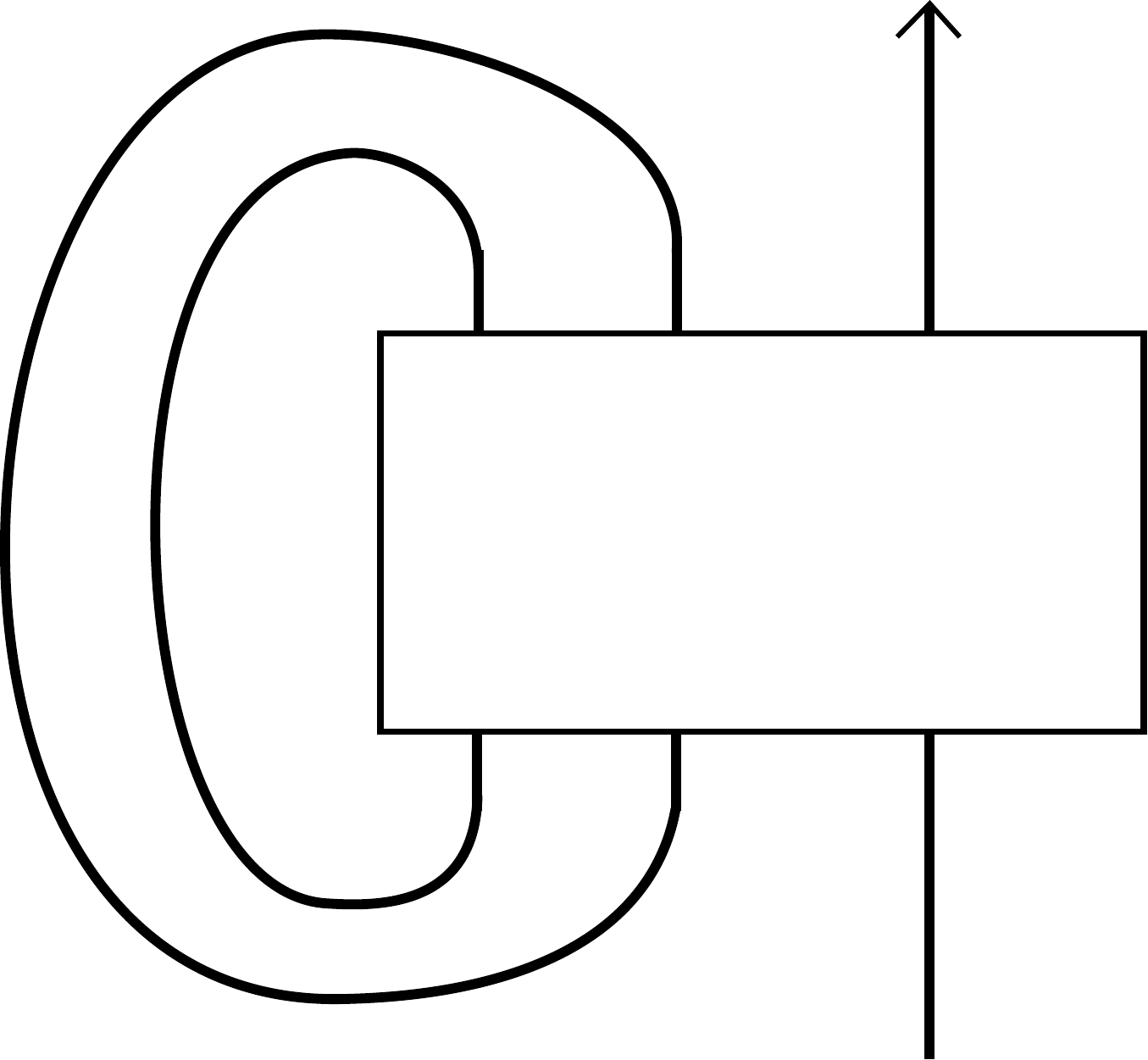}} 
\end{equation*}
Now, recall that the closure of a braid $b \in B_n$ is a knot precisely when the induced permutation $\sigma_b \in \mathfrak{S}_n$ is a cycle of length $n$. Let $c \in B_n$ be a braid such that $\sigma_c \sigma_b \sigma_c^{-1} = (1, \ldots , n)$ (since the conjugacy classes of $\mathfrak{S}_n$ are exactly given by the cycle type, this is always possible). In particular, $\mathrm{cl}(b)=\mathrm{cl}(cbc^{-1})$ by the Markov theorem. The key observation is that conjugation preserves the difference \#(bottom right first) $-$  \#(bottom left first), for it suffices to check it for the generators of the braid group, for which it is straightforward. What this means is that it is enough to check $\rot (D) = d(D)$ for $D= \mathrm{cl}(b)$ with $\sigma_b = (1, \ldots , n)$. Since the signs of the crossings play no role in the equality to demonstrate, we can freely change the signs in the braid. Therefore, we can replace $b$ by the following braid $b'$:
\begin{equation*}
\centre{
\labellist \small \hair 2pt
\pinlabel{$ b'$}  at 628 324
\pinlabel{\tiny{$ \cdots$}}  at 436 500
\endlabellist
\centering
\includegraphics[width=0.20\textwidth]{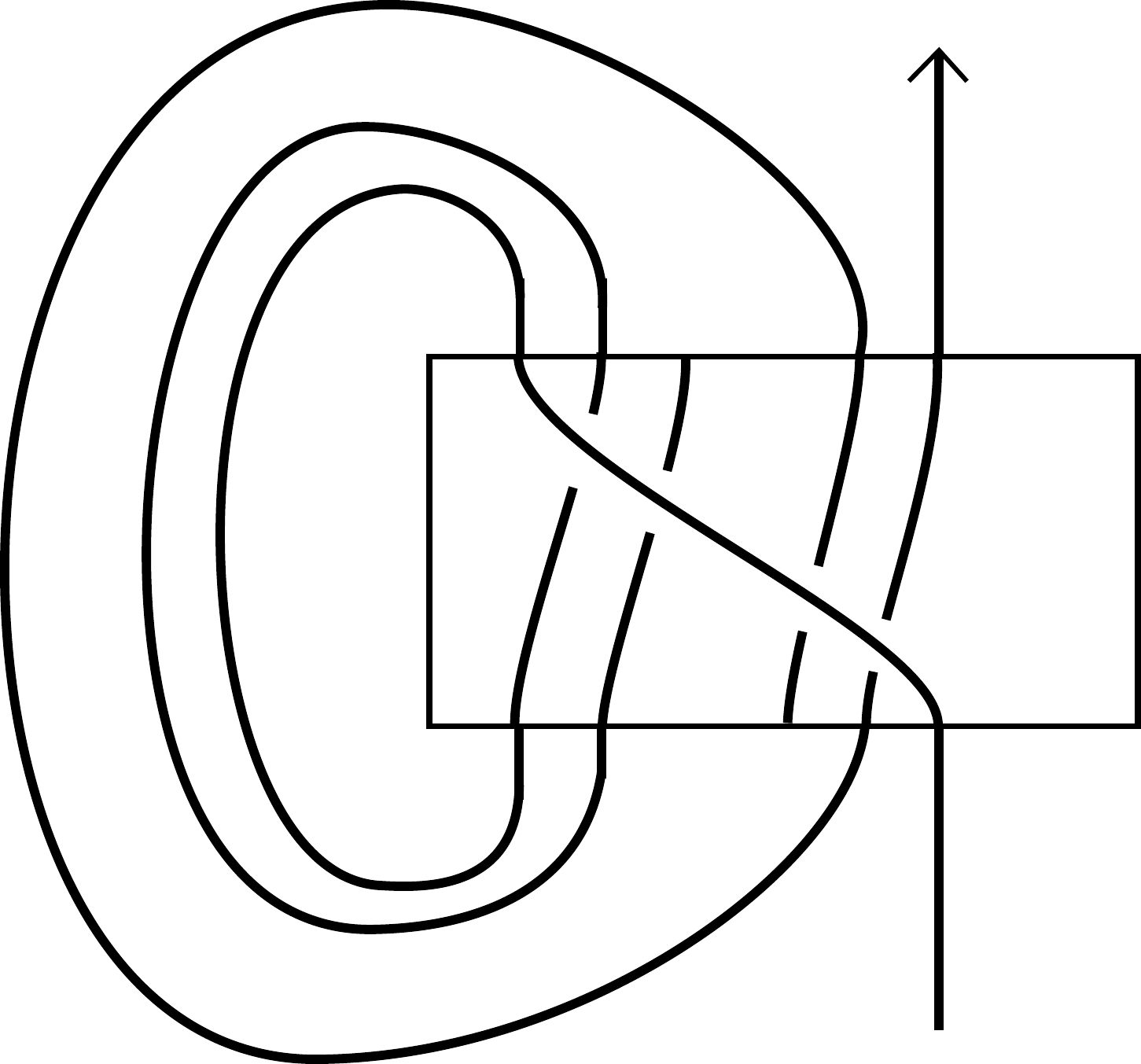}} 
\end{equation*}
For this digram, the equality  $\rot (D) = d(D)$  holds, and we conclude.
\end{proof}

\begin{corollary}\label{cor:rot+wt}
Let $D$ be a knot diagram. Then we have $$ \rot (D) + \mathrm{wr} (D) = 2  \cdot \sum_{ \substack{c \\ \mathrm{under \ first}}}    \sign (c) $$ (and in particular $\rot (D) + \mathrm{wr} (D)$ is always even).
\end{corollary}

\subsection{Universality of \texorpdfstring{$\Tup$}{Tup}}\label{subsec:universality_Tup}

Let us now describe a universal property that $\Tup$ satisfies, which essentially is a modification of Joyal-Street-Verity's notion of trace that we explained in \cref{subsec:T+}.

Now let $\C$ be a strict balanced monoidal category, and  consider $P:\C  \to \SS$ a strict monoidal functor preserving the balanced structure. Observe that for such a functor to exist, we must have $P(X)=P(Y)=n$ whenever there is an arrow $f: X \to Y$. Then the image $P(f)$ of a morphism $f$ as before  will be denoted by $\sigma_f \in \mathfrak{S}_n$.  Moreover, given objects $X, Y, U$ in $\C$ with $P(X)=P(Y)=n$ and $P(U)=m$, we will write  
\begin{equation}
 \mathrm{Hom}_\C^{ad} (X \otimes U, Y \otimes U) 
\end{equation}
for the set of morphisms $f: X \otimes U \to Y \otimes U$ with the property that the permutation $\sigma_f \in \mathfrak{S}_{n+m}$ contains no cycles of length $\leq m$ including only elements of the set $\{ n+1, \ldots , n+m \}$, and will call them \textit{$(X,Y,U)$-admissible}\index{admissible morphism}.

Let us introduce  now  the  main concept of this section: let $\C$ be a strict balanced category with monoidal product $ \otimes$, unit $\bm{1}$, braiding $\tau$ and twist $ \theta$.
An \textit{open trace}  for $\C$ is a pair $(P, \mathrm{tr})$  where $P: \W \to \mathfrak{S}$ is a strict monoidal functor preserving the balanced structure   and $\mathrm{tr}$  is a family of natural set-theoretical maps
\begin{equation}
\mathrm{tr}_{X,Y}^U : \mathrm{Hom}_\C^{ad} (X \otimes U, Y \otimes U) \to \hom{\C}{X}{Y}
\end{equation}
satisfying the following axioms:
\begin{enumerate}[leftmargin=4\parindent]
\item[(OTC1)] For every $f:X \to Y$,  $$\mathrm{tr}_{X,Y}^{\bm{1}}(f)=f.$$
\item[(OTC2)] For every $g: X \otimes U \otimes V  \to  Y \otimes U \otimes V  $ which is $(X,Y, U \otimes V)$-admissible,  $$\mathrm{tr}_{X,Y}^{U \otimes V}(g)= \mathrm{tr}_{X,Y}^U (\mathrm{tr}_{X \otimes U,Y \otimes U}^V(g)).$$
\item[(OTC3)] For every $f: X_1 \otimes U \to Y_1 \otimes U$ which is $(X_1, Y_1,U)$-admissible and every $g: X_2  \to Y_2 $,  
\begin{align*}
\mathrm{tr}_{X_1,Y_1}^U(f) \otimes g &= \mathrm{tr}_{X_1 \otimes X_2,Y_1 \otimes Y_2}^U \Big( (\id_{Y_1} \otimes \tau^{-1}_{Y_2,U})(f \otimes g)(\id_{X_1} \otimes\tau_{X_2,U})  \Big) \\
&= \mathrm{tr}_{X_1 \otimes X_2,Y_1 \otimes Y_2}^U \Big( (\id_{Y_1} \otimes \tau_{U,Y_2})(f \otimes g)(\id_{X_1} \otimes\tau^{-1}_{U,X_2})  \Big).
\end{align*}
\item[(OTC4)] For every object $U$, 
$$ \mathrm{tr}_{U,U}^U(\tau_{U,U}) = \theta_U \qquad , \qquad  \mathrm{tr}_{U,U}^U(\tau^{-1}_{U,U}) = \theta^{-1}_U.$$
\end{enumerate}

A balanced monoidal category $\C$ equipped with an open trace $(P, \mathrm{tr})$ is called a \textit{open-traced monoidal category}.

\begin{example}\label{ex:Tup_open_traced}
The category $\Tup$ of upwards tangles is an open traced monoidal category, as follows:  there is a canonical strict monoidal functor $P:\Tup \to \SS$ resulting from applying \cref{rmk:functor_M->C} to the family of monoid maps \eqref{eq:pi_Tup->S}. Note that indeed this functor preserves the balancing structure.

We can define a canonical open trace relative to this functor by closing up some components. More precisely, an $(+^n, +^n, +^m)$-admissible upwards tangle is an $(n+m)$-component upwards tangle $T$ such that its induced permutation $\sigma_T \in \SS_{n+m}$ contains no cycles of length $\leq m$ including only elements of the set $\{ n+1, \ldots , n+m \}$. Then define $\mathrm{tr}_{+^n, +^n}^{+^m}(T)$ as the tangle resulting from closing up the $m$ rightmost components of $T$, as depicted below:

\begin{equation}\label{eq:generic_closure}
\centre{
\labellist \small \hair 2pt
\pinlabel{$ T$}  at 216 294
\pinlabel{\footnotesize{$ \cdots$}}  at 116 441
\pinlabel{\footnotesize{$ \cdots$}}  at 325 441
\pinlabel{\footnotesize{$ n$}}  at 116 480
\pinlabel{\footnotesize{$ m$}}  at 325 480
\endlabellist
\centering
\includegraphics[width=0.25\textwidth]{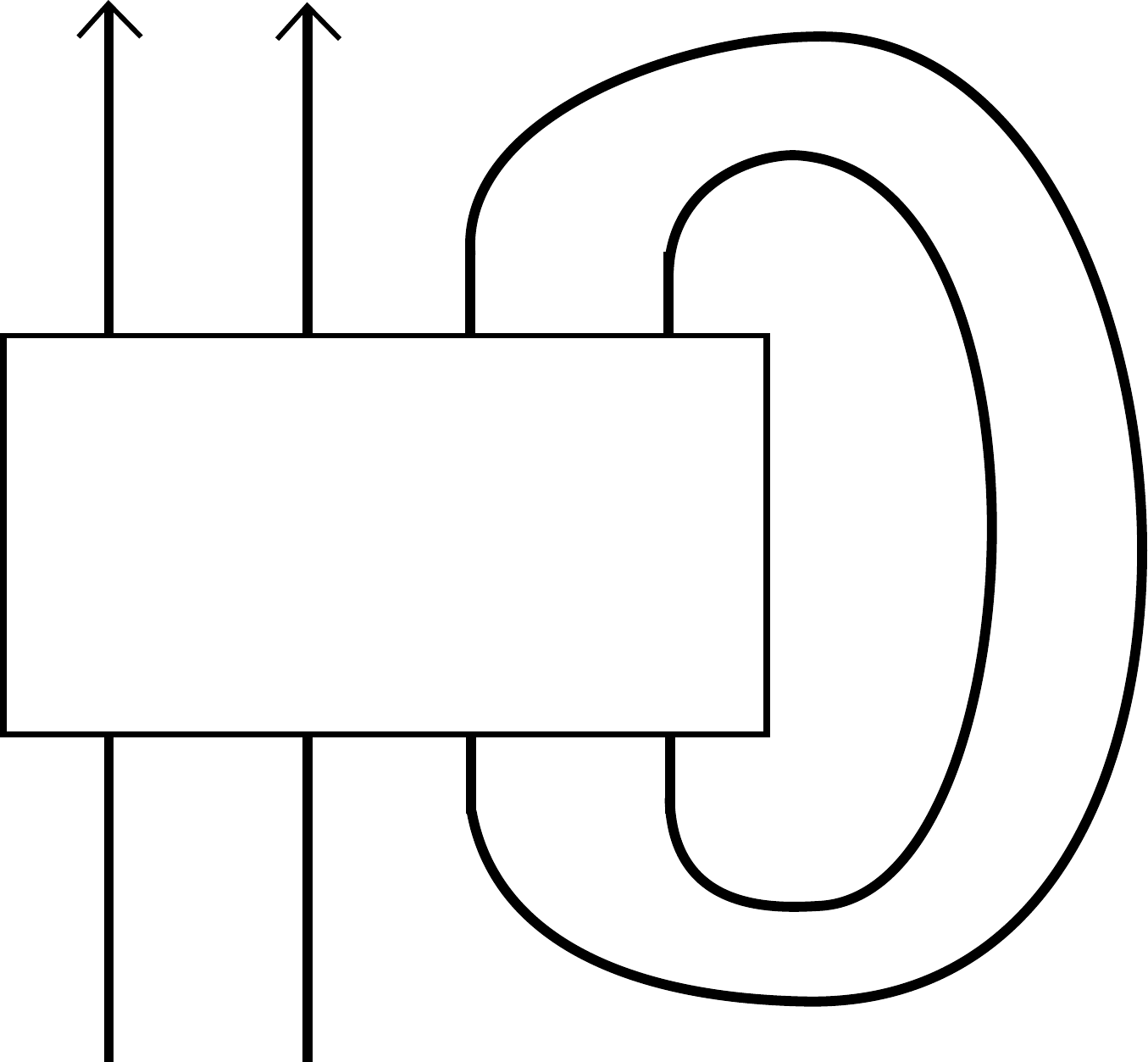}}
\end{equation}

Observe that the condition imposed on the induced permutation $\sigma_T$ guarantees that $\mathrm{tr}_{+^n, +^n}^{+^m}(T)$ contains no closed components and thence it is a morphism of $\Tup$. That the axioms hold is a consequence of the (framed) Reidemeister moves.
\end{example}

\begin{example}\label{ex:T+_open_traced}
Let us see how the category $\T^+$ inherits an open-trace structure from its trace. Consider the collection of monoid maps $$ \operatorname{End}_{\T^+}(+^n) \to \operatorname{End}_{\Tup}(+^n) \overset{\pi}{\to} \SS_n $$ where the first arrow is the canonical map that forgets the closed components and $\pi$ is as in \eqref{eq:pi_Tup->S}. Just as in \cref{ex:Tup_open_traced}, these maps assemble into a balanced functor $P: \T^+ \to \SS$. The restriction of the trace of $\T^+$ to admissible tangles trivially gives an open-trace structure.
\end{example}

We need one last observation before stating the universal property of $\Tup$.  Fix a positive integer $n >0$. For every $k \geq 0$, there is a group homomorphism $\SS_k \to \SS_{kn}$ that sends $\sigma \in \SS_k$ to the permutation that shuffles  $\{ 1, \ldots, kn  \}$ as blocks of $n$ elements according to $\sigma$. It is readily verified that these morphisms assemble into a monoidal functor $G_n: \SS \to \SS$ that preserves the balanced structure.

\begin{theorem}\label{thm:UP_opentraced}
Let $(\C, P, \mathrm{tr})$ be a strict open-traced  monoidal category, and let $X \in \C$.  Then there exists a unique strict monoidal functor $$F_X=F: \mathcal{T}^{\mathrm{up }} \to  \C$$ such that $F(+)=X$ and $F$ preserves the open-traced  structure, that is, $$F(\, \PC \,)= \tau_{X,X} \quad , \quad F(\begin{array}{c}
\centering
\includegraphics[scale=0.035]{figures/twist_gen} \end{array})= \theta_X  \quad , \quad F(\mathrm{tr}_{+^n , +^n}^{+^m}(f))=\mathrm{tr}_{X^{\otimes n}, X^{\otimes n}}^{X^{\otimes m}}(Ff) $$
whenever $f: +^{n+m} \to +^{n+m}$ is $(+^n, +^n, +^m)$-admissible, and the following diagram commutes,
$$
\begin{tikzcd}
\Tup \rar{F} \dar & \C \dar{P} \\
\mathfrak{S} \rar{G_{| X |}} & \mathfrak{S}
\end{tikzcd}
$$
where $| X |$ denotes the image of $X$ under the structure functor $P: \C \to \SS$.

In other words, $\Tup$ is the ``free strict open-traced monoidal category generated by one object''.
\end{theorem}
\begin{proof}
The strategy of the proof is the same as that for \cref{thm:turaev}, see e.g. \cite[\S I.3--4]{turaev}.  To start with, we need to work with a particular class of diagrams for upwards tangles. A \textit{generic rotational diagram}\index{generic rotational diagram} $D$ of an upwards tangle is a rotational diagram that results from a partial closure of a braid diagram as in \eqref{eq:generic_closure} (where $T$ in this case is required to be a braid diagram). 

First we claim that every upwards tangle has a generic rotational diagram. Indeed by \cref{lemma:every_tangle_has_rot_diag} every upwards tangle has a rotational diagram. By a sequence of rotational Reidemeister moves $(\Omega 2c)$ and $(\Omega 3)$, we can move all full rotations $C$ (resp. $C^-$ to the leftmost (resp. rightmost) part of the diagram so that we obtain a braid diagram that has been partially closed on the right and on the left.
\begin{equation*}
\centre{
\labellist \small \hair 2pt
\pinlabel{$ b$}  at 530 270
\pinlabel{\scriptsize{$ \cdots$}}  at 525 420
\pinlabel{\scriptsize{$ \cdots$}}   at 730 420
\pinlabel{\scriptsize{$ \cdots$}}   at 323 420
\endlabellist
\centering
\includegraphics[width=0.35\textwidth]{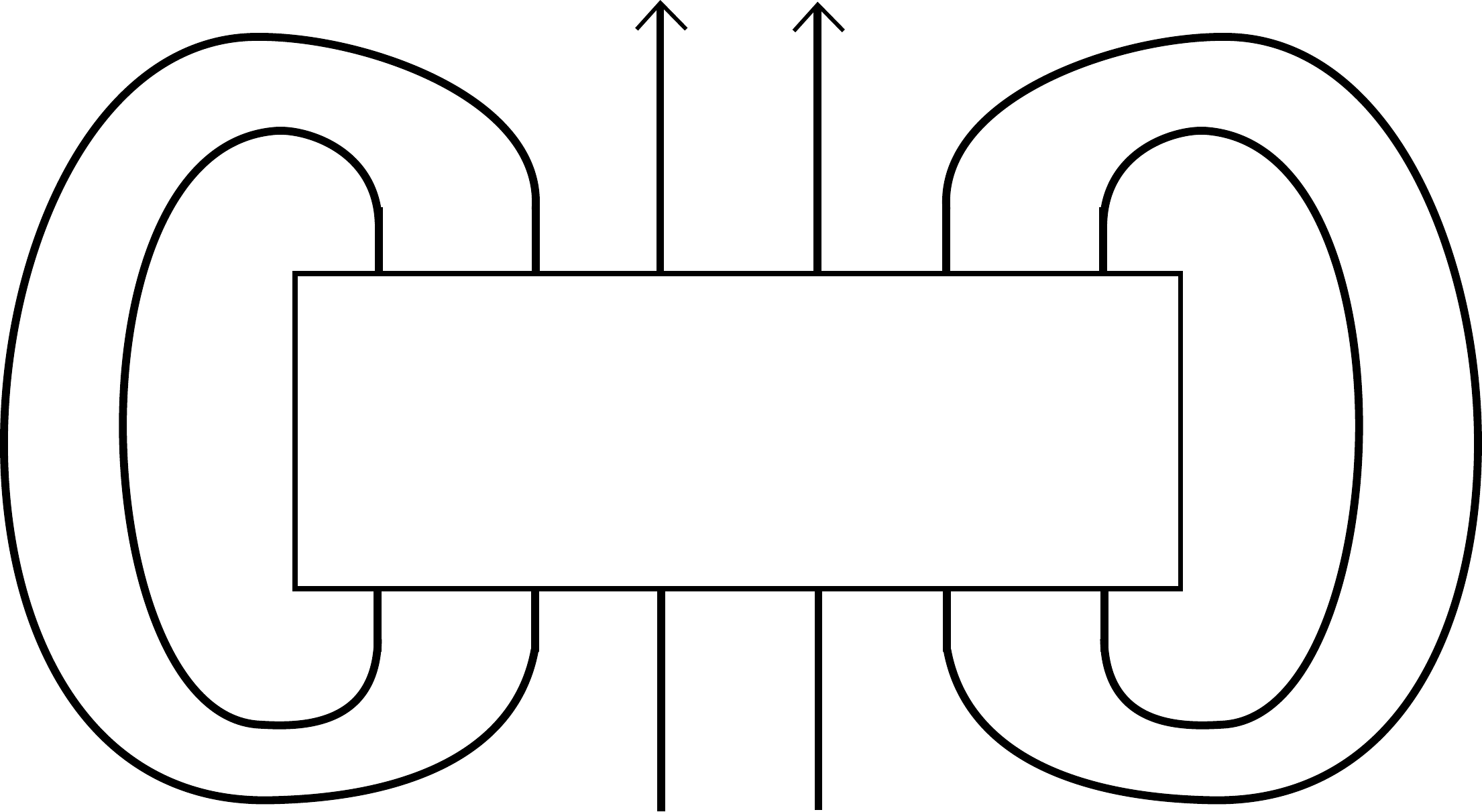}}
\end{equation*}
 By $(\Omega 1f)$ we can move all $C$ to the right-hand side as $C^-$ by adding some crossings on the top and bottom of the braid:
\begin{equation*}
\centre{
\labellist \small \hair 2pt
\pinlabel{$ b$}  at 360 620
\pinlabel{\scriptsize{$ \cdots$}}  at 564 744
\pinlabel{\scriptsize{$ \cdots$}}  at 352 744
\pinlabel{\scriptsize{$ \cdots$}}  at 152 744
\endlabellist
\centering
\includegraphics[width=0.37\textwidth]{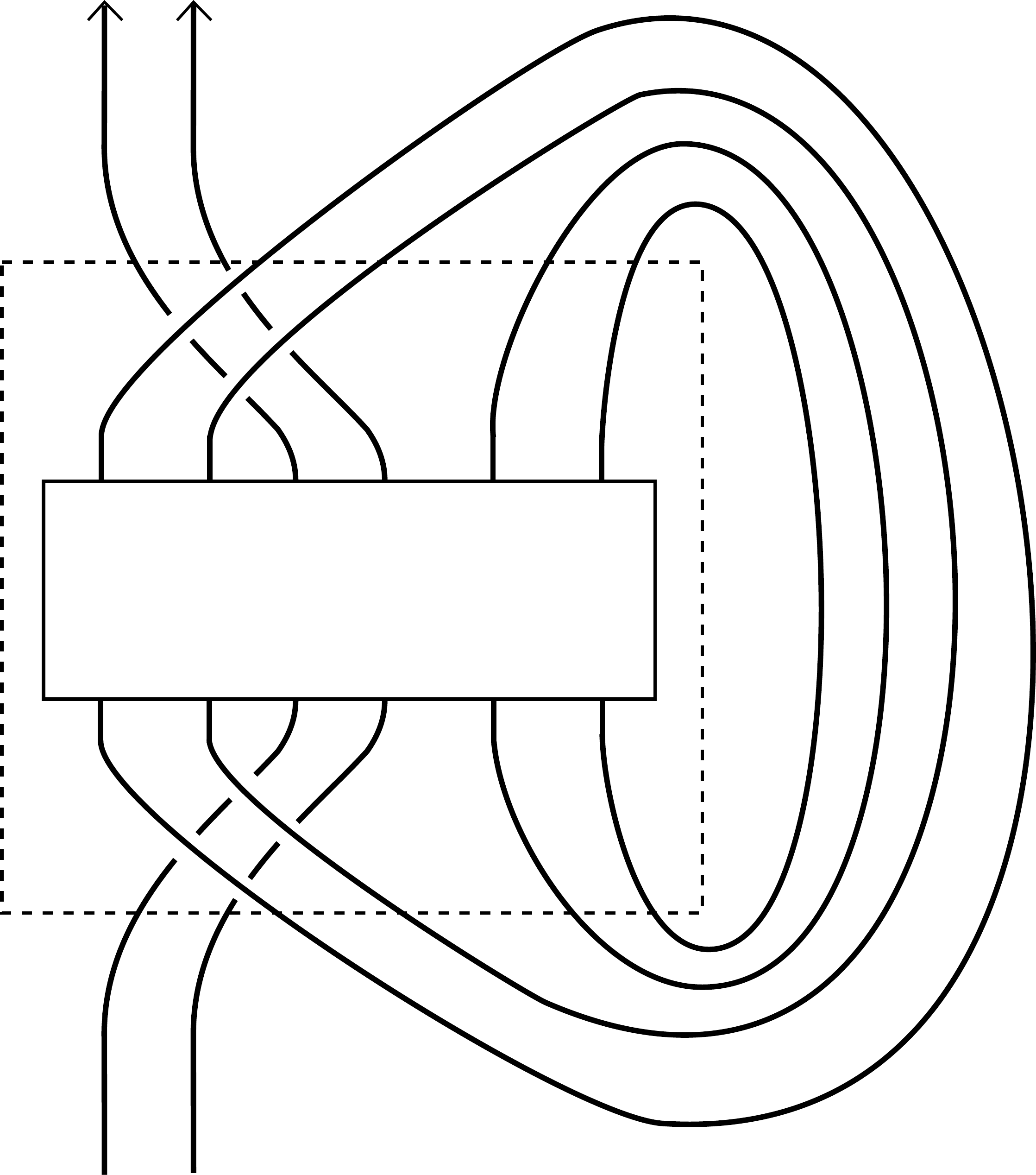}}
\end{equation*}
This concludes the claim.

In order to ensure the well-definedness of the functor $F_X$, we must understand how two braids with isotopic partial closures are related.  We view every braid as an upwards tangle by assigning the 0-framing to each of the strands (in other words, we consider the blackboard framing of any diagram).  If $b \in B_{n+m}$, we will write $\mathrm{cl}_r(m)$ for the tangle obtained by closing the $m$ rightmost strands of $b$, as in \eqref{eq:generic_closure}.

Now, we will define two different transformations of braids. Fix a positive integer $k \geq 0$. Define the \textit{$k$-open, framed Markov moves} as
\begin{enumerate}[leftmargin=3\parindent]
\item[(MI)] Conjugating a braid $b \in B_{k+n}$ by an element of the form $1_k \otimes a$, where $1_k \in B_k$ is the unit, $a \in B_{n}$ and $\otimes$ denotes concatenation of braids:
$$ b \longleftrightarrow   (1_k \otimes a)b(1_k \otimes a^{-1}),$$
\item[(MII)] Replacing $b \in B_n$ by $(\sigma_{n}^{-1} \sigma_{n+1})^{\pm 1}b \in B_{n+2}$, or the inverse of this:
$$ b \longleftrightarrow ( \sigma_{n}^{-1} \sigma_{n+1})^{\pm 1}b .$$
\end{enumerate}

We claim that that if $b \in B_{k+n}$ and $b' \in B_{k+m}$, then the  tangles with $k \geq 0$ open strands  $\mathrm{cl}_{n}(b)$ and  $\mathrm{cl}_{m}(b')$ are isotopic if and only if $b$ and $b'$ are related by a finite sequence of  $k$-open, framed Markov moves. Of course, this is simply a suitable version of Markov theorem \cite{birman} (in the framed case). We can argue making use of the $L$-moves devised by Lambropoulou and Rourke \cite{L,LR}, which are two moves equivalent to the classical Markov moves for links (in the unframed case). In the unframed, $k$-open situation, one can mimic the argument in \cite{LR} to see that the partial closures of two given braids are isotopic if and only if the braids are related by a sequence of $L$-moves where in each of the moves the new component is always closed up. Now the same argument the relates the $L$-moves with the Markov moves in \cite{LR} for the full closure amounts to the move (MI) above and the unframed version of (MII), $b \longleftrightarrow \sigma_{n}^{\pm 1} b $. We simply have to correct the Reidemeister one move to have the framed version in the $k$-open case.

Let us now prove the statement.    According to the claim above, any upwards tangle can be expressed as the partial closure of a braid. Since the functor $F$ must preserve the braiding and the open trace, the uniqueness follows.

Now let us see the existence. Define $F(+^n):= X^{\otimes n}$ on objects. On morphisms, given an upwards tangle $T$ with $k >0$ open components, consider a generic rotational digram $D$ of $T$ of the form $D=\mathrm{cl}_n (b)$ for a braid (diagram) $b \in B_{k+n}$, which is $(k,k,n) $-admissible by construction. Note that $F(b)$ is determined by the condition $F(\, \PC \,)= \tau_{X,X} $ of the statement. Then define $$F(T) := \mathrm{tr}_{X^{\otimes k}, X^{\otimes k}}^{X^{\otimes n}}(Fb).$$
To demonstrate the well-definedness of this, it suffices to see that the definition is invariant under the $k$-open, framed Markov moves stated above.  The invariance under (MI) follows by the naturality of $$\mathrm{tr}_{X,Y}^U : \mathrm{Hom}_\C^{ad} (X \otimes U, Y \otimes U) \to \hom{\C}{X}{Y}$$ in $U$. The invariance under (MII)  is a consequence of axiom (OTC4). Of course, that $Fb$ is well-defined on isotopy classes of braids is due to $\C$ being braided.
\end{proof}

\begin{example}
If $\T^+$ is endowed with the open-traced structure from \cref{ex:T+_open_traced}, then the unique structure-preserving functor $F:\Tup \to \T^+$ such that $F(+)=+$ is precisely the canonical embedding $\Tup \hooklongrightarrow \T^+$.
\end{example}

\begin{example}
Let $n \geq 1$. The unique structure-preserving self-functor $F:\Tup \to \Tup$ such that $F(+)=+^n$  replaces every component of a given upwards tangle by $n$ parallel copies of it.
\end{example}

\begin{construction}\label{constr:open_tr_from_tr}
Let $\C$ be a traced monoidal category, and let $X$ be an object of $\C$ such that all tensor powers of $X$ are different objects. Let us see how we can obtain a monoidal subcategory of $\C$ containing $X$ which inherits an open trace from $\C$.

Consider the composite (that we will denote by $F$) $$ \Tup \hooklongrightarrow \T^+ \to \C  $$ where the first functor is the canonical embedding and the second is the unique strict monoidal functor sending $+$ to $X$ and preserving the traced structure. Let us write $\C_X$ for the \textit{image} of $F$, that is $\C_X$ is the subcategory of $\C$ whose objects are $X^{\otimes n}$, $n \geq 0$, and whose arrows are images $F(T)$ of arbitrary upwards tangles. Note that this category is well-defined because by hypothesis $F$ is injective-on-objects. Moreover $\C_X$ trivially inherits the monoidal product from $\C$ and a balanced structure since $F$ is in particular balancing-preserving.

Furthermore there is a balancing-preserving functor $P: \C_X \to \SS$ given by $P(X^{\otimes n})=n$ and $P(F(T))= \sigma_T$. With respect to this functor, the trace of $\C$ restricts to an open trace in $\C_X$ given by $ \mathrm{tr}_{X^{\otimes n}, X^{\otimes n}}^{X^{\otimes m}}(F(T)) : = F(\mathrm{tr}_{+^n , +^n}^{+^m}(T))$ where $T$ is an $(+^n,+^n, +^m)$-admissible upwards tangle. Of course, the resulting functor $\Tup \to \C_X$ is open-traced.
\end{construction}

If $A$ is a ribbon Hopf algebra over some (non-trivial) commutative ring $\Bbbk$, recall from \cref{subsec:ribbon_algebras} that its representation theory gives rise to a strict ribbon category $\mathsf{fMod}_A^{\mathrm{str}}$ . Recall also that every ribbon category is canonically traced using \eqref{eq:trace_for_ribbon}. If $W$ is a finite-free $A$-module, applying \cref{constr:open_tr_from_tr} to $\mathsf{fMod}_A^{\mathrm{str}}$, we get a open-traced monoidal category $(\mathsf{fMod}_A^{\mathrm{str}})_W$. Please note that we can indeed apply the construction because in the strictification of a category all tensor powers of an object have different lengths and therefore are different objects.  The unique strict monoidal functor $\Tup \to (\mathsf{fMod}_A^{\mathrm{str}})_W$ arising from \cref{thm:UP_opentraced} will be also called the \textit{Reshetikhin-Turaev functor} and will be denoted as
%If $W$ is a finite-free $A$-module of rank at least 2, then all powers of $W$ must be different  because of the so-called \textit{invariant basis number property}, that (non-trivial) commutative rings satisfy.  Applying \cref{constr:open_tr_from_tr} to $\mathsf{fMod}_A^{\mathrm{str}}$, we get a open-traced monoidal category $(\mathsf{fMod}_A^{\mathrm{str}})_W$. The unique strict monoidal functor $\Tup \to (\mathsf{fMod}_A^{\mathrm{str}})_W$ arising from \cref{thm:UP_opentraced} will be also called the \textit{Reshetikhin-Turaev functor} and will be denoted as
\begin{equation}\label{eq:RT_open_trace}
RT_W: \Tup \to (\mathsf{fMod}_A^{\mathrm{str}})_W 
\end{equation}
The following lemma is immediate and justifies the name choice:

\begin{lemma}\label{lemma:restriction}
In the above setup, the following diagram commutes,
$$
\begin{tikzcd}
\Tup \rar{RT_W} \dar[hook] & (\mathsf{fMod}_A^{\mathrm{str}})_W \dar[hook]  \\
\T \rar{RT_W} & \mathsf{fMod}_A^{\mathrm{str}}
\end{tikzcd}
$$
where the horizontal top functor is \eqref{eq:RT_open_trace}, the horizontal bottom functor arises from \cref{thm:turaev}, and the two vertical functors are the inclusions.
\end{lemma}

\begin{remark}
If $ \C$ is a strict traced monoidal category, and a balancing-preserving functor $P:\C \to \SS$ exists, then the trace of $\C$ induces an open trace $(P, \mathrm{tr})$ on $\C$ by restricting to admissible morphisms. 

The reader should be mindful that, in general, a traced monoidal category might not be open-traced, as the existence of a balancing-preserving monoidal functor $\C \to \SS$  --which is an essential part of the structure-- is not guaranteed in general. That can be the case even when the  object monoid of $\C$ is free in one object, e.g. $\C$ being the  subcategory of the category of topological spaces and continuous maps on the objects $\R^n$, $n \geq 0$, and arrows arrows self-maps $\R^n \to \R^n$, which is balanced with respect to the Cartesian product and trivial twist. There is no sensible way to functorially assign  a permutation to every such self-map preserving the symmetric braiding.
\end{remark}

\section{The universal tangle invariant}\label{sec:universal_inv}

The aim of this section is to produce, in a canonical way, a strict monoidal functor that encodes the so-called \textit{universal invariant} defined by Lawrence \cite{lawrence}. By ``canonical'' we mean that arises from a universal property, just like the Reshetikhin-Turaev invariants arise from \cref{thm:turaev}.

The universal invariant is usually defined from a ribbon Hopf algebra $A$ and it is known to dominate the family of  Reshetikhin-Turaev invariants built out of the ribbon category $\mathsf{fMod}_A$ (hence the name). However, the comultiplication, counit and antipode do not play a role when defining the universal invariant. Therefore we want to consider a version  here that uses the minimal algebraic data needed to build such an invariant.

%INTRO. I want a more general structure than ribbon Hopf. Somehow this is the minimal algebraic structure required to produce a isotopy invariant multiplying beads \todo{say something}

\subsection{XC-algebras}

Let $\Bbbk$ be a commutative ring with unit, and let $A=(A, \mu,  1)$ be a $\Bbbk$-algebra. An \textit{XC-structure} on $A$ is a choice of
% An \textit{XC-algebra}\index{XC-algebra} over $k$ is a $k$-algebra $A=(A, \mu,  1)$ together with 
two preferred, invertible elements 
$$ R \in A \otimes A \qquad , \qquad \kappa \in A,  $$
called the \textit{universal $R$-matrix}\index{universal $R$-matrix} and the \textit{balancing element}\index{balancing element}, respectively, satisfying the following axioms:

\begin{enumerate}[leftmargin=4\parindent, itemsep=2mm]
\item[(XC0)] $R^{\pm 1}=(\kappa \otimes \kappa) \cdot R^{\pm 1} \cdot (\kappa^{-1} \otimes \kappa^{-1})$,
\item[(XC1f)] $\mu^{[3]}(R_{31}\cdot  \kappa_2 )=\mu^{[3]}(R_{13}\cdot  \kappa^{-1}_2) $,
\item[(XC2c)] $  1\otimes \kappa^{-1} = (\mu\otimes\mu^{[3]} )(R_{15}\cdot R_{23}^{-1}\cdot \kappa^{-1}_4 )$,
\item[(XC2d)] $\kappa \otimes 1 = (\mu^{[3]}\otimes\mu )(R_{15}^{-1}\cdot R_{34}\cdot \kappa_2 )$,
\item[(XC3)] $R_{12}R_{13}R_{23}=R_{23}R_{13}R_{12}$,
\end{enumerate}
where  we have written $\mu^{[3]}$ for the 3-fold multiplication map and for $1 \leq i,j \leq n$, $i \neq j$ we have put
\begin{equation}\label{eq:R_ij}
 R_{ij}^{\pm 1}:= \begin{cases}  (1^{\otimes i-1} \otimes \id \otimes 1^{\otimes j-i-1} \otimes \id \otimes 1^{n-j})(R^{\pm 1}), & i<j\\
(1^{\otimes j-1} \otimes \id \otimes 1^{\otimes i-j-1} \otimes \id \otimes 1^{n-i})(\mathrm{flip}_{A,A}(R^{\pm 1})), & i>j
\end{cases} 
\end{equation}
and similarly $\kappa^{\pm 1}_i = (1^{\otimes i-1} \otimes \id \otimes  1^{\otimes n-i})(\kappa^{\pm 1})$.

The element $\nu := \mu^{[3]}(R_{13}\cdot  \kappa^{-1}_2)$ from (XC1f) is called the \textit{inverse of the classical ribbon element}\index{inverse of the classical ribbon element}. The triple $(A, R, \kappa)$ consisting of  a $\Bbbk$-algebra and an XC-structure will be called an \textit{XC-algebra}. Originally this notion appeared in \cite{barnatanveenpolytime} under the name of ``snarl algebra''.

\begin{example}
For any $\Bbbk$-algebra $A$, setting $R=1 \otimes 1$ and $\kappa=1$ gives a trivial XC-algebra structure. 
\end{example}

\begin{example}\label{ex:Dilbert}
Let $A=\mathcal{M}_2(\mathbb{C})$ be the algebra of $2 \times 2$ matrices with complex coefficients with the usual matrix multiplication. Then the elements
$$R=\begin{pmatrix}
1 & 0 \\0 & 0
\end{pmatrix}\otimes \begin{pmatrix}
1 & 0 \\0 & 1
\end{pmatrix} +  \begin{pmatrix}
0 & 0 \\0 & 1
\end{pmatrix}\otimes \begin{pmatrix}
-1 & 0 \\0 & 1
\end{pmatrix} +2 \begin{pmatrix}
0 & 1 \\0 & 0
\end{pmatrix} \otimes \begin{pmatrix}
0 & 0 \\1 & 0
\end{pmatrix} $$ 
and $$ \kappa = {\bf i} \begin{pmatrix}
1 & 0\\0 & -1
\end{pmatrix} $$
determine an   XC-algebra structure, where ${\bf i}$ represents the complex imaginary unit. Indeed it is readily verified that $R^{-1}=R$ and $\kappa^{-1}=-\kappa$, and the axioms (XC0)--(XC3) can be checked performing the corresponding matrix multiplications.
\end{example}

\begin{example}
The previous example can be further generalised to a one-parameter family of XC-algebras. Let $A=\mathcal{M}_2(\mathbb{C})$ as before. Then for every choice of $ \lambda \in  \mathbb{C}$,
$$R_\lambda = \begin{pmatrix}
1 & 0 \\0 & 0
\end{pmatrix} \otimes \begin{pmatrix}
1 & 0 \\0 & \lambda
\end{pmatrix} +  \begin{pmatrix}
0 & 0 \\0 & 1
\end{pmatrix}\otimes \begin{pmatrix}
-\lambda^{-1} & 0 \\0 & 1
\end{pmatrix} +2 \begin{pmatrix}
0 & 1 \\0 & 0
\end{pmatrix} \otimes \begin{pmatrix}
0 & 0 \\1 & 0
\end{pmatrix} 
$$
and $\kappa$ as before define an XC-algebra structure on $A$. Again this is  a direct computation which we omit here. In this case, $R_{\lambda}^{-1} = R_{\lambda^{-1}}$.
\end{example}

In general, we have two main sources of examples of XC-algebras. The notation clash with ribbon Hopf algebras is of course not incidental, and these provide the first of them:

\begin{proposition}\label{prop:ribbon->XC}
Every ribbon Hopf algebra $(A,R, \kappa)$ is an XC-algebra.
\end{proposition}
\begin{proof}
For (XC0) we have
$$ R= (S^2 \otimes S^2)(R)= (\kappa \otimes \kappa)\cdot R \cdot(\kappa^{-1}\otimes \kappa^{-1})$$ using \eqref{eq:(S otimes id)R} and similarly
$$R^{-1}= (S^2 \otimes S^2)(R^{-1})= (\kappa \otimes \kappa)\cdot R^{-1} \cdot(\kappa^{-1}\otimes \kappa^{-1}).$$ 
For (XC1f), first we note that
$$ v^{-1}= u^{-1}\kappa = \sum_i \beta_i S^2(\alpha_i) \kappa = \sum_i \beta \kappa \alpha_i. $$ Since $S(v)=v$ by \eqref{eq:ribbon_elmt_axioms}, then $S(v^{-1})=v^{-1}$ and therefore 
\begin{align*}
\mu^{[3]}(R_{31}\cdot  \kappa_2 ) &= \sum_i \beta_i \kappa \alpha_i = S \left( \sum_i \beta_i \kappa \alpha_i \right) = \sum_i S(\alpha_i) S(\kappa) S(\beta_i) \\ & =\sum_i \alpha_i \kappa^{-1} \beta_i = \mu^{[3]}(R_{13}\cdot  \kappa^{-1}_2)
\end{align*}
To obtain (XC2c) we compute
\begin{align*}
(\mu\otimes\mu^{[3]} )(R_{15}\cdot R_{23}^{-1}\cdot \kappa^{-1}_4 )&= \sum_{i,j}  \alpha_i \bar{\alpha}_j  \otimes \bar{\beta}_j \kappa^{-1} \beta_i \\
&= (1\otimes \kappa^{-1} ) \cdot\left( \sum_{i,j} \alpha_i \bar{\alpha}_j  \otimes S^2(\bar{\beta}_j)  \beta_i \right)\\
&= (1\otimes \kappa^{-1} ) \cdot\left(  \sum_{i,j} \alpha_i \bar{\alpha}_j  \otimes S(S^{-1}(\beta_i )\cdot S(\bar{\beta}_j)  )\right) \\
&= (1\otimes \kappa^{-1} ) \cdot (\id \otimes S) \left(  \sum_{i,j} \alpha_i \bar{\alpha}_j  \otimes S^{-1}(\beta_i )\cdot S(\bar{\beta}_j)  \right) \\
&= (1\otimes \kappa^{-1} ) \cdot(\id  \otimes S ) \big( (\id \otimes S^{-1}  )(R) \cdot ( \id \otimes S )(R^{-1}) \big)\\
&=(1\otimes \kappa^{-1} ) \cdot (\id  \otimes S )(R^{-1}\cdot R)\\
&=(1\otimes \kappa^{-1} ) \cdot (\id  \otimes S )(1 \otimes 1)\\
&= (1\otimes \kappa^{-1} ) .
\end{align*}
(XC2d) follows from a similar computation:
\begin{align*}
(\mu^{[3]}\otimes\mu )(R_{15}^{-1}\cdot R_{34}\cdot \kappa_2 )&= \sum_{i,j}  \bar{\alpha}_j \kappa \alpha_i\otimes \beta_i \bar{\beta}_j\\
&= \left( \sum_{i,j} \bar{\alpha}_j S^2(\alpha_i)\otimes \beta_i \bar{\beta}_j \right) \cdot(\kappa \otimes 1) \\
&= \left(  \sum_{i,j} S(S(\alpha_i) S^{-1}(\bar{\alpha}_j )) \otimes \beta_i \bar{\beta}_j  \right)(\kappa \otimes 1)\cdot \\
&= (S \otimes\id  ) \left(  \sum_{i,j} S(\alpha_i) S^{-1}(\bar{\alpha}_j ) \otimes \beta_i \bar{\beta}_j  \right) \cdot(\kappa \otimes 1)\\
&= (S \otimes\id  ) \big( (S\otimes \id  )(R) \cdot (S^{-1}\otimes \id )(R^{-1}) \big)\cdot(\kappa \otimes 1) \\
&= (S \otimes\id )(R^{-1}\cdot R)\cdot(\kappa \otimes 1) \\
&= (S\otimes\id   )(1 \otimes 1)\cdot(\kappa \otimes 1) \\
&= (\kappa \otimes 1) .
\end{align*}
Lastly,  (XC3) is precisely the Yang-Baxter equation \eqref{eq:YB}. 
\end{proof}

\begin{remark}
According to the proof, when $(A, R, \kappa)$ is ribbon, then $\nu = v^{-1}$, which justifies our name ``inverse of the classical ribbon element'' for $\nu$.
\end{remark}

\begin{example}\label{ex:Sweedler}
Let $\Bbbk$ be a commutative ring with unit in which 2 is invertible. The \textit{Sweedler algebra} \index{Sweedler algebra} is the $\Bbbk$-algebra $SW$ given by the following presentation: $$SW := \frac{\Bbbk \langle s, w \rangle}{(s^2=1 , w^2=0, sw = -ws)}.$$ It is well-known  (see e.g. \cite{majid_foundations}) that the Sweedler algebra is a Hopf algebra with structure maps determined by the condition that $$ \Delta (w)=w \otimes 1 + s \otimes w \qquad , \qquad \varepsilon(w)=0 \qquad , \qquad S(w)=-sw  $$ and $s$ being grouplike. Furthermore,  $SW$ admits a ribbon structure (and therefore an XC-algebra structure) given by $$R:=1 \otimes 1  - 2 p \otimes p + w \otimes w + 2 wp \otimes wp - 2 w \otimes wp  \qquad , \qquad \kappa := s,$$ where for convenience we have written $p := (1-s)/2$. Note  that this ribbon structure is triangular, $R^{-1}= R_{21}$.
\end{example}

\begin{remark}
We remind the reader that a rich source of examples of ribbon Hopf algebras arises from the \textit{Drinfeld double} construction \cite{drinfeld_original}: if $A$ is a finite-free Hopf algebra, then there exists a unique quasi-triangular Hopf algebra structure on $D(A):= A \otimes A^{*}$ such that
\begin{enumerate}
\item The canonical map $A \otimes A^{*} \to D(A), \ x \otimes y \mapsto xy$ is a $\Bbbk$-module isomorphism,
\item The canonical embeddings $A \hooklongrightarrow D(A)$ and $A^{*,\mathrm{cop}} \hooklongrightarrow D(A)$ are Hopf algebra homomorphisms,
\item The pair $(D(A),R)$, where   $R \in A \otimes A^* \subset D(A) \otimes D(A)$ is the canonical element that corresponds with the identity of $A$, is a quasi-triangular Hopf algebra.
%\item $Z_{D(A,B)}$ gives a well-defined isotopy invariant of braids.
\end{enumerate}
The Drinfeld double $D(A)$ may or may not contain a balancing element $\kappa$, but even if it does not, it can always be formally adjoint  as follows \cite{RT,kassel}: if $(H,R)$ is a quasi-triangular Hopf algebra, then there exists a unique Hopf algebra structure on $H(\kappa):= H \oplus H \kappa$ such that
\begin{enumerate}
\item The canonical embedding $H  \hooklongrightarrow H(\kappa)$ is a Hopf algebra homomorphism,
\item The triple $( H(\kappa), R, \kappa)$ is a ribbon Hopf algebra,
\end{enumerate}
where we view $R \in H(\kappa) \otimes H(\kappa)$ via the canonical embedding $H \otimes H \hooklongrightarrow H(\kappa) \otimes H(\kappa)$.

The upshot is that if $A$ is a finite-free, rank $n$ Hopf algebra, then one can construct  a rank $2 n^2$ ribbon Hopf algebra $\mathcal{D}(A):= D(A)(\kappa)$.
\end{remark}

\begin{example}\label{ex:double_Sweedler}
We would like to give a concise example of the previous discussion. As before, let $\Bbbk$ be a commutative ring with unit where 2 is invertible, and consider the Sweedler algebra $SW$ from \cref{ex:Sweedler}. Then the enlarged Drinfeld double  $\mathcal{D}(SW)$ can be described explicitly as follows: $\mathcal{D}(SW)$ is the quotient of the free $\Bbbk$-algebra $ \Bbbk \langle s,w , \sigma , \omega , c \rangle$ modulo the two-sided ideal generated by the following relations:
\begin{itemize}
\item $s^2 =1 = \sigma^2 \quad , \quad w^2=0=\omega^2 \quad , \quad c^2=s \sigma \quad , \quad w \omega - \omega w = s - \sigma$,
\item $xy=yx$ whenever $x,y=c,s, \sigma$,
\item $xy= -yx$ whenever $x= c,s, \sigma$ and $y=w, \omega$.
\end{itemize}
Note that this is indeed a rank 32 algebra. The universal $R$-matrix and the balancing element are given by 
$$R:= \tfrac{1}{2}( 1 \otimes (1+\sigma) + s\otimes (1-\sigma) + w \otimes (\omega + \sigma \omega) + sw \otimes ( \omega - \sigma \omega)) \quad , \quad \kappa := c.$$ The rest of Hopf algebra structure maps are determined by those of $SW$, because the Sweedler algebra is self-dual.
\end{example}

\subsection{Traced XC-algebras}

Let $A$ be a $\Bbbk$-algebra, and let $[A,A]$ denote its commutator submodule, that is, the submodule of $A$ spanned by all commutators $[a,b]=ab-ba$ for $a,b \in A$. Then a \textit{trace} \index{trace} for $A$ is a $\Bbbk$-module homomorphism $$\mathrm{tr}: A \to \Bbbk $$ that vanishes on $[A,A]$, that is, that satisfies $\mathrm{tr}(ab)=\mathrm{tr}(ba)$ for all $a,b \in A$. It is clear that traces on $A$ are in bijection with $\Bbbk$-module homomorphisms $A/[A,A] \to \Bbbk$. A $\Bbbk$-algebra together with a fixed choice of trace is called a \textit{traced algebra}\index{traced algebra}. A \textit{traced XC-algebra} \index{traced XC-algebra} is an algebra endowed with both a trace and an XC-structure.

\begin{example}
Let $SW$ be the Sweedler algebra from \cref{ex:Sweedler}. Define $\mathrm{tr}: SW \to \Bbbk$ by $$ \mathrm{tr}(s)=1 \qquad , \qquad \mathrm{tr}(w)=0 = \mathrm{tr}(sw). $$ This map defines a trace and makes $SW$ into a traced XC-algebra.
\end{example}

\begin{example}
Let $A= \mathcal{M}_2(\mathbb{C})$ be the XC-algebra from \cref{ex:Dilbert}. Then the usual trace of a matrix, that is, the sum of its diagonal elements, is a trace, turning $A$ into a traced XC-algebra.
\end{example}

The previous example motivates the following definition: a $\Bbbk$-algebra $A$ is said to be an \textit{endomorphism algebra}  if $A= \mathrm{End}_\Bbbk (V)$ is the $\Bbbk$-algebra of endormorphisms of some finite-free $\Bbbk$-module $V$. We will always assume that $\mathrm{End}_\Bbbk (V)$ carries the usual algebra structure, that is, $f  \cdot g := f \circ g$ gives the product and  $\id_V$ is the unit element. Obviously, the XC-algebra from \cref{ex:Dilbert} is an endomorphism algebra under the identification $\mathcal{M}_2(\mathbb{C}) \cong \mathrm{End}_{\mathbb{C}}(\mathbb{C}^2)$.

Let us now pass to discuss the second class of examples of XC-algebras, which come from representation theory. Recall that, given a $\Bbbk$-algebra $A$, an \textit{$A$-module} $V$ can be viewed as a $\Bbbk$-algebra homomorphism $$\rho: A \to \mathrm{End}_\Bbbk(V).$$

The following observation is immediate and  provides a large family of examples.

\begin{lemma}\label{lemma:endom_alg_tr_XC}
Let $A$ be an XC-algebra with universal $R$-matrix $R$ and balancing element $\kappa$. Then every $A$-module $(V, \rho)$ gives rise to a traced XC-algebra structure on $\mathrm{End}_\Bbbk(V)$ with universal $R$-matrix and balancing element given by $(\rho \otimes \rho)(R)$ and $\rho(\kappa)$, respectively.
\end{lemma}

This implies that finite-free representations of ribbon Hopf algebras give rise to traced XC-algebras, in particular the representations of the Drinfeld-Jimbo quantisations $U_h(\mathfrak{g})$ of complex semisimple Lie algebras.

\subsection{The category of elements \texorpdfstring{$\mathcal{E}(A)$}{E(A)}}

We will now integrate XC-algebras in our monoidal categorical framework. Before that let us introduce some useful notation: if $\sigma \in \SS_n$ and $V$ is a $\Bbbk$-module, we write $\sigma_* : V^{\otimes n} \to V^{\otimes n}$ for the linear map that permutes the factors of $V^{\otimes n}$ taking the $i$-th factor to the $\sigma(i)$-th factor, explicitly $$ \sigma_* : V^{\otimes n} \to V^{\otimes n} \qquad , \qquad \sigma_* (x_1 \otimes \cdots \otimes x_n):= x_{\sigma^{-1}(1)} \otimes \cdots \otimes x_{\sigma^{-1}(n)}. $$

\begin{construction}\label{const:I_n}
Let $(A,R, \kappa)$ be an XC-algebra over a commutative ring $\Bbbk$, and let $n>0$ be an integer. Let us consider the cartesian product $A^{\otimes n} \times \mathfrak{S}_n$ of the underlying sets of the $n$-fold tensor power of $A$ and the symmetric group of $n$ elements. We define over $A^{\otimes n} \times \mathfrak{S}_n$ a monoid structure as follows: its multiplication law is given by 
\begin{equation}\label{eq:composition_E(A)}
(u, \sigma) \cdot (v,\tau):= (\tau^{-1}_*(u) \cdot v \ , \  \sigma \tau) .
\end{equation}
The associativity of the composition follows from the fact that $\tau_*$ is an algebra map and that $(  \sigma\tau)_* = \sigma_* \circ \tau_*$. The unit element of this monoid is given by the pair $\bm{1}:= (1 \otimes  \cdots \otimes 1, \id).$

For $ 1 \leq i <n$, let us write $$\hat{R}_i := (  1 \otimes \cdots \otimes  R  \otimes \cdots \otimes 1 , s_i  ) \in A^{\otimes n} \times \mathfrak{S}_n.$$ Note that this element is invertible in $A^{\otimes n} \times \mathfrak{S}_n$ with inverse $$\hat{R}_i^{-1} := (  1 \otimes \cdots \otimes  R^{-1}_{21}  \otimes \cdots \otimes 1 , s_i ) .$$  We define $\mathcal{J}_n$ as the submonoid of $A^{\otimes n} \times \mathfrak{S}_n$ generated by the elements $\bm{1}$ and $\hat{R}_i^{\pm 1}$ for $ 1 \leq i <n$.

Now, for every $m \geq 0$, let $\mathcal{F}_{n,m} \subset A^{\otimes n+m} \times  \SS_{n+m}$ (resp. 
$\mathcal{J}_{n,m} \subset \mathcal{J}_{n+m}$) be the subset of pairs $(u, \sigma) \in A^{\otimes n+m} \times  \SS_{n+m}$ (resp.   $(u, \sigma) \in \mathcal{J}_{n+m}$) such that $\sigma$ contains no cycles of length $\leq m$ including only elements of the set $\{ n+1 , \ldots , n+m  \}$. We define set-theoretical maps 
\begin{equation}\label{eq:phi_nm}
 \varphi_{n,m} :  \mathcal{F}_{n,m} \to  A^{\otimes n} \times \mathfrak{S}_n
\end{equation}
inductively as follows: first set $\varphi_{n,0}$ to be the natural inclusion. Now  if $$(u, \sigma)= \left( \sum x_1 \otimes \cdots \otimes x_{n+1}, \sigma \right) \in  \mathcal{F}_{n,1},$$ note that $\sigma (n+1) \neq n+1$. Then  define
\begin{align*}
\varphi_{n,1} &(u, \sigma) := \\
&\left(\sum x_1 \otimes  \cdots \otimes x_{\sigma^{-1}(n+1)-1} \otimes x_{n+1}\kappa x_{\sigma^{-1} (n+1)}  \otimes x_{\sigma^{-1}(n+1)+1} \otimes \cdots \otimes x_{n}, \widehat{\sigma}\right), 
\end{align*}
%$$ \varphi_{n,1}(u, \sigma) := \left(\sum x_1 \otimes  \cdots \otimes x_{\sigma^{-1}(n+1)-1} \otimes x_{\sigma^{-1} (n+1)} \kappa^{-1} x_{n+1} \otimes x_{\sigma^{-1}(n+1)+1} \otimes \cdots \otimes x_{n}, \widehat{\sigma}\right),  $$
where for $j=1, \ldots , n$, $$\widehat{\sigma}(j) = \begin{cases} \sigma(n+1), & j= \sigma^{-1}(n+1), \\ \sigma (j), & \mathrm{else}  \end{cases}.$$
For $m >1$, let $\varphi_{n,m} := \varphi_{n,m-1} \circ \varphi_{n+m-1,1}$. Note that this composite is well-defined, that is, that the image of $ \varphi_{n+m-1,1}$ lies in $\mathcal{F}_{n,m-1}$. Define $ \varphi'_{n,m}$ as the restriction of $ \varphi_{n,m}$ to $\mathcal{J}_{n,m} $,
\begin{equation*}
 \varphi'_{n,m}:= (\varphi_{n,m})_{|\mathcal{J}_{n,m} } :  \mathcal{J}_{n,m} \to  A^{\otimes n} \times \mathfrak{S}_n.
\end{equation*}

Finally, let $\mathcal{I}_{n}$ be the smallest submonoid of $A^{\otimes n} \times \mathfrak{S}_n$ that contains the subsets $\mathrm{Im}(\varphi'_{n,m})$ for $m \geq 0$.  For $n=0$, we set $\mathcal{I}_{0}:= \Bbbk$ with its multiplicative monoid structure.
\end{construction}

\begin{construction}\label{constr:E(A)}
Given an  XC-algebra $A$, we will now construct an open-traced  monoidal category $\mathcal{E}(A)$. Consider the family $(\mathcal{I}_n)_{n}$ of monoids defined in the previous construction. Now consider the  monoid homomorphisms 
$$\rho_{n,m} : \mathcal{I}_n \times \mathcal{I}_m \to \mathcal{I}_{n+m} \quad, \quad \rho_{n,m}((u, \sigma), (v, \xi)) := (u \otimes v, \sigma \otimes \xi),$$
where if $\sigma \in \mathfrak{S}_n$ and $\xi \in \mathfrak{S}_m$ then $\sigma \otimes  \xi \in \mathfrak{S}_{n+m}$ denotes the block permutation. If $m=0$, then this formula is meant to be $\rho_{n,0}((u,\sigma), \lambda):= (\lambda u, \sigma)$ and similarly for $\rho_{0,n}$. Also,  $\rho_{0,0}$ is just the multiplication law of the base ring $\Bbbk$.

It is readily seen that these maps satisfy the equalities \eqref{eq:M_n1} and \eqref{eq:M_n2}, so by  \cref{subsec:constr_M_n} we obtain a strict monoidal category that we denote by $\mathcal{E}(A)$ and call the  \textit{category of elements}  of the XC-algebra $A$.

We can further endow $\mathcal{E}(A)$ with a balancing structure. For the braiding, define $\tau_{0,0} := 1 \in \Bbbk$, $\tau_{1,0}=\tau_{0,1}:=(1, \id)$, and $  \tau_{1,1} :=(R, (12)) \in \eend{\mathcal{E}(A)}{2}$. For $n,m > 1$, the definition of $\tau_{n,m}$  is forced by the axioms of a braided category -- so these can be defined inductively. For the twist, let $\theta_0 := 1$ and let $\theta_1 := (\nu , \id)$, where $\nu$ is the inverse of the classical ribbon element. For $n>1$, the definition of $\theta_n$ is forced by the axioms of the twist -- so these are also defined inductively.

Now we will endow $ \mathcal{E}(A)$ with an open trace. To begin with, consider the canonical functor $P: \mathcal{E}(A) \to \mathfrak{S}$ which is the identity on objects and on arrows simply projects the permutation. Clearly this is a strict monoidal functor respecting the balancing structure. Now, if $f :n +m \to n+m$ is $(n,n,m)$-admissible, then define 
% Set $\mathrm{tr}_{n,n}^0(f):=f$ for $f:n \to n$. Now if $f$ is $(n-1,n-1,1)$-admissible, that is, if $\sigma (n) \neq n$, then define 
\begin{equation}\label{eq:def_trace}
\mathrm{tr}_{n,n}^m(f) := \varphi_{n,m}(f) 
\end{equation}
where the maps $\varphi_{n,m}$ are those defined in \eqref{eq:phi_nm}.
 \end{construction}

\begin{theorem}\label{thm:E(A)_opentraced}
For any XC-algebra $A$, its category of elements $\mathcal{E}(A)$ is indeed a strict  open-traced monoidal category.
\end{theorem}
\begin{proof}
The category $\mathcal{E}(A)$ is strict monoidal by \cref{subsec:constr_M_n}. In order to show that it is open traced, let us make the following observation: we can regard the axioms (XC0) -- (XC3) for an XC-algebra as the algebraic analogues of the rotational  Reidemeister moves \eqref{eq:R0} -- \eqref{eq:R2c_and_R3} for rotational diagrams. More precisely, suppose that $A$ has unit $1 \in A$, universal $R$-matrix $R= \sum_i \alpha_i \otimes \beta_i$, inverse $R^{-1} = \sum_i \bar{\alpha}_i \otimes \bar{\beta}_i$ and balancing element $\kappa \in A$. Place beads representing the elements $1, R, R^{-1}, \kappa^{-1}$ and $\kappa$ in the building blocks $I, X, X^-, C$ and $C^-$, respectively, as depicted below,

\begin{equation}\label{eq:beads}
\centre{
\labellist \small \hair 2pt
\pinlabel{$ \color{violet} \bullet$} at 16 30
\pinlabel{$ \color{violet} 1$} at -26 30
\pinlabel{$ \color{violet} \bullet$} at 294 30
\pinlabel{$ \color{violet} \bullet$} at 406 30
\pinlabel{$ \color{violet} \alpha_i$} [r] at 279 30
\pinlabel{$ \color{violet} \beta_i$} [l] at 421 30
\pinlabel{$ \color{violet} \bullet$} at 711 30
\pinlabel{$ \color{violet} \bullet$} at 825 30
\pinlabel{$ \color{violet} \bar{\beta}_i$} [r] at 696 30
\pinlabel{$ \color{violet} \bar{\alpha}_i$} [l] at 850 30
\pinlabel{$ \color{violet} \bullet$} at 1083 92
\pinlabel{$ \color{violet} \kappa^{-1}$} [r] at 1090 120
\pinlabel{$ \color{violet} \bullet$} at 1599 92
\pinlabel{$ \color{violet} \kappa$} [l] at 1614 104
\endlabellist
\centering
\includegraphics[width=0.5\textwidth]{figures/building_blockss}}
\end{equation}
\vspace{0.01cm}

\noindent putting the  ``alpha'' always in the overstrand and the ``beta'' in the understrand. If $D$ is a rotational diagram of an upwards tangle (with ordered components), for every $1 \leq i \leq n$ let $\mathfrak{Z}_A(D)_{(i)}$ be the (formal) word given by writing from right  to left the labels of the beads in the $i$-th component according to the  orientation of the strand. Then put
\begin{equation}\label{eq:def_mathfrak_Z}
\mathfrak{Z}_A(D) = \sum \mathfrak{Z}_A(D)_{(1)} \otimes \cdots \otimes \mathfrak{Z}_A(D)_{(n)} \in A^{\otimes n}
\end{equation}
where the summation runs through all subindices in $R^{\pm 1}$ (one for each crossing). Then it is straightforward to see that  the images under $\mathfrak{Z}_A$ of the equalities of the rotational Reidemeister moves \eqref{eq:R0} -- \eqref{eq:R2c_and_R3} are exactly the axioms (XC0) -- (XC3). By \cref{cor:uptangles_rotdiag}, this means that $\mathfrak{Z}_A$ descents, for every $n>0$,  to a well-defined set-theoretical map
$$
\mathfrak{Z}_A:  \eend{\Tup}{+^n} \to A^{\otimes n}
$$
(when $A$ is ribbon, this is essentially the non-functorial isotopy invariant devised by Lawrence \cite{lawrence}).
Since $R$ (resp. $R^{-1}$) is the value under $\mathfrak{Z}_A$ of the positive (resp. negative) crossing, it follows that $\mathcal{J}_n$  can be described as the set of pairs $(\mathfrak{Z}_A(b), \sigma_b)$ where $b$ is an $n$-braid. Similarly, since $\kappa$ is the value attached to $C^-$, the map $\varphi_{n,m}$  is precisely the algebraic counterpart of the partial closure \eqref{eq:generic_closure} for braids. Therefore $\mathcal{I}_n$ consists of products of elements whose first component is of the form $\mathfrak{Z}_A(D)$, where $D$ is the partial closure of a braid.  The upshot of this is   equalities in $\mathcal{I}_n \subset A^{\otimes n}$ involve only elements that can be expressed in terms of the universal $R$-matrix and balancing element and therefore  can be checked diagrammatically invoking the isotopy invariance of $\mathfrak{Z}_A$.

% we will be able to demonstrate equalities in $A^{\otimes n}$ simply by the isotopy invariance of $\mathfrak{Z}_A$.

Let us apply this strategy to show in the first place that the category $\mathcal{E}(A)$ is braided. The set of morphisms $\tau_{n,m}$ satisfies the axioms of a braiding by construction, so we only have to check that they are natural, that is, that $\tau$ is indeed a natural transformation $\otimes \Longrightarrow \otimes^{\mathrm{op}}$. The naturality means that for any pair of arrows $f: n \to n $ , $g: m \to m$ we have $$ (g \otimes f) \circ  \tau_{n,m} = \tau_{n,m} \circ (f \otimes g).$$

Now observe that the first component of $\tau_{n,m}$ equals $\mathfrak{Z}_A (c_{n,m})$ where $$c_{n,m}:=(\sigma_m \cdots \sigma_1)(\sigma_{m+1} \cdots \sigma_2) \cdots (\sigma_{m+n-1} \cdots \sigma_n) \in B_{n+m}$$   is the braiding constrain of $\mathcal{B}^0$. Therefore, the naturality of $\tau$ is a consequence of the isotopy
\begin{equation*}
\centre{
\labellist \small \hair 2pt
\pinlabel{\footnotesize{$ \cdots$}}  at 125 42
\pinlabel{\footnotesize{$ \cdots$}}  at 365 42
\pinlabel{\footnotesize{$ \cdots$}}  at 991 252
\pinlabel{\footnotesize{$ \cdots$}}  at 1249 252
\pinlabel{$ D'$}  at 90 416
\pinlabel{$ D$}  at 404 415
\pinlabel{$ D$}  at 964 135
\pinlabel{$ D'$}  at 1272 139
\pinlabel{$ =$}  at 665 280
\endlabellist
\centering
\includegraphics[width=0.5\textwidth]{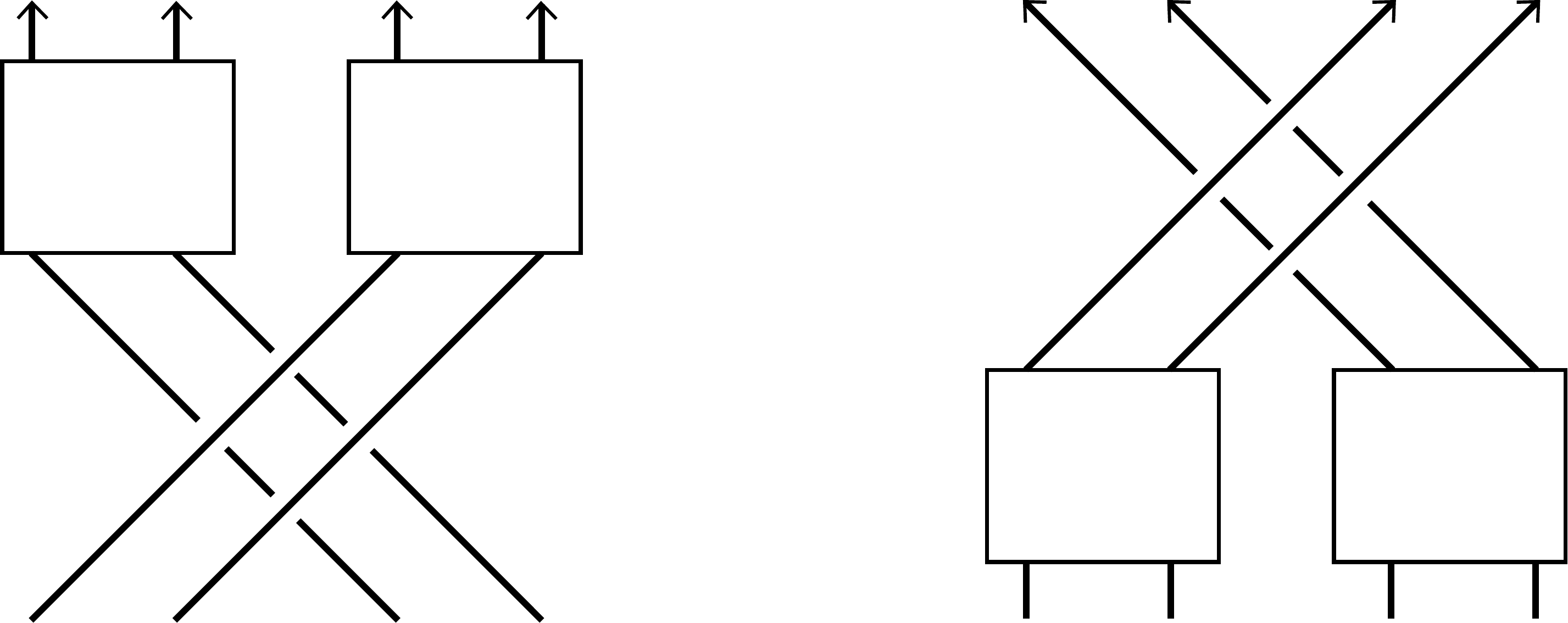}}
\end{equation*}
after applying $ \mathfrak{Z}_A$, where $D$ and $D'$ denote partial closures of braids.

Let us pass now to the twist $\theta$.  By construction, $\theta$ satisfies the axioms of a twist, so again all one has to show is that $\theta$ is a natural transformation $\theta : \id_{\mathcal{E}(A)} \Longrightarrow \id_{\mathcal{E}(A)}$, that is, that for every $f: n \to n$ we have $$f \circ \theta_n = \theta_n \circ f,$$ or in other words, that $ \theta_n$ is central in the monoid $\mathcal{I}_n$.  Again the key observation is that the first component of  $\theta_n$ equals the value under $\mathfrak{Z}_A$ of the twist $t_n$  of $\mathcal{B}$. Therefore, the naturality of $\theta$ follows then from the tangle isotopy
\begin{equation*}
\centre{
\labellist \small \hair 2pt
\pinlabel{$ D$}  at 81 710
\pinlabel{$ D$}  at 914 110
\pinlabel{$ =$}  at 665 360
\pinlabel{\scriptsize{$ \cdots$}}  at 84 100
\pinlabel{\scriptsize{$ \cdots$}}  at 920 222
\endlabellist
\centering
\includegraphics[width=0.5\textwidth]{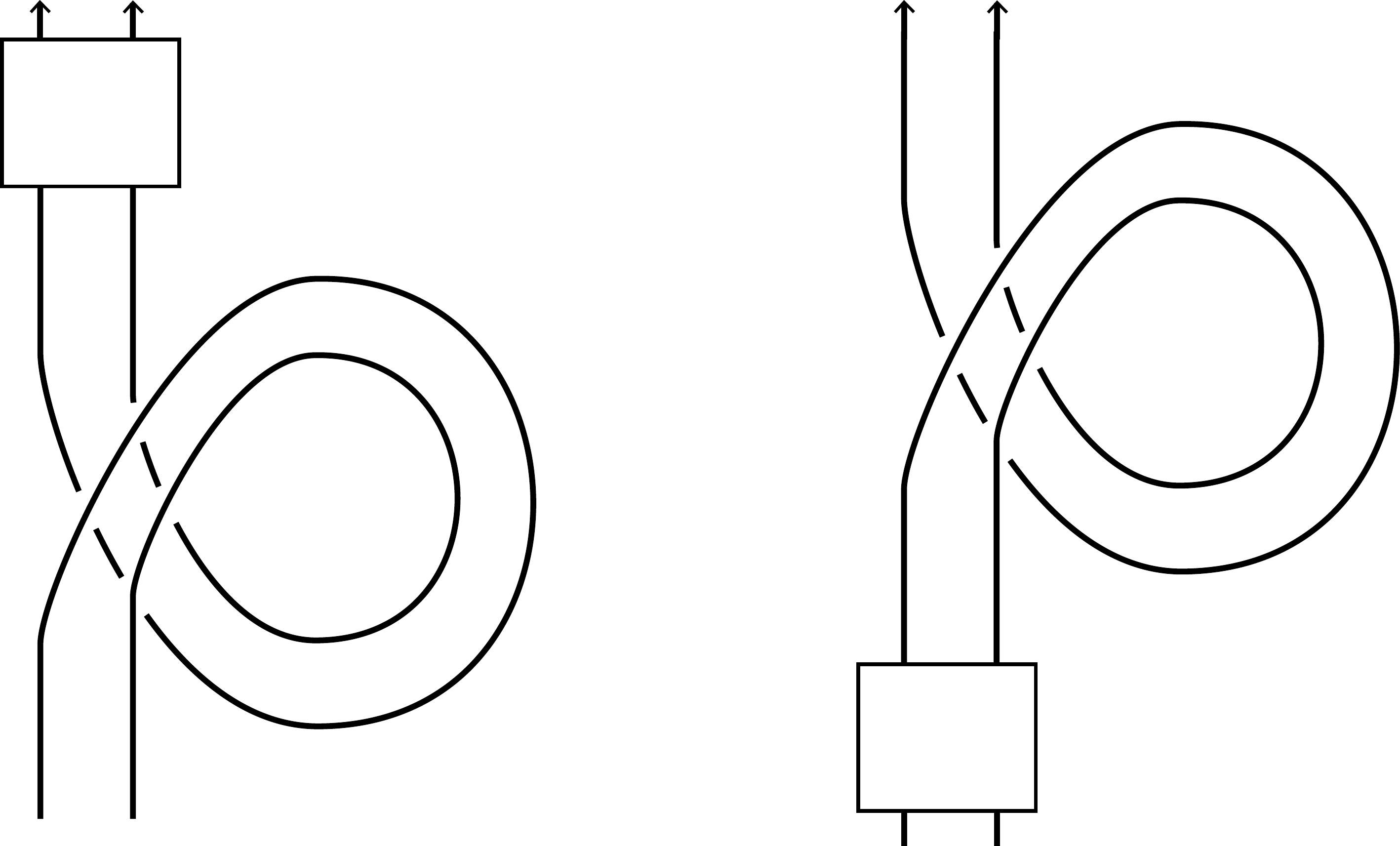}}
\end{equation*}
after applying $\mathfrak{Z}_A$, where again $D$ denotes the  partial closure of a braid.

%
%This amounts to showing that 
%$$R^{\pm 1} \cdot (1 \otimes \nu) =  (1 \otimes \nu) \cdot  R^{\pm 1} \qquad , \qquad  R^{\pm 1} \cdot ( \nu \otimes 1) =   ( \nu \otimes 1) \cdot  R^{\pm 1} $$
%and $$ \nu \cdot \kappa^{\pm} = \kappa^{\pm} \cdot \nu. $$ The naturality of the braiding implies the two first equalities, and the last one follows from (XC0), for applying the multiplication map $\mu$ to 
%$$ (\kappa^{-1} \otimes \kappa^{-1}) \cdot R= R \cdot (\kappa^{-1} \otimes \kappa^{-1})$$ yields the required equality. The general case can be shown by an inductive argument.

Lastly let us check that $\mathcal{E}(A)$ is open traced. First let us argue that \eqref{eq:def_trace} is well-defined, that is, that $\mathrm{tr}_{n,n}^m (f) \in \mathcal{I}_n$ for any  $f: n+m \to n+m$ which is $(n,n,m)$-admissible. Any such $f$ can be viewed as the value under $\mathfrak{Z}_A$ of a stacking of partial closures of braids. Then the isotopy
\begin{equation*}
\centre{
\labellist \small \hair 2pt
\pinlabel{$ b_1$}  at 294 1716
\pinlabel{$ b_1$}  at 1944 1716
\pinlabel{$ b_n$}  at 294  816
\pinlabel{$ b_n$}  at 1944  816
\pinlabel{$ \vdots$}  at 96 1299
\pinlabel{$ \vdots$}  at  324  1299
\pinlabel{$ \vdots$}  at  1722  1299
\pinlabel{$ \vdots$}  at  1950  1299
\pinlabel{$ \vdots$}  at  2233  1299
\pinlabel{$ =$}  at 1308 1242
%\pinlabel{\scriptsize{$ \cdots$}}  at 84 100
%\pinlabel{\scriptsize{$ \cdots$}}  at 920 222
\endlabellist
\centering
\includegraphics[width=0.7\textwidth]{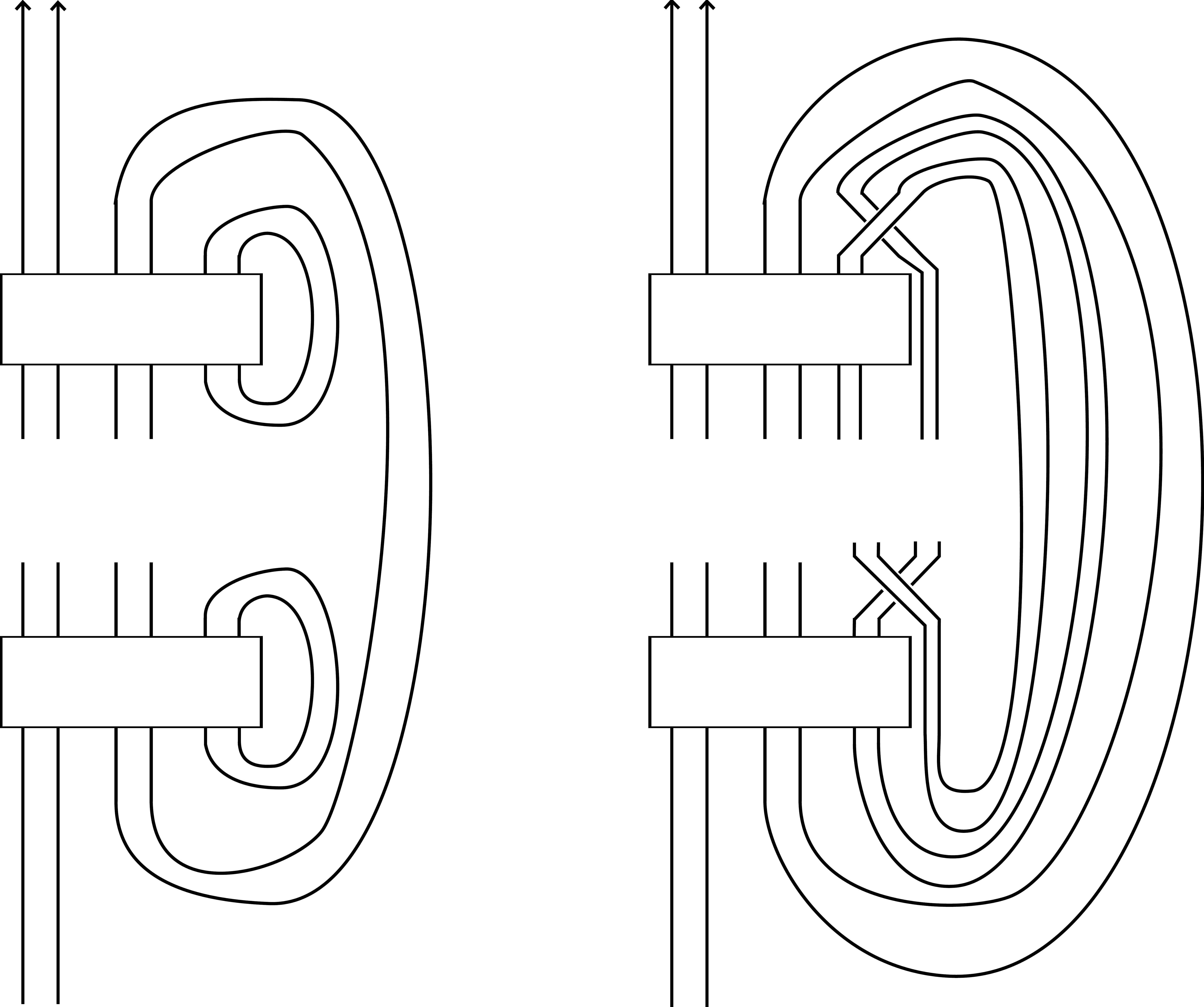}}
\end{equation*}
demostrates, after applying $\mathfrak{Z}_A$, that the trace is well-defined. Let us now check the axioms:  (OTC1), (OTC2) and (OTC4) hold by construction. (OTC3) follows after applying $\mathfrak{Z}_A$ to the following isotopic tangles:
\begin{equation*}
\centre{
\labellist \small \hair 2pt
\pinlabel{$ D$}  at 195 520
\pinlabel{$ D'$}  at 843 530
\pinlabel{$ D$}  at 1627 514
\pinlabel{$ D'$}  at 2000 525
\pinlabel{$ =$}  at 1152 514
\pinlabel{\scriptsize{$ \cdots$}}  at 115 304
\pinlabel{\scriptsize{$ \cdots$}}  at 1531 260
\pinlabel{\scriptsize{$ \cdots$}}  at 325 358
\pinlabel{\scriptsize{$ \cdots$}}  at 845 337
\endlabellist
\centering
\includegraphics[width=0.8\textwidth]{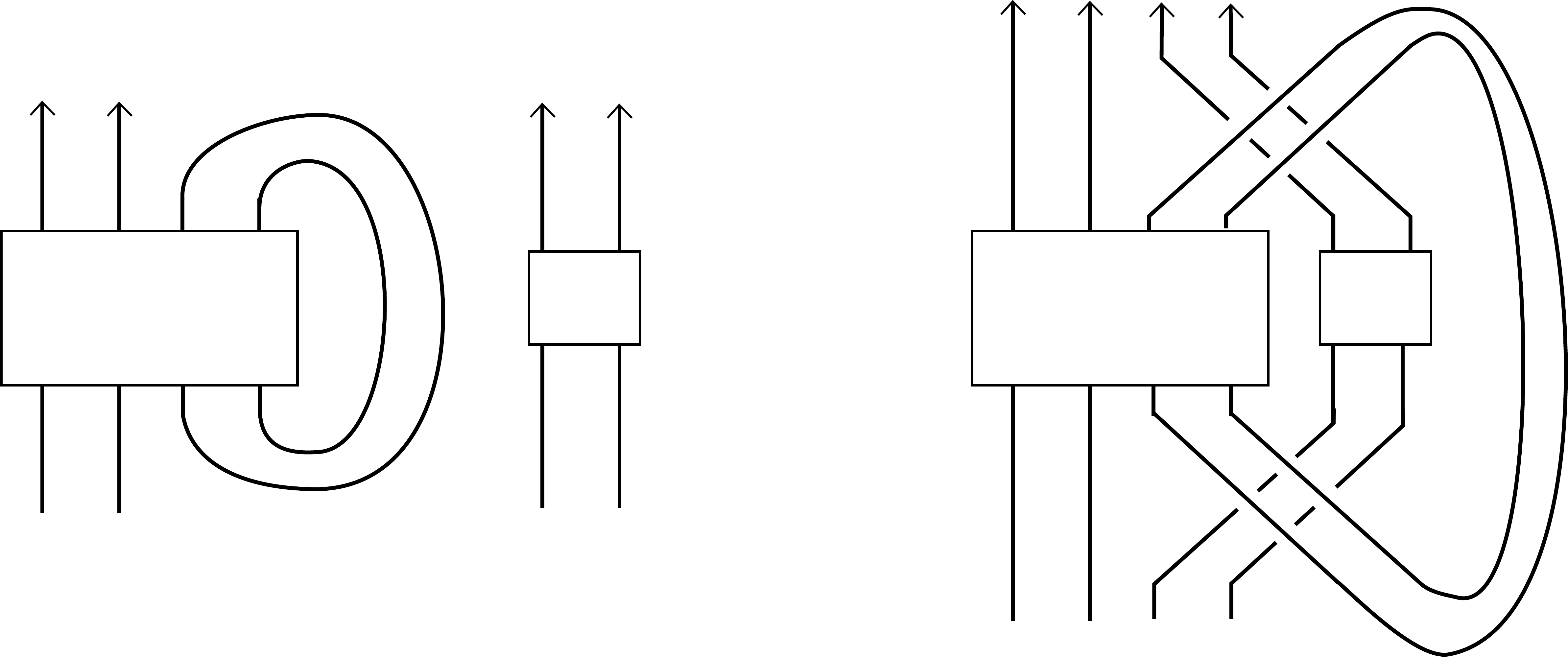}}
\end{equation*}
It is only left to show the naturality of $\mathrm{tr}_{n,n}^m$ on both $n$ and $m$.  The naturality on $n$ means that for any pair of arrows $u,v: n \to n$ we have
$$ \mathrm{tr}_{n,n}^m ((v \otimes  \id) \circ f \circ (u\otimes \id)) = v \circ \mathrm{tr}_{n,n}^m (f) \circ u. $$ Again this is a consequence of the following tangle isotopy after applying $\mathfrak{Z}_A$.
\begin{equation*}
\centre{
\labellist \small \hair 2pt
\pinlabel{$ D$}  at 113 880
\pinlabel{$ D'$}  at 111 152
\pinlabel{$ D$}  at 1166 773
\pinlabel{$ D'$}  at 1168 260
\pinlabel{$ =$}  at 808 530
\pinlabel{\scriptsize{$ \cdots$}}  at 115 304
\pinlabel{\scriptsize{$ \cdots$}}  at 1160 72
\pinlabel{\scriptsize{$ \cdots$}}  at 325 358
\pinlabel{\scriptsize{$ \cdots$}}  at 1364 333
\endlabellist
\centering
\includegraphics[width=0.6\textwidth]{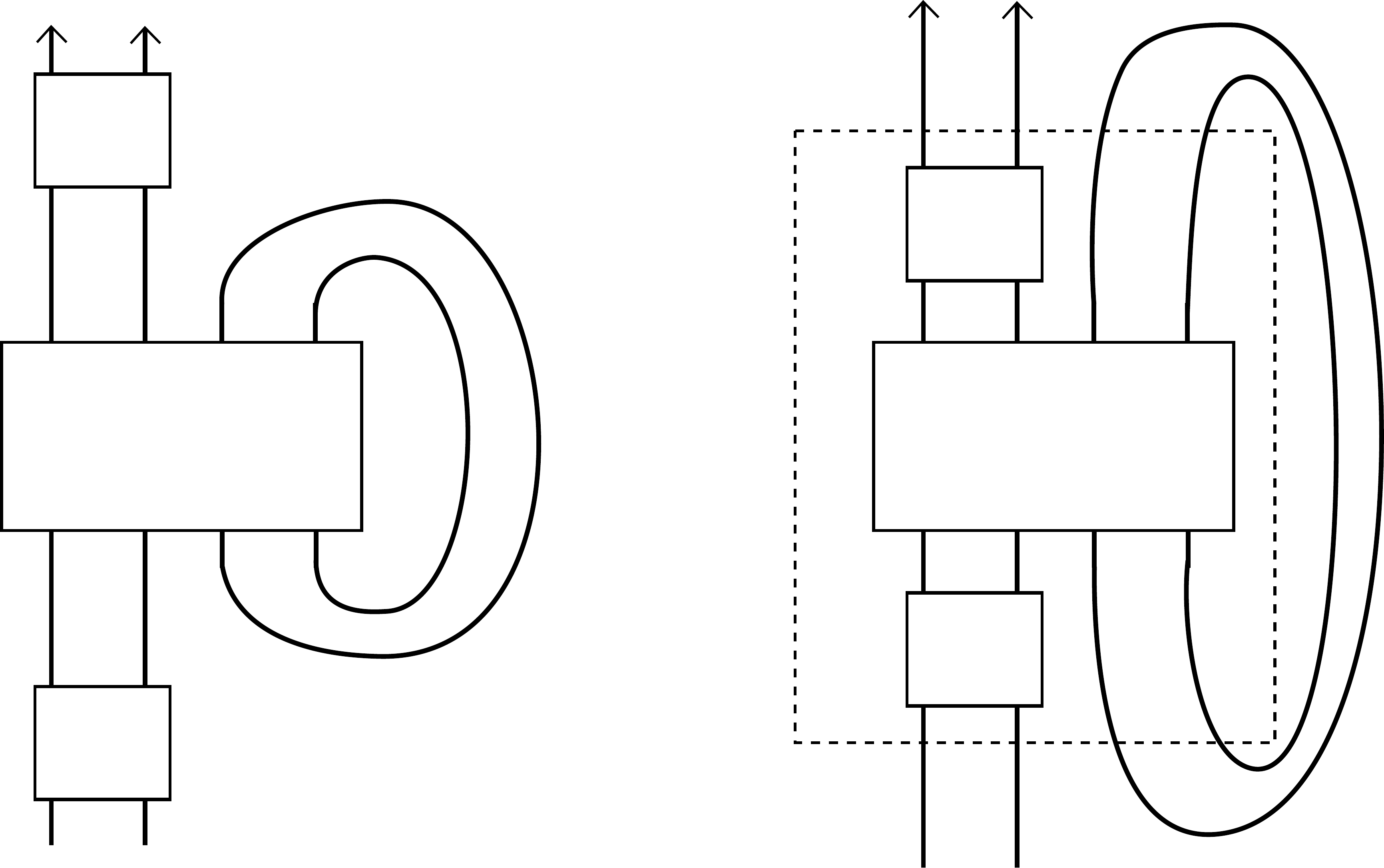}}
\end{equation*}
Similarly, the naturality on $m$ means that for any $u: m \to m$ we have
$$ \mathrm{tr}_{n,n}^m (f \circ (\id \otimes u))  = \mathrm{tr}_{n,n}^m (  (\id \otimes u)\circ f) .$$ This follows from the tangle isotopy
\begin{equation*}
\centre{
\labellist \small \hair 2pt
\pinlabel{$ D$}  at 200 300
\pinlabel{$ D$}  at 1230 523
\pinlabel{$ D'$}  at 317 551
\pinlabel{$ D'$}  at 1345 282
\pinlabel{$ =$}  at 835 400
\pinlabel{\scriptsize{$ \cdots$}}  at 110 100
\pinlabel{\scriptsize{$ \cdots$}}  at 1130 222
\pinlabel{\scriptsize{$ \cdots$}}  at 317 134
\pinlabel{\scriptsize{$ \cdots$}}  at 1345 695
\endlabellist
\centering
\includegraphics[width=0.6\textwidth]{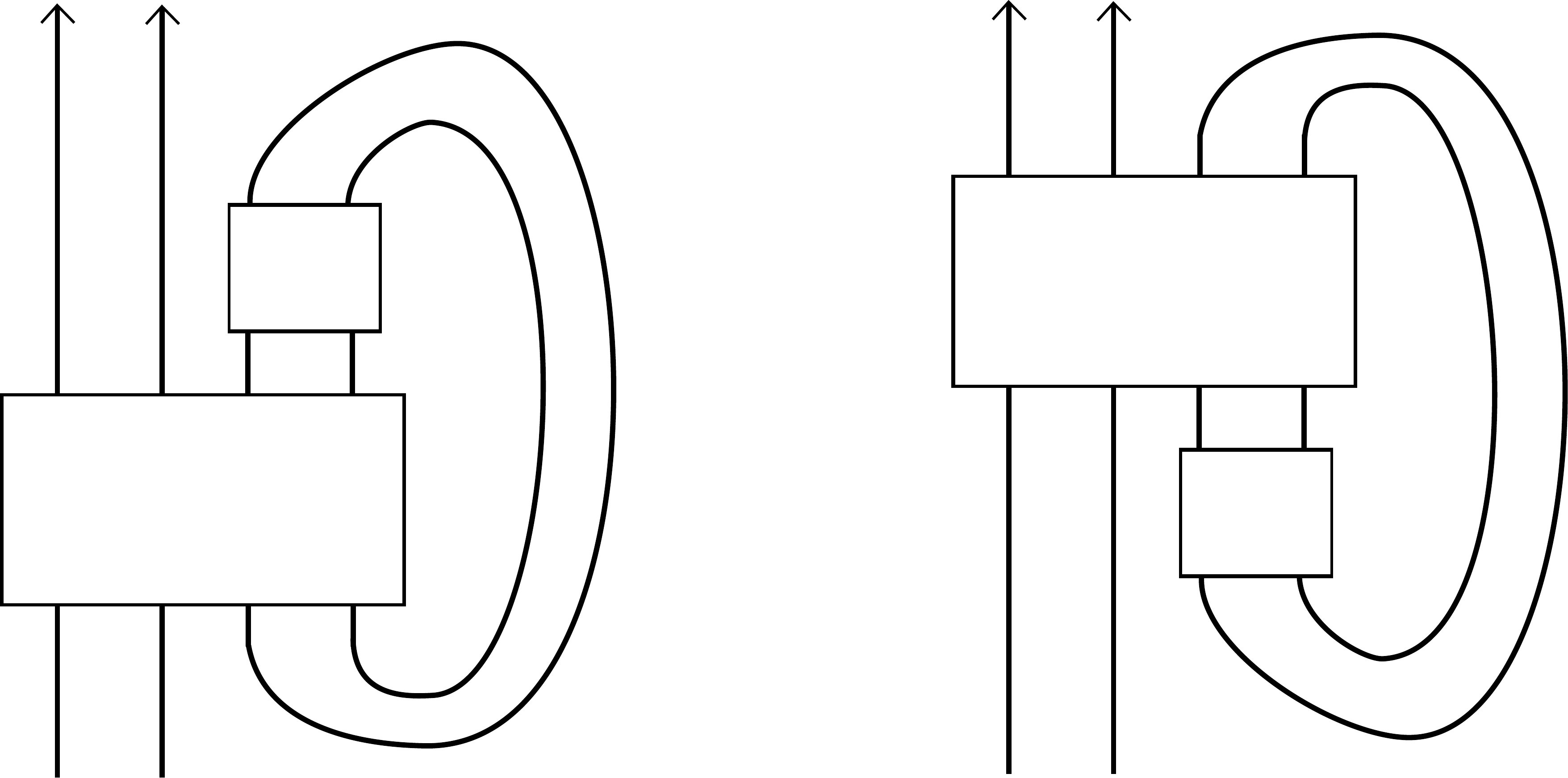}}
\end{equation*}
after applying $\mathfrak{Z}_A$. This concludes the proof.
\end{proof}

The unique strict monoidal functor 
\begin{equation}\label{eq:Z_A}
Z_A  : \Tup \to \mathcal{E}(A)
\end{equation}
which  according to \cref{thm:UP_opentraced} maps $+$ to $1$ and preserves the open-trace monoidal structure  will be called the \textit{universal invariant} with respect to the XC-algebra $A$.

\begin{remark}\label{rem:Z_mathfrakZ}
It should be clear from the proof of the previous theorem that $Z_A$ is bijective-on-objects, full and that for an upwards tangle $T \in \operatorname{End}_{\Tup}(+^n)$ we have that $$ Z_A(T)= (\mathfrak{Z}_A (T), \sigma_T) .$$ 
\end{remark}

\begin{example}
Let $D$ be the 2-component  upwards tangle diagram below, that we have already decorated  according to \eqref{eq:beads}:
\vspace*{5pt}
\begin{equation*}
\centre{
\labellist \small \hair 2pt
\pinlabel{{\normalsize $D$}} at -134 854
\pinlabel{$ \color{red} \bullet$} at 262 172
\pinlabel{$ \color{red} \bullet$} at 356 172
\pinlabel{$ \color{red} \alpha_i$} [r] at 247 172
\pinlabel{$ \color{red} \beta_i$} [l] at 371 172
\pinlabel{$ \color{OliveGreen} \bullet$} at 262 363
\pinlabel{$ \color{OliveGreen} \bullet$} at 356 363
\pinlabel{$ \color{OliveGreen} \alpha_j$} [r] at 247 363
\pinlabel{$ \color{OliveGreen} \beta_j$} [l] at 371 363
\pinlabel{$ \color{magenta} \bullet$} at 428 535
\pinlabel{$ \color{magenta} \bullet$} at 542 535
\pinlabel{$ \color{magenta} \bar{\beta}_\ell$} [r] at 413 535
\pinlabel{$ \color{magenta} \bar{\alpha}_\ell$} [l] at 557 535
\pinlabel{$ \color{blue} \bullet$} at 250 712
\pinlabel{$ \color{blue} \bullet$} at 364 712
\pinlabel{$ \color{blue}  \bar{\beta}_s$} [r] at 235 712
\pinlabel{$ \color{blue} \bar{\alpha}_s$} [l] at 379 712
\pinlabel{$ \color{orange} \bullet$} at 3 498
\pinlabel{$ \color{orange} \kappa^{-1}$} [r] at -13 498
\endlabellist
\centering
\includegraphics[width=0.2\textwidth]{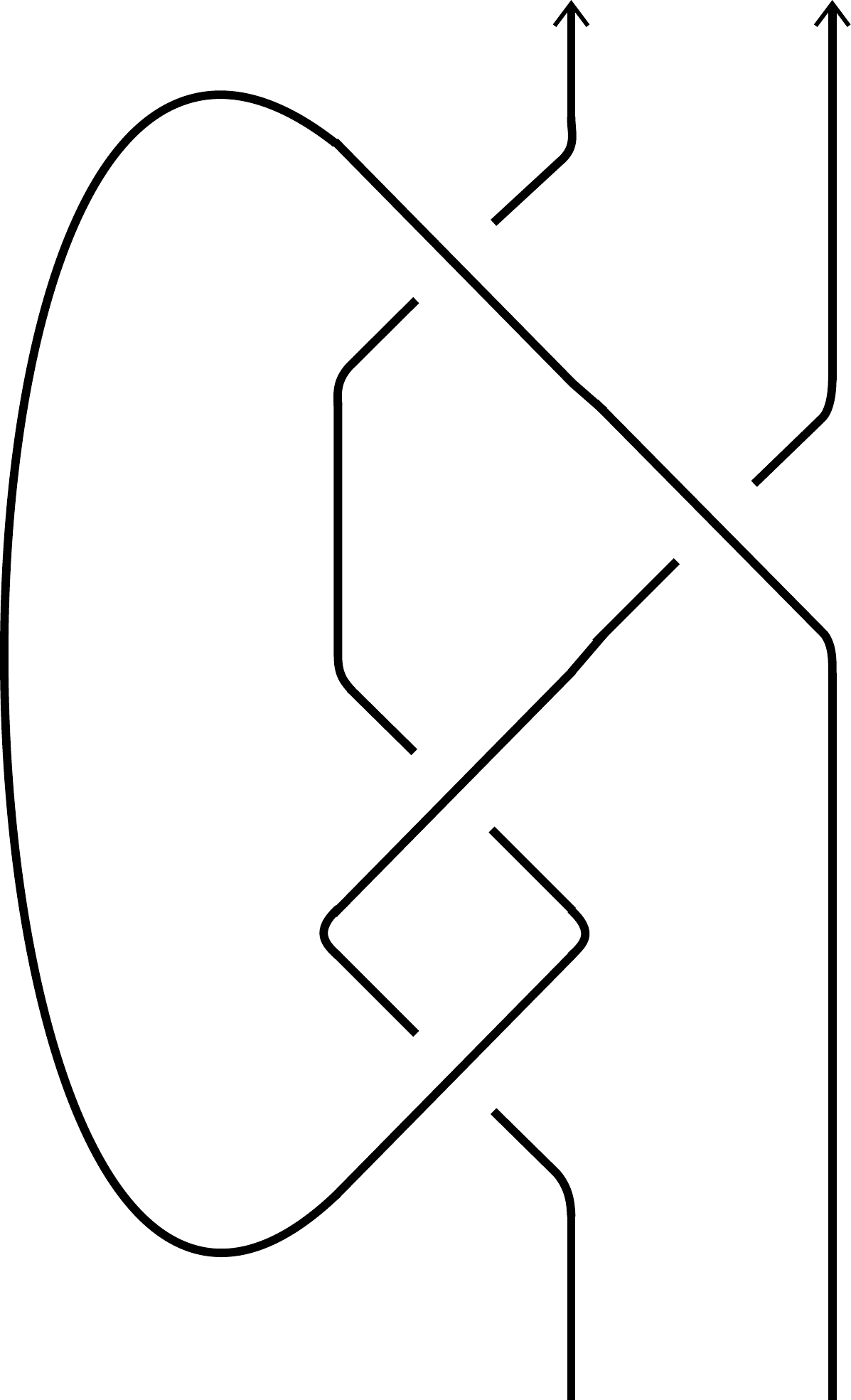}}
\end{equation*}
Then we have 
$$\mathfrak{Z}_A(D)= \sum_{i,j,\ell,s}   \bar{\beta}_\ell \alpha_j \beta_i \otimes      \bar{\beta}_s \beta_j\alpha_i\kappa^{-1} \bar{\alpha}_s \bar{\alpha}_\ell\in A \otimes A,$$
so that its universal invariant is $Z_A(T)= (\mathfrak{Z}_A(D), (12))$.
\end{example}

We would now like to consider the case where $A$ is a traced XC-algebra. In this case it will turn out that $\mathcal{E}(A)$ is a traced monoidal category in the sense of Joyal-Street-Verity:

\begin{construction}\label{const:I_n2}
Let $(A, R, \kappa) $ be a traced XC-algebra. We are now going to slightly modify  \cref{const:I_n}  in order to include the trace of $A$. First, we let $\mathcal{J}_{n} \subset A^{\otimes n} \times  \SS_n$ be as in \cref{const:I_n}. Let us now define set-theoretical maps
\begin{equation}\label{eq:varphi_tr}
 \varphi_{n,m} :  A^{\otimes n+m} \times \mathfrak{S}_{n+m} \to  A^{\otimes n} \times \mathfrak{S}_n
\end{equation}
inductively as follows: first set $\varphi_{n,0}$ to be the identity. If $$(u, \sigma)= \left( \sum x_1 \otimes \cdots \otimes x_{n+1}, \sigma \right) \in   A^{\otimes n+1} \times \mathfrak{S}_{n+1} ,$$ we distinguish two situations: if $\sigma (n+1) \neq n+1$, then  define $\varphi_{n,1} (u, \sigma)$ as in \cref{const:I_n}, that is,
\begin{align*}
\varphi_{n,1} &(u, \sigma) := \\
&\left(\sum x_1 \otimes  \cdots \otimes x_{\sigma^{-1}(n+1)-1} \otimes  x_{n+1} \kappa x_{\sigma^{-1} (n+1)} \otimes x_{\sigma^{-1}(n+1)+1} \otimes \cdots \otimes x_{n}, \widehat{\sigma}\right), 
\end{align*}
%$$ \varphi_{n,1}(u, \sigma) := \left(\sum x_1 \otimes  \cdots \otimes x_{\sigma^{-1}(n+1)-1} \otimes x_{\sigma^{-1} (n+1)} \kappa^{-1} x_{n+1} \otimes x_{\sigma^{-1}(n+1)+1} \otimes \cdots \otimes x_{n}, \widehat{\sigma}\right),  $$
where for $j=1, \ldots , n$, $$\widehat{\sigma}(j) = \begin{cases} \sigma(n+1), & j= \sigma^{-1}(n+1), \\ \sigma (j), & \mathrm{else}  \end{cases}.$$
If $\sigma(n+1)=n+1$, then set
$$ \varphi_{n,1}(u, \sigma) := \left( \mathrm{tr}(\kappa x_{n+1} ) \sum x_1 \otimes  \cdots \otimes  x_{n}, \widetilde{\sigma} \right),  $$
where $\widetilde{\sigma} \in \mathfrak{S}_{n}$ is given by $\widetilde{\sigma} (j) := \sigma (j)$ for $j=1, \ldots , n$.

For $m >1$, let $\varphi_{n,m} := \varphi_{n,m-1} \circ \varphi_{n+m-1,1}$, and define $ \varphi'_{n,m}$ as the restriction of $ \varphi_{n,m}$ to $\mathcal{J}_{n} $,
\begin{equation}
 \varphi'_{n,m}:= (\varphi_{n,m})_{|\mathcal{J}_{n} } :  \mathcal{J}_{n} \to  A^{\otimes n} \times \mathfrak{S}_n.
\end{equation}

Finally, let $\mathcal{I}_{n}$ be the smallest submonoid of $A^{\otimes n} \times \mathfrak{S}_n$ that contains the subsets $\mathrm{Im}(\varphi'_{n,m})$ for $m \geq 0$.  For $n=0$, we set $\mathcal{I}_{0}:= \Bbbk$ with its multiplicative monoid structure.
\end{construction}

\begin{construction}
If $(A, R, \kappa)$ is a traced XC-algebra, we will now slightly modify its category of elements $\mathcal{E}(A)$ defined in \cref{constr:E(A)} to incorporate the trace. As a balanced  monoidal category, we let $\mathcal{E}(A)$ be defined as in \cref{constr:E(A)}, but using instead the monoids $\mathcal{I}_n$ from \cref{const:I_n2}. 

Now, given a morphism $f: n+m \to n+m$ in $\mathcal{E}(A)$, we define 
\begin{equation}
\mathrm{tr}_{n,n}^m(f) := \varphi_{n,m}(f) 
\end{equation}
where the maps $\varphi_{n,m}$ are those in \eqref{eq:varphi_tr}.
\end{construction}

\begin{theorem}\label{thm:E(A)_traced}
For any traced XC-algebra $A$, the category $\mathcal{E}(A)$ is a strict traced monoidal category.
%\vspace*{-3pt}
\end{theorem}
\begin{proof}
The proof is entirely similar to that of \cref{thm:E(A)_opentraced} so we will just mention what  the changes needed to adapt it are.  For $n \geq 0$, let us call \textit{$n$-tangle} to any element of $\operatorname{End}_{\T^+}(n)$, that is, a tangle with $n$ open components (and possibly also closed components) such that the open components point up near the head and the tail of the strands. By a similar argument as in \cref{lemma:every_tangle_has_rot_diag}, every $n$-tangle has a diagram in rotational form. Now, if $D$ is a rotational diagram of an $n$-tangle with $m \geq 0$ closed components, and the components are ordered so that for $1 \leq i \leq n$, the $i$-th component is open and for $n< i \leq n+m$ the $i$-th component is closed, then  let $\mathfrak{Z}_A(D)_{(i)}$ be the (formal) word given by writing from left to right the labels of the beads in the $i$-th component according to the  orientation of the strand, taking any point as basepoint for the closed components. Then set 
\begin{equation}
\mathfrak{Z}_A(D) :=\sum \mathrm{tr}(\mathfrak{Z}_A(D)_{(n+1)} ) \cdots \mathrm{tr}(\mathfrak{Z}_A(D)_{(n+m)} )  \cdot \mathfrak{Z}_A(D)_{(1)} \otimes \cdots \otimes \mathfrak{Z}_A(D)_{(n)} \in A^{\otimes n}
\end{equation}
where the summation runs through all subindices in $R^{\pm 1}$ (one for each crossing). Since the trace factors through the commutator subgroup, this element is indeed well-defined. In particular, using the analogous of \cref{cor:uptangles_rotdiag} for $n$-tangles, we get a map 
$$
\mathfrak{Z}_A:  \eend{\T^+}{+^n} \to A^{\otimes n}
$$
The rest of the proof is now identical to that of \cref{thm:E(A)_opentraced} using this map.
\end{proof}

The unique strict monoidal functor 
\begin{equation}\label{eq:traced_universal_invariant}
Z_A: \T^+ \to \mathcal{E}(A)
\end{equation}
arising from  \cref{thm:UP_traced} that sends $+$ to $1$ and preserves the braiding, twist and trace will be called the (traced) \textit{universal invariant} of $A$.

\begin{remark}
Similarly to \cref{rem:Z_mathfrakZ},       we also have in this case that for an $n$-tangle  $T \in \operatorname{End}_{\T^+}(n)$ we have that $$ Z_A(T)= (\mathfrak{Z}_A (T), \sigma_T) .$$
\end{remark}

For concrete examples that recover the Jones and Alexander polynomial of links, we refer the reader to \cite[Ch. 2]{becerra_thesis}.

\subsection{Comparison with Kerler-Kauffman-Radford functorial invariant}

In this subsection we would like to compare the canonical functor $$Z_A: \Tup \to \mathcal{E}(A)$$ from \eqref{eq:Z_A} with (an appropriate  version of) Kerler-Kauffman-Radford ``decoration functor'' $$\mathsf{Dec}_A : \Tup \to s\Tup (A)$$ which has as target the so-called category of ``singular upwards $A$-tangles''  \cite{KR3, kerler1}. Here $(A,R, \kappa)$ will be an XC-algebra over some commutative ring $\Bbbk$ .

Given a rotational diagram $D$ of an upwards tangle (with blackboard framing), we write $sD$ for the singular diagram on the square  which is identical to $D$ but forgets about the ``over/under'' information at the crossings. By a \textit{decorated singular diagram} we will understand a singular diagram where each strand carries an arbitrary number of  decorations (``beads''), each of them labelled by an element in $A$. For example if $x,y,z \in A$, we could have
\vspace*{5pt}
\begin{equation*}
\centre{
\labellist \small \hair 2pt
\pinlabel{$ \color{violet} \bullet$} at 10 70
\pinlabel{$ \color{violet} \bullet$} at 280 50
\pinlabel{$ \color{violet} y$} at 280 10
\pinlabel{$ \color{violet} x$} at -40 70
\pinlabel{$ \color{violet} \bullet$} at 130 160
\pinlabel{$ \color{violet} z$} at 170 200
\endlabellist
\centering
\includegraphics[width=0.13\textwidth]{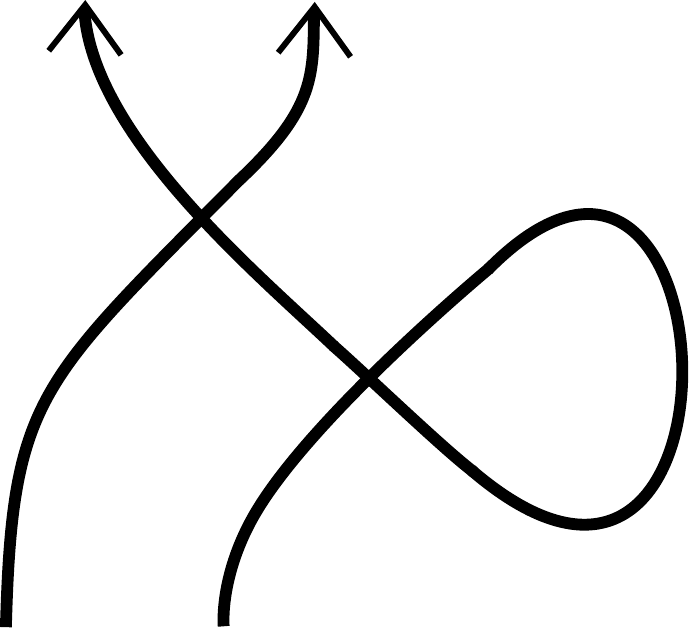}}
\end{equation*}

Given a singular diagram with beads placed, we can label these with two different set of labels. We can consider then formal linear combinations of these decorated singular diagrams, that we will still call in the same way.

We will consider an equivalence relation on the set of decorated singular diagrams. Firstly, we will identify two such diagrams if the underlying singular diagrams are homotopic rel. endpoints and such a homotopy preserves the order of the beads (extended linearly). Secondly, the labels are understood to be ``multilinear'' in the sense that a diagram with a bead labelled with $\lambda x$, where $\lambda \in \Bbbk$ and $x \in A$, can be replaced by $\lambda$ times the same decorated diagram with the corresponding bead labelled with $x$; and a diagram with a bead labelled with $x+y$, can be replaced by a sum of two diagrams where in each of them the bead is labelled by $x$ and $y$, respectively. Thirdly, we will identify a given decorated singular  diagram $sD$ with another one which is identical to $sD$ except that in a certain component all labelled beads are replaced by a single one (placed in that component) which is labelled by the (formal) word of elements of $A$ resulting from writing down the labels of the beads from right to left following the orientation of the strand. For instance given (finite) collections $(x_i)_i$, $(y_j)_j$, $(z_k)_k$ of elements of $A$ we will draw
\vspace*{5pt}
\begin{equation*}
\centre{
\labellist \small \hair 2pt
\pinlabel{$ \color{violet} \bullet$} at 10 70
\pinlabel{$ \color{violet} \bullet$} at 280 50
\pinlabel{$ \color{violet} y_j$} at 280 10
\pinlabel{$ \color{violet} x_i$} at -40 70
\pinlabel{$ \color{violet} \bullet$} at 130 160
\pinlabel{$ \color{violet} z_k$} at 170 200
\endlabellist
\centering
\includegraphics[width=0.13\textwidth]{figures/singular_diagram}}
\end{equation*}
for the corresponding sum over all indices.

The category of \textit{singular upwards $A$-tangles} $s\Tup (A)$ is defined as follows: the objects of $s\Tup (A)$ are the same as for $\Tup$, that is elements of the free monoid $\mathrm{Mon}(+)$, and the arrows are (equivalence classes of)  decorated singular diagrams. Composition is given by stacking as in $\Tup$, and the identity of $+^n$ is given by obvious diagram of the identity of $+^n$ in $\Tup$ where every component has been decorated with the unit of $A$. It is also clear that $s\Tup (A)$ admits a strict monoidal structure mimicking that of $\Tup$.

Then Kerler-Kauffman-Radford  \textit{decoration functor} $\mathsf{Dec}_A $ is defined to be the identity on objects and on arrows it is given by assigning to a given rotational diagram of an upwards tangle the underlying singular diagram which has been decorated placing labelled beads according to the rules \eqref{eq:beads}.

In order to make the comparison between $Z_A$ and $\mathsf{Dec}_A$, it will be convenient to consider a category isomorphic to $s\Tup (A)$ and that can be viewed as an ``algebraisation'' of it. This will make use once more of the bundle monoidal category construction from \cref{subsec:constr_M_n}.

\begin{construction}
Let $n>0$ be an integer. As in \cref{const:I_n}, let us take the set-theoretical cartesian product $A^{\otimes n} \times \mathfrak{S}_n$, and we consider over it the monoid structure with product given by 
$$
(u, \sigma) \cdot (v,\tau):= (\tau^{-1}_*(u) \cdot v \ , \  \sigma \tau) 
$$
as in \eqref{eq:composition_E(A)}, with unit $\bm{1}:= (1 \otimes  \cdots \otimes 1, \id)$. Note that for $n=0$, $A^{\otimes n} \times \mathfrak{S}_n = \Bbbk$. As in \ref{constr:E(A)}, there is a family of monoid maps
$$\rho_{n,m} : (A^{\otimes n} \times \mathfrak{S}_n) \times (A^{\otimes m} \times \mathfrak{S}_m) \to (A^{\otimes n+m} \times \mathfrak{S}_{n+m})$$ given by $\rho_{n,m}((u, \sigma), (v, \xi)) := (u \otimes v, \sigma \otimes \xi)$ which satisfies the conditions from \cref{subsec:constr_M_n}. The resulting strict monoidal category will be denoted by $\widetilde{s\Tup (A)}$.
\end{construction}

There is an obvious strict monoidal functor $$s\Tup (A) \to \widetilde{s\Tup (A)}$$ which is the identity on objects and
that assigns to any decorated singular diagram the pair $(\mathfrak{Z}_A(sD), \sigma_{sD})$ where $\mathfrak{Z}_A$ is defined as in \eqref{eq:def_mathfrak_Z} but for singular diagrams. The equivalence relation taken in the set of decorated singular diagrams guarantees that this functor is well-defined.

\begin{lemma}
The previous functor is an isomorphism of strict monoidal categories.
\end{lemma}
\begin{proof}
We only have to check that the functor is fully faithful. The main observation is that the underlying singular diagram is fully determined by the associated permutation up to homotopy rel. endpoints. Then any arrow $(\sum x_1 \otimes \cdots \otimes x_n, \sigma)$ is the image of the decorated singular diagram which consists of the diagram corresponding with $\sigma$ decorated with beads labelled as $x_1 , \ldots, x_n$ in the strands. 

Now if two decorated singular diagrams have the same image under the previous functor, then the underling singular diagrams must be the same. Moreover we can replace all beads in a given strand by a single one. This implies that the decorated diagrams are equivalent, hence the faithfulness.  
\end{proof}

\begin{remark}
The reader should realise that the relations of decorated singular diagrams amount to the following: labelled beads in the same component correspond to elements in $A \otimes_A \cdots \otimes_A A$, whereas beads in different components correspond to elements in $A \otimes_\Bbbk \cdots \otimes_\Bbbk A$
\end{remark}

Because of the previous lemma, we will more easily regard the category $s\Tup (A)$ as its algebraic counterpart $\widetilde{s\Tup (A)}.$.

Let us consider the monoids $\mathcal{I}_n$   constructed in \cref{const:I_n}, which by definition are submonoids of $A^{\otimes n} \times \mathfrak{S}_n$. The inclusion maps induce a strict monoidal embedding $$\mathcal{E}(A) \hooklongrightarrow s\Tup (A).  $$

We can now compare the functor $Z_A$ with the Kerler-Kauffman-Radford decoration functor $\mathsf{Dec}_A $. The proof is trivial by construction.

\begin{theorem}\label{thm:comparison_KKR}
The functor $\mathsf{Dec}_A $ factors through $Z_A$, that is, the following diagram of strict monoidal functors commute. 
$$
\begin{tikzcd}
\Tup \rar{\mathsf{Dec}_A} \dar[swap]{Z_A} & s\Tup (A) \\
\mathcal{E}(A) \arrow[hook]{ur} &
\end{tikzcd}
$$
\end{theorem}

\begin{remark}
The previous theorem allows us to regard $Z_A$ as a refinement of Kerler-Kauffman-Radford decoration functor $\mathsf{Dec}_A $. The functor $Z_A$ arises canonically from the universal property of $\Tup$ described in \cref{thm:UP_opentraced} because the category of elements $\mathcal{E}(A)$ is open-traced. However, the Kerler-Kauffman-Radford's category $s\Tup (A)$  does not in general admit a brading such that  $\mathsf{Dec}_A $ is a braided functor, even when $A$ is a ribbon Hopf algebra. The reason is that the naturality of such a braiding is equivalent to the centrality of $R$ in $A\otimes A$, that is,  the equality $R(x\otimes y) = (x \otimes y)R$ for all $x,y \in A$, which in general does not hold. For instance, for the Sweedler algebra $SW$ from \cref{ex:Sweedler} we have $$R(w \otimes 1) =  \tfrac{1}{2} \big( (w+ sw)\otimes 1 + (w-sw)\otimes s  \big)$$ whereas  $$(w \otimes 1)R =  \tfrac{1}{2} \big( (w- sw)\otimes 1 + (w+sw)\otimes s  \big).$$
\end{remark}

\section{Universality of \texorpdfstring{$Z_A$}{ZA}}\label{sec:universality}

In this last section, we would like to study how the functorial universal invariant $Z_A$ relates to the Reshetikhin-Turaev functor from \cref{thm:turaev} when $A$ is a ribbon Hopf algebra.

\subsection{Domination with respect to the Reshetikhin-Turaev functor}

The Reshetikhin-Turaev invariant arises canonically from the ribbon structure on $\mathsf{fMod}_A$.  However, as we mentioned above, ribbon structures on  $\mathsf{fMod}_A$ (extending the $\Bbbk$-tensor product and the $\Bbbk$-module duality) essentially arise  from  ribbon algebra structures on $A$. At least conceptually, what this means is that $A$ should encode all information needed to recover the functor $RT_W$ for any $A$-module. In the rest of the section we will see that this is indeed the case, and that this information is encoded in the functor $Z_A$.

\begin{construction}\label{constr:rho_W}
Let $A$ be a ribbon Hopf algebra and let $W$ be a finite-free $A$-module. We define a functor $$\rho_W: \mathcal{E}(A) \to (\mathsf{fMod}_A^{\mathrm{str}})_W$$ as follows: on objects, set $\rho_W(n) := (W)^{\otimes n} =(W, \overset{{}^n}{\ldots}, W)$ . On arrows, given a morphism $(u, \sigma) \in \eend{\mathcal{E}(A)}{n}$ with $n>0$, define $$\rho_{W}(u, \sigma) : W^{\otimes n} \to W^{\otimes n} \quad , \quad \rho_{W}(u, \sigma) (w_1 \otimes \cdots  \otimes w_n ) :=\sigma_*( u \cdot (w_1  \otimes \cdots \otimes  w_n))  . $$ For $n=0$, simply set $\rho_W(\lambda)$ to be the multiplication by $\lambda$ map. That this is indeed a functor follows from the following computation:
\begin{align*}
\rho_W( (u, \sigma) \circ (v,\tau))(w_1  \otimes \cdots \otimes  w_n) &= \rho_W (\tau^{-1}_*(u) \cdot v \ , \  \sigma \tau)(w_1  \otimes \cdots \otimes  w_n)\\
&= (\sigma \tau)_* ((\tau^{-1}_*(u) \cdot v)(w_1  \otimes \cdots \otimes  w_n))\\
&=\sigma_* (u \cdot \tau_*(v) )\cdot (\sigma_* \tau_* ) (w_1  \otimes \cdots \otimes  w_n)\\
&= \sigma_* (u \cdot \tau_*(v) \cdot \tau_* (w_1  \otimes \cdots \otimes  w_n))\\
&=  \rho_W (u, \sigma) (\tau_* ( v \cdot (w_1  \otimes \cdots \otimes  w_n))) \\
&= (\rho_W (u, \sigma)  \circ \rho_W (v, \tau)) (w_1  \otimes \cdots \otimes  w_n)
\end{align*}
This functor is actually strict monoidal, indeed given arrows $(u, \sigma) : n \to n$ and $(v, \tau): m \to m$ in $\mathcal{E}(A)$ we have
\begin{align*}
\rho_W((u, \sigma) \otimes (v, \tau))(w_1  \otimes \cdots \otimes  w_{n+m}) &= \rho_W (u \otimes v, \sigma \otimes \tau)(w_1  \otimes \cdots \otimes  w_{n+m})\\
&=(\sigma \otimes \tau)_*((u \otimes b) \cdot (w_1  \otimes \cdots \otimes  w_{n+m}))\\
&= \sigma_* (u \cdot (w_1  \otimes \cdots \otimes  w_{n}) )\otimes  \\ &\phantom{=====} \tau_* (v \cdot (w_{n+1} \otimes \cdots \otimes  w_{n+m}))\\
&= (\rho_W(u, \sigma) \otimes \rho_W(b, \tau))(w_1  \otimes \cdots \otimes  w_{n+m}).
\end{align*}
Moreover, the functor $\rho_W$ preserves the open-traced structure. By the axioms of an open-traced category (in particular balanced category) and given that the functor is strict monoidal, it is enough to check this for the generator of $\mathcal{E}(A)$.  The braiding, $\tau_{1,1} = (R,(12))$ of $\mathcal{E}(A)$ is sent to the map 
\begin{equation}\label{eq:rhoV_braided}
\rho_W(R,(12))(w \otimes w') = R_{21} \cdot (w' \otimes w)= \tau_{W,W}(w \otimes w'), 
\end{equation}
which is the braiding of $\mathsf{fMod}_A^{\mathrm{str}}$ by \eqref{eq:braiding_fMod}. On the other hand, the twist $\theta_1 = (\nu, \id)$ of  $\mathcal{E}(A)$ is sent to the map 
 \begin{equation}\label{eq:rhoV_twist}
\rho_W(\nu,\id)(w) = \nu \cdot  w = \theta_W(w) 
\end{equation} 
by \eqref{eq:twist_fMod}.

Lastly, to check that the open trace is preserved, the argument is a bit more involved. Let us start by making explicit the evaluation $\widetilde{\mathrm{ev}}_{W}$ for the right duality.  According to \eqref{eq:right_evaluation}, for $x \in W$ and $\omega \in W^*$ we have
\begin{align*}
\widetilde{\mathrm{ev}}_W (x \otimes \omega) &= \mathrm{ev}_W \tau_{W,W^*}      (\theta_W \otimes \id_{W^*})  (x \otimes \omega)\\
&= \mathrm{ev}_W  \tau_{W,W^*}  (v^{-1}x \otimes \omega)\\
&= \mathrm{ev}_W \left( \sum_i \beta_i \omega \otimes \alpha_i v^{-1}x \right)\\
&= \sum_i ( \beta_i \omega) (\alpha_i v^{-1}x )\\
&= \sum_i (S(\beta_i) \alpha_i v^{-1}x)\\
&= \omega (uv^{-1}x)\\
&= \omega (\kappa x).
\end{align*}

Let us now consider a morphism $(u, \sigma) \in \eend{\mathcal{E}(A)}{n}$. For notational convenience let us suppose that $\sigma (n-1)=n$ and $\sigma (n)=n-1$. Let us also write $u=  \sum x_1 \otimes \cdots \otimes x_n$, and $(e_i)$ for a basis for $W$ and $(\omega_i)$ for its dual basis. Then taking into account \eqref{eq:trace_for_ribbon} we have

\begin{align*}
&\mathrm{tr}_{W^{\otimes n-1}, W^{\otimes n-1}}^W (\rho_W(u,\sigma))(w_1 \otimes \cdots \otimes w_{n-1})  \\
&\phantom{==} = (\id_{W^{\otimes n-1}} \otimes \widetilde{\mathrm{ev}}_{W}) (\rho_W(u,\sigma) \otimes \id_{W^*})  (\id_{W^{\otimes n-1}} \otimes \mathrm{coev}_{W})(w_1 \otimes \cdots \otimes w_{n-1})\\
&\phantom{==} = \sum_i (\id_{W^{\otimes n-1}} \otimes \widetilde{\mathrm{ev}}_{W}) (\rho_W(u,\sigma) \otimes \id_{W^*}) (w_1 \otimes \cdots \otimes w_{n-1} \otimes e_i \otimes \omega_i)\\
&\phantom{==} = \sum \sum_i (\id_{W^{\otimes n-1}} \otimes \widetilde{\mathrm{ev}}_{W}) (x_{\sigma^{-1}(1)} w_{\sigma^{-1}(1)} \otimes \cdots \otimes x_{\sigma^{-1}(n-2)} w_{\sigma^{-1}(n-2)} \otimes x_{n }e_i \otimes x_{n-1} w_{n-1} \otimes \omega_i)\\
&\phantom{==} = \sum \sum_i  x_{\sigma^{-1}(1)} w_{\sigma^{-1}(1)} \otimes \cdots \otimes x_{\sigma^{-1}(n-2)} w_{\sigma^{-1}(n-2)} \otimes \omega_i (\kappa x_{n-1} w_{n-1}) x_{n }e_i\\
&\phantom{==} = \sum   x_{\sigma^{-1}(1)} w_{\sigma^{-1}(1)} \otimes \cdots \otimes x_{\sigma^{-1}(n-2)} w_{\sigma^{-1}(n-2)} \otimes \kappa x_{n } x_{n-1} w_{n-1}\\
&\phantom{==} = \rho_W (\sum x_1 \otimes \cdots \otimes x_{n-2} \otimes x_n \kappa x_{n-1}, \hat{\sigma})(w_1 \otimes \cdots \otimes w_{n-1})\\
&\phantom{==} =\rho_W (\mathrm{tr}_{n-1,n-1}^1(u, \sigma))(w_1 \otimes \cdots \otimes w_{n-1})
\end{align*}
as desired. Lastly to have a well-defined functor $\rho_W$ we need to make sure that each of the maps $\rho_W(u, \sigma)$ is $A$-linear. We argue as follows: by definition, the category $(\mathsf{fMod}_A^{\mathrm{str}})_W$ is spanned by bradings, twists and open traces of admissible  morphisms monoidally spanned by these two, and the same holds for $\mathcal{E}(A)$. Now, since braidings, twists and traces in $\mathsf{fMod}_A^{\mathrm{str}}$ are $A$-linear, and we have seen that $\rho_W$ preserves these three, then we conclude that each $\rho_W(u, \sigma)$ must be $A$-linear as well.
\end{construction}

For every finite-free $A$-module $W$, we have then constructed a strict monoidal functor $$\rho_W: \mathcal{E}(A) \to (\mathsf{fMod}_A^{\mathrm{str}})_W$$ that preserves the open-traced structure. The following theorem (which is well-known in the non-functorial setting) justifies the adjective ``universal'' for $Z_A$:

\begin{theorem}\label{thm:ZA_universal}
Let $A$ be a ribbon Hopf algebra and let $W$ be a finite-free $A$-module. Then we have the following commutative diagram:
%of open-traced strict monoidal functors between open-traced monoidal categories:
$$
\begin{tikzcd}
\Tup \rar{RT_W} \dar[swap]{Z_A} & (\mathsf{fMod}_A^{\mathrm{str}})_W\\
\mathcal{E}(A) \arrow{ur}[swap]{\rho_W} &
\end{tikzcd}
$$
That is, for any  finite-free $A$-module $W$, the Reshetikhin-Turaev invariant $RT_W$ factors through $Z_A$: we have $$RT_W (T) = \rho_W (Z_A(T))$$ for any upwards tangle $T$.
\end{theorem}
\begin{proof}
Since $Z_A$ and $\rho_W$ preserve the braiding, twist and open traces, so does its composite $\rho_W \circ Z_A$, and moreover $(\rho_W \circ Z_A)(+)=(W)$ hence we  have  $RT_W =\rho_W \circ Z_A$ by \cref{thm:UP_opentraced}. 
\end{proof}

\subsection{Endomorphism XC-algebras from representations of ribbon Hopf algebras}

We now turn to the following situation: let $A$ be a ribbon Hopf algebra and let $W$ be a finite-free $A$-module. On the one hand, $W$ gives rise to the Reshetikhin-Turaev invariant $$RT_W: \T^+ \to \mathsf{fMod}_A^{\mathrm{str}}. $$ On the other hand, $A$ is in particular an XC-algebra, so by \cref{lemma:endom_alg_tr_XC} the endomorphism algebra $\mathrm{End}_\Bbbk(W)$ inherits a traced XC-algebra structure, so the $A$-module $W$ gives rise to the traced universal invariant
$$ Z_{\mathrm{End}_\Bbbk (W)}: \T^+ \to \mathcal{E}(\mathrm{End}_\Bbbk (W)) $$ from \eqref{eq:traced_universal_invariant}. 
It turns out that these two invariants are essentially the same.

\begin{construction}
Let $A$ be a ribbon Hopf algebra and let $W$ be a finite-free $A$-module. We define a monoidal embedding $$\iota_W :\mathcal{E}(\mathrm{End}_\Bbbk (W)) \hooklongrightarrow  \mathsf{fMod}_A^{\mathrm{str}}$$ as follows: on objects, set $\iota_W (n) := (W)^{\otimes n}$. On arrows, given a morphism $$\left( \sum f_1 \otimes \cdots \otimes f_n, \sigma \right) : n \to n$$ in $\mathcal{E}(\mathrm{End}_\Bbbk (W))$, set $$ \iota_W \left( \sum f_1 \otimes \cdots \otimes f_n, \sigma \right) := \sum \sigma_*(f_1 \otimes \cdots \otimes f_n)  : W^{\otimes n} \to W^{\otimes n},$$
where we view each $f_1 \otimes \cdots \otimes f_n$ in the right-hand side as an element of $\mathrm{End}_\Bbbk (W^{\otimes n})$ via the canonical isomorphism $$\mathrm{End}_\Bbbk (W)^{\otimes n} \toiso \mathrm{End}_\Bbbk (W^{\otimes n}),$$ and $\sigma_* : W^{\otimes n} \to W^{\otimes n}$ permutes the factors.

Let us start by checking that $\iota_W$ is indeed a functor: for composable arrows $\left( \sum f_1 \otimes \cdots \otimes f_n, \sigma \right)$ and $\left( \sum g_1 \otimes \cdots \otimes g_n, \tau \right)$ we have
\begin{align*}
\iota_W &\left(\left( \sum f_1 \otimes \cdots \otimes f_n, \sigma \right) \circ \left( \sum g_1 \otimes \cdots \otimes g_n, \tau \right)   \right)\\
 &= \iota_W \left( \sum (f_{\tau (1)} \otimes \cdots \otimes f_{\tau (n)} )    \cdot (g_1 \otimes \cdots \otimes g_n) , \sigma\tau  \right) \\
&= \sigma_* \tau_* \left( \sum (f_{\tau (1)} \otimes \cdots \otimes f_{\tau (n)} )    \cdot (g_1 \otimes \cdots \otimes g_n) \right) \\
&= \sigma_* \left( \sum f_1 \otimes \cdots \otimes f_{n} \right)    \cdot \tau_* \left( \sum  g_1 \otimes \cdots \otimes g_n \right) \\
&= \iota_W \left( \sum f_1 \otimes \cdots \otimes f_n, \sigma \right) \circ  \iota_W  \left( \sum g_1 \otimes \cdots \otimes g_n, \tau   \right).
\end{align*}
Now this functor is clearly an embedding, that is, an injective-on-objects, faithful functor. Besides, it is strict monoidal: 
\begin{align*}
\iota_W &\left(\left( \sum f_1 \otimes \cdots \otimes f_n, \sigma \right) \otimes  \left( \sum g_1 \otimes \cdots \otimes g_n, \tau \right)   \right)\\
&=\sum  \iota_W  (f_1 \otimes \cdots \otimes f_n \otimes g_1 \otimes \cdots \otimes g_n, \sigma \otimes \tau )\\
&= \sum (\sigma_* \otimes \tau_* )(f_1 \otimes \cdots \otimes f_n \otimes g_1 \otimes \cdots \otimes g_n)  \\
&= \sigma_* \left( \sum f_1 \otimes \cdots \otimes f_{n} \right)   \otimes  \tau_* \left( \sum  g_1 \otimes \cdots \otimes g_n \right) \\
&= \iota_W \left( \sum f_1 \otimes \cdots \otimes f_n, \sigma \right) \otimes  \iota_W \left( \sum g_1 \otimes \cdots \otimes g_n, \tau \right)  .
\end{align*}
Finally, let us check that $\iota_W$ is in fact traced. Consider the following square of categories and functors:
$$
\begin{tikzcd}
\mathcal{E}(A) \arrow{r}{\rho_W} \dar &  (\mathsf{fMod}_A^{\mathrm{str}})_W \dar[hook] \\
\mathcal{E}(\mathrm{End}_\Bbbk (W)) \arrow[hook]{r}{\iota_W} & \mathsf{fMod}_A^{\mathrm{str}}
\end{tikzcd}
$$
Here the left  vertical functor is the canonical one induced by the algebra  map $A \to \mathrm{End}_\Bbbk(W)$ defining the $A$-module structure of $W$, which trivially preserves the braiding, twist and open trace of admissible morphisms. Now, comparing the definitions of $\rho_W$ and $\iota_W$ immediately shows that the square  commutes. In particular, this implies that $\iota_W$  preserves the braiding and twist, and the computation for the preservation of open  trace of admissible morphisms is similar to that of \cref{constr:rho_W}. Therefore it is only left to check that $\iota_W$ preserves the trace of morphisms involving the trace of an endomorphism. Let us consider an arbitrary element $\left( \sum f_1 \otimes \cdots \otimes f_n, \sigma \right) \in \eend{\mathcal{E}(\mathrm{End}_\Bbbk (W))}{n}$ and without loss of generality let us assume $\sigma (n)=n$. As before we write $(e_i)$ for a basis for $W$ and $(\omega_i)$ for its dual basis.  We compute:

\begin{align*}
&\mathrm{tr}_{W^{\otimes n-1}, W^{\otimes n-1}}^W \left( \iota_W\left( \sum f_1 \otimes \cdots \otimes f_n, \sigma \right)\right)(w_1 \otimes \cdots \otimes w_{n-1})  \\
&\phantom{==} = (\id_{W^{\otimes n-1}} \otimes \widetilde{\mathrm{ev}}_{W}) (\iota_W\left( \sum f_1 \otimes \cdots \otimes f_n, \sigma \right) \otimes \id_{W^*})  (\id_{W^{\otimes n-1}} \otimes \mathrm{coev}_{W})(w_1 \otimes \cdots \otimes w_{n-1})\\
&\phantom{==} = \sum_i (\id_{W^{\otimes n-1}} \otimes \widetilde{\mathrm{ev}}_{W}) (\iota_W\left( \sum f_1 \otimes \cdots \otimes f_n, \sigma \right) \otimes \id_{W^*}) (w_1 \otimes \cdots \otimes w_{n-1} \otimes e_i \otimes \omega_i)\\
&\phantom{==} = \sum \sum_i (\id_{W^{\otimes n-1}} \otimes \widetilde{\mathrm{ev}}_{W}) (f_{\sigma^{-1}(1)} w_{\sigma^{-1}(1)} \otimes \cdots \otimes f_{\sigma^{-1}(n-1)} w_{\sigma^{-1}(n-1)} \otimes f_{n }e_i  \otimes \omega_i)\\
&\phantom{==} = \sum \sum_i  x_{\sigma^{-1}(1)} f_{\sigma^{-1}(1)} \otimes \cdots \otimes f_{\sigma^{-1}(n-1)} w_{\sigma^{-1}(n-1)} \otimes \omega_i (\kappa f_{n} e_i) \\
&\phantom{==} =  \sum  \mathrm{tr}(\kappa f_n)  f_{\sigma^{-1}(1)} w_{\sigma^{-1}(1)} \otimes \cdots \otimes f_{\sigma^{-1}(n-1)} w_{\sigma^{-1}(n-1)}\\
&\phantom{==} = \iota_W (\sum   \mathrm{tr}(\kappa f_n)  f_1 \otimes \cdots \otimes f_{n-1} , \tilde{\sigma})(w_1 \otimes \cdots \otimes w_{n-1})\\
&\phantom{==} =\iota_W (\mathrm{tr}_{n-1,n-1}^1(\sum f_1 \otimes \cdots \otimes f_n, \sigma))(w_1 \otimes \cdots \otimes w_{n-1})
\end{align*}
as we claimed.
\end{construction}

\begin{theorem}\label{thm:RTV_ZEndV}
Let $A$ be a ribbon Hopf algebra and let $W$ be a finite-free $A$-module. Then we have the following commutative diagram:
$$
\begin{tikzcd}
& \T^+ \arrow{dl}[swap]{ Z_{\mathrm{End}_\Bbbk(W)}} \arrow{dr}{RT_W} & \\
\mathcal{E}(\mathrm{End}_\Bbbk(W)) \arrow[hook]{rr}{\iota_W} &&  \mathsf{fMod}_A^{\mathrm{str}}
\end{tikzcd}
$$
That is, viewing $\mathcal{E}(\mathrm{End}(W))$ as a traced monoidal subcategory of $\mathsf{fMod}_A^{\mathrm{str}},$ the functors $ Z_{\mathrm{End}_\Bbbk(W)}$ and $RT_W$ coincide.
\end{theorem}
\begin{proof}
The functors $\iota_W$ and $ Z_{\mathrm{End}_\Bbbk(W)}$ are traced strict monoidal, then so is its composite $\iota_W \circ Z_{\mathrm{End}_\Bbbk(W)}$, which satisfies that $(\iota_W \circ Z_{\mathrm{End}_\Bbbk(W)})(+)=(W)$, so we conclude that $RT_W= \iota_W \circ Z_{\mathrm{End}_\Bbbk(W)}$ by \cref{thm:UP_traced}.
\end{proof}

We can summarise the relations between the various functors that have appeared in this paper:

\begin{corollary}\label{cor:prism}
For any ribbon Hopf algebra $A$ and any finite-free $A$-module $W$, we have the following commutative prism:
$$
\begin{tikzcd}
\Tup \arrow[rrrr,"RT_W"] \arrow[rrd,"Z_A"] \arrow[dd,hook] &  &               &  & (\mathsf{fMod}_A^{\mathrm{str}})_W  \arrow[dd,hook] \\
                                       &  & \mathcal{E}(A) \arrow{rru}{\rho_W} &  &                           \\
\T^+ \arrow[rrrr,pos=0.6,"RT_W"] \arrow[rrd,"Z_{\eend{\Bbbk}{W}}"]            &  &               &  & \mathsf{fMod}_A^{\mathrm{str}}            \\
                                       &  & \mathcal{E}(\eend{\Bbbk}{W}) \arrow[hook]{rru}{\iota_W} \arrow[from=uu,pos=0.3,crossing over]        &  &                          
\end{tikzcd}
$$
\end{corollary}

%\nocite{*}
%\bibliographystyle{amsalpha}
%\bibliographystyle{alpha}
\bibliographystyle{halpha-abbrv}
\bibliography{bibliografia}

\end{document}